\documentclass[11pt]{article}

\RequirePackage[OT1]{fontenc}
\RequirePackage{amsthm,amsmath}

\newtheorem{thm}{Theorem}[section]
\newtheorem{cor}[thm]{Corollary}
\newtheorem{lem}[thm]{Lemma}
\newtheorem{prop}[thm]{Proposition}
\newtheorem{defn}[thm]{Definition}

\newtheorem{aspt}[thm]{Assumption}
\newtheorem{rem}[thm]{Remark}

\newtheorem{theorem}{Theorem}[section]
\newtheorem{lemma}[theorem]{Lemma}
\newtheorem{proposition}[theorem]{Proposition}

\usepackage{amsmath,amssymb,amsfonts, amsthm, amsbsy, mathtools, epsfig}
\usepackage[usenames]{color}
\usepackage{rotating}
\usepackage{boxedminipage}
\usepackage{epstopdf}
\usepackage{enumerate}

\usepackage{verbatim}
\usepackage{graphicx}
\usepackage{caption}
\usepackage{subcaption}
\usepackage{bbm}
\usepackage{xcolor}
\usepackage{bm}
\usepackage{mathrsfs,color,dsfont}
\usepackage{algorithm}

\usepackage{thmtools}
\usepackage{thm-restate}
\usepackage{cleveref}

\usepackage{algpseudocode}
\algnewcommand{\Inputs}[1]{%
  \State \textbf{Inputs:}
  \Statex \hspace*{\algorithmicindent}\parbox[t]{.8\linewidth}{\raggedright #1}
}
\algnewcommand{\Initialize}[1]{%
  \State \textbf{Initialize:}
  \Statex \hspace*{\algorithmicindent}\parbox[t]{.8\linewidth}{\raggedright #1}
}
\algnewcommand{\Outputs}[1]{%
  \State \textbf{Outputs:}
  \Statex \hspace*{\algorithmicindent}\parbox[t]{.8\linewidth}{\raggedright #1}
}
\usepackage{todonotes}

\newcommand{\bX}{{\bf X}}
\newcommand{\bx}{{\bf x}}
\newcommand{\bW}{{\bf W}}
\newcommand{\bw}{{\bf w}}

\newcommand{\by}{{\bf y}}
\newcommand{\bZ}{{\bf Z}}
\newcommand{\bz}{{\bf z}}
\newcommand{\Ex}{{\mathcal{X}}}
\newcommand{\D}{{\mathcal{D}}}

\newcommand{\R}{{\mathbb R}}
\renewcommand{\P}{{\mathbb P}}

\newcommand{\B}{\bar {\bf \mathcal{B}}}

\newcommand{\Ly}{{\mathbf{Ly}}}

\newcommand{\C}{{\mathcal{C}}}

\newcommand{\lc}{{\mathcal{L}}}

\newcommand{\bbC}{\mathbb{C}}

\newcommand{\RNum}[1]{\uppercase\expandafter{\romannumeral #1\relax}}
\makeatother

\newcommand{\sign}{\text{sign}}
\newcommand{\var}{\text{var}}
\newcommand\tabcaption{\def\@captype{table}\caption}

\definecolor{DSgray}{cmyk}{0,1,0,0}

\newcommand{\E}{\mathbb{E}}
\newcommand{\bfx}{\mathbf{x}}
\newcommand{\bfw}{\mathbf{w}}

\newcommand{\calE}{\mathcal{E}}

\begin{document}
\title{Spectral Gap of Replica Exchange Langevin Diffusion on Mixture Distributions}
\author{Jing Dong\footnote{Columbia University, Email: jing.dong@gsb.columbia.edu}  ~and Xin T. Tong\footnote{Department of Mathematics, National University of Singapore, Email: mattxin@nus.edu.sg}}
\date{\today}
\maketitle

\begin{abstract}

Langevin diffusion (LD) is one of the main workhorses for sampling problems.
However, its convergence rate can be significantly reduced if the target distribution is a mixture of multiple densities, especially when each  component concentrates around a different mode. Replica exchange Langevin diffusion (ReLD) is a sampling method that can circumvent this issue. In particular, ReLD adds  another LD sampling a high-temperature version of the target density, and exchange the locations of two LDs according to a Metropolis-Hasting type of law. This approach can be further extended to multiple replica exchange Langevin diffusion (mReLD), where $K$ additional LDs are added to sample distributions at different temperatures and exchanges take place between neighboring-temperature processes. While ReLD and mReLD have been used extensively in statistical physics, molecular dynamics, and other applications, there is little existing analysis on its convergence rate and choices of temperatures. This paper closes these gaps assuming the target distribution is a mixture of log-concave densities. We show ReLD can obtain constant or even better convergence rates even when the density components of the mixture concentrate around isolated modes. We also show using mReLD with $K$ additional LDs can achieve the same result while the exchange frequency only needs to be $(1/K)$-th power of the one in ReLD. 
\end{abstract}

\section{Introduction}
\label{sec:intro}
Given a $d$-dimensional distribution $\pi(x)\propto \exp(-H(x))$, a standard way to generate samples from $\pi(x)$ is simulating the over-damped Langevin  diffusion (LD)
\begin{equation}\label{eq:LD}
dX(t)=\nabla \log \pi(X(t))dt+\sqrt{2}dB(t)
\end{equation}
for a long enough time horizon, where $B(t)$ is a $d$-dimensional Brownian motion. The main justification of this approach is that, under mild conditions, we can show that the invariant measure of \eqref{eq:LD} is the target distribution $\pi$.

This approach can be quite efficient when the potential or Hamiltonian function $H(x)$ is strongly-convex. However, if $H$ has multiple local minimums, $m_1,\ldots, m_I$, sitting inside deep potential wells, LD can be very inefficient. In fact, the process will spend a large amount time circulating inside one potential well  before it can reach another potential wells ( see, for example, \cite{MS14}). Such behavior significantly slows down the convergence rate of LD to stationarity (see Proposition \ref{prop:eg1} for a concrete example).  

\subsection{Replica exchange Langevin diffusion (ReLD)}
\label{sec:introReLD}
Replica exchange, also known as parallel tempering, is a method that has been used extensively in molecular dynamic (MD) to improve the convergence rate of the sampling process to stationarity \cite{earl2005}. In its simplest implementation to LD, we consider a parallel LD process with a stronger stochastic force:
\begin{equation}\label{eq:LDY}
dY(t)=\nabla \log \pi(Y(t))dt- \tau Y(t)/M^2 dt+\sqrt{2\tau}dW(t),
\end{equation}
where $W(t)$ is an independent $d$-dimensional Brownian  motion and $\tau$ is a parameter often known as the temperature.
$M$ is a large number so that the local minimums of $H$ satisfy $\max_{1\leq i \leq I} \|m_i\|\leq M$. 
 The stationary distribution of \eqref{eq:LDY} takes the form 
 \[
 \pi^y(y)\propto \exp\left(-\frac1\tau H(y)- \frac{\|y\|^2}{2M^2}\right).
 \] 
 When $\tau$ is selected to be a large number, the effective Hamiltonian for $Y_t$ is approximately $\tau^{-1}H(y)$, which shares the same local minimums with $H(x)$, but the height of the potential wells are only $\frac{1}{\tau}$ of the latter. Consequentially, it is easier for \eqref{eq:LDY} to climb out of potential wells and visit other local minimums. We also remark that adding $-Y(t)/M^2$ in the drift term of \eqref{eq:LDY} simplifies the theoretical analysis, as the stationary distribution of $\pi^y(y)$ will be concentrated in $B(0,M):=\{x:\|x\|\leq M\}$ even if $\tau\to \infty$. For most simulation problems in practice, the state space is bounded, in which case we do not need to include this additional drift term.  
 
Even though $Y(t)$ is not sampling the target density $\pi(x)$, it can be used to help $X(t)$ sample $\pi(x)$. To do so, 
let $\rho>0$ denote a swapping intensity, so that sequential swapping events take place according to an exponential clock with rate $\rho$. At a swapping event time $t$, $X(t)$ and $Y(t)$ swap their positions (values) with  a Metropolis-Hasting type of probability $s(X(t), Y(t))$, where
\begin{equation}
\label{eq:simples}
s(x,y)=1\wedge \frac{\pi(y)\pi^y(x)}{\pi(x)\pi^y(y)}.
\end{equation}
We will refer to this joint process $((X(t), Y(t)))$ as the replica exchange Langevin diffusion (ReLD). 
It can verified that $\pi\times\pi^y$ is the invariant distribution of ReLD under mild ergodicity conditions, i.e., the ReLD $(X(t),Y(t))$  will converge to $\pi\times \pi^y $ as $t\rightarrow\infty$. Thus, we can use the trajectory of $X(t)$ to generate samples from $\pi$. 

Exchanging $X(t)$ with $Y(t)$ can improve the convergence of $X(t)$ and we demonstrate the basic idea through Figure \ref{fig:1}. As mentioned before, the main reason why sampling directly from the LD can be slow for multimodal $\pi$ is that $X(t)$ can be trapped in a potential well for a  long period of time. In Figure \ref{fig:1}, suppose $X(t)$ is currently in $B(m_1,r)$, which is a ball of radius $r$ centered at a mode $m_1$. In order for $X(t)$ to visit a different mode $m_2$, it needs to visit the boundary of a potential well, e.g., the origin in Figure \ref{fig:1}, and this can take a long time. On the other hand, it is much easier for $Y(t)$ to cross the potential wells. In particular, $Y(t)$ can move ``freely" in a larger region demonstrated as $B(0,R)$ in Figure \ref{fig:1}, which includes all the local minimums. The exchange mechanism \eqref{eq:simples} swaps $X(t)$ and $Y(t)$ with decent probability if $Y(t)$ is in a different ``high-probability" area for $X(t)$, say $B(m_2,r)$. This helps $X(t)$ visit other modes, which effectively improves the convergence rate of $X(t)$.

\begin{figure}[h!]
\begin{center}
\includegraphics[scale=0.4]{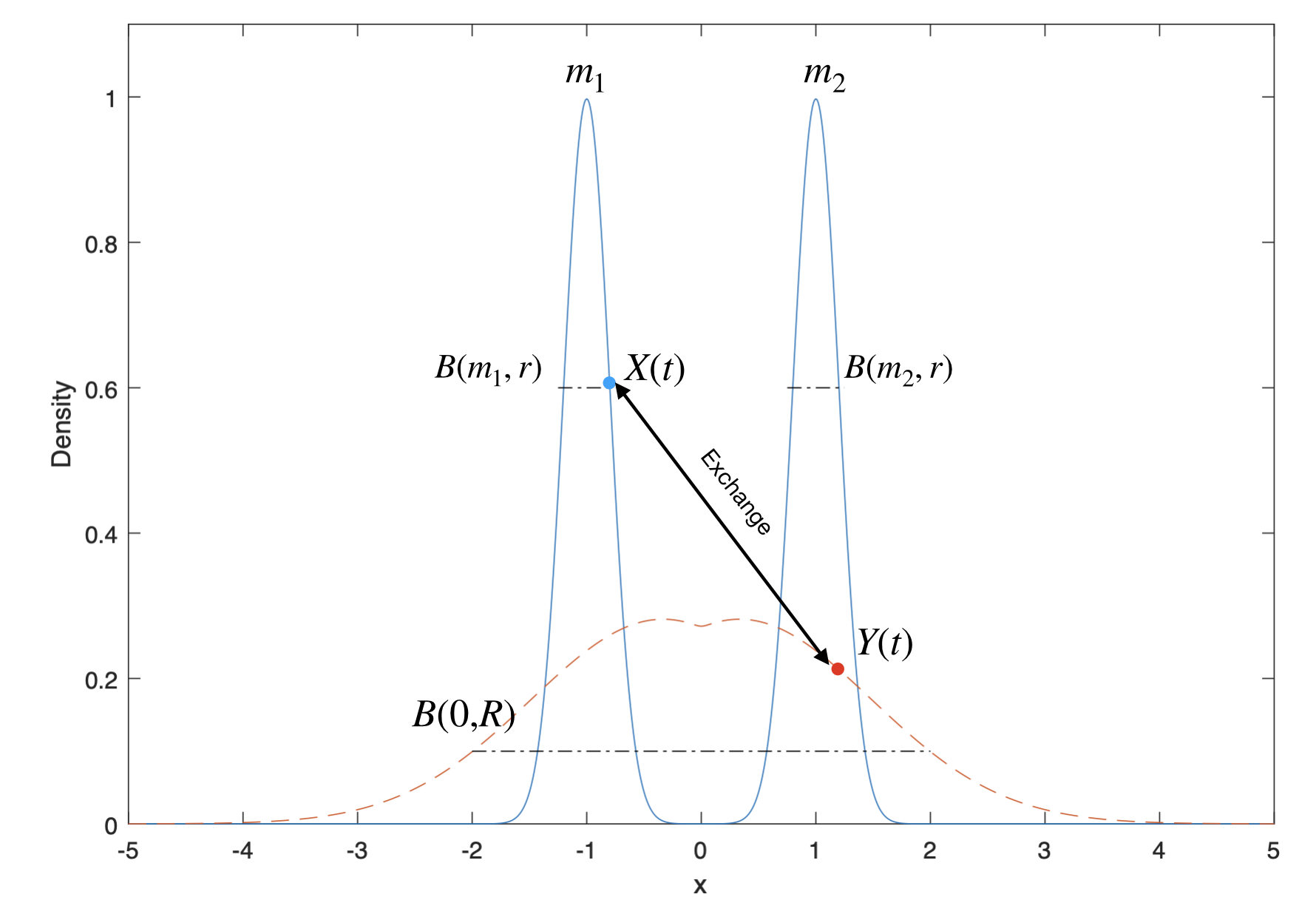}
\end{center}
\caption{ReLD: The (blue) solid line plots the density function of a bi-modal density $\pi$. The (red) dashed line plots the tempered density function $\pi^y$. The (black) dash-dot intervals are ``high-probability" areas for the two processes to sample from.}
\label{fig:1}
\end{figure}

\subsection{Multiple replica exchange Langevin Diffusion (mReLD)}
\label{sec:intromReLD}
One major issue with ReLD introduced in Section \ref{sec:introReLD} is that the exchanges may not happen often enough. 
To see this, note that in Figure \ref{fig:1}, when $X(t)$ is near the first mode $m_1$, the exchange probability \eqref{eq:simples} can be very small unless $Y(t)$ is in the ``high-probability" areas $B(m_1,r)$ or $B(m_2,r)$ as well. But since $Y(t)$ is circling inside a large area $B(0,R)$, the chance that it is in $B(m_1,r)$ or $B(m_2,r)$ can be small if $r\ll R$. 
To amend this issue, MD simulation often implements multiple replica exchange (mReLD), which involves simulating multiple parallel LDs with an increasing sequence of temperatures $1=\tau_0<\tau_1<\ldots<\tau_{K}$:
\begin{equation}\label{eq:ezmReLD}
\begin{split}
dX_0(t) &= \nabla \log\pi(X_0(t))dt + \sqrt{2\tau_0}dW_0(t),\\
dX_1(t) &= \nabla \log\pi(X_1(t))dt + \sqrt{2\tau_1}dW_1(t),\\
&\dots\\
dX_K(t) &=  \nabla \log\pi(X_K(t))dt- \tau_K X_K(t)/M^2 + \sqrt{2\tau_K}dW_K(t).
\end{split}\end{equation}
Marginally, $X_k(t)$ has  an invariant distribution 
\begin{equation}
\label{eqn:ezmeasure}
\pi_k(x)\propto \exp(-\tfrac{1}{\tau_k}H(x)) \mbox{ for  }k\leq K-1,\mbox{ and } \pi_K(x)\propto\exp(-\tfrac{1}{\tau_K}H(x)-\tfrac{\|x\|^2}{2M^2}).
\end{equation}
The exchange takes place between neighboring replicas. In particular, for each $k$, $0\leq k\leq K-1$, a swapping event takes place according to an independent exponential clock with rate $\rho$, at which $X_k(t)$ and $X_{k+1}(t)$ are swapped with probability
$s_k(X_k(t), X_{k+1}(t))$, where
\[
s_k(x_k,x_{k+1})=1\wedge \frac{\pi_k(x_{k+1})\pi_{k+1}(x_k)}{\pi_k(x_k)\pi_{k+1}(x_{k+1})}.
\]
Evidently, ReLD is a special case of mReLD with $K=1$. Note that in practice, one can also consider exchanges taking place between non-adjacent replicas. However, in our setting, exchanging between adjacent replicas yields better convergence rate both intuitively and via rigorous analysis (see Section \ref{sec:mReLD_proof} for more details). 

Adding intermediate temperatures improves the low exchange probability issue mentioned earlier and we illustrate the basic idea through Figure \ref{fig:2}, where we run three parallel LDs, i.e., $K=2$. The ``high-probability" areas for $X_0(t), X_1(t)$, and $X_2(t)$ are $B(m_1,r_0)\cup B(m_2,r_0)$, $B(m_1,r_1)\cup B(m_2,r_1)$, and $B(0,r_2)$ respectively. We note that $r_0<r_1<r_2$. The exchange between $X_0(t)$ and $X_2(t)$ may not happen often, since $X_2(t)$ has only small chance of being inside $B(m_1,r_0)\cup B(m_2,r_0)$. On the other hand, $X_1(t)$ stays mostly inside $B(m_1,r_1)\cup B(m_2,r_1)$. Conditioned on that, it has better chance being inside $B(m_1,r_0)\cup B(m_2,r_0)$ than $X_2(t)$, and hence it can exchange with $X_1(t)$ more often. From this discussion, we see that adding additional replicas improves the chance of successful exchanges. Meanwhile, non-adjacent replica exchanges are unlike to happen, so we decide to exclude them in our design of mReLD.

\begin{figure}[h!]
\begin{center}
\includegraphics[scale=0.4]{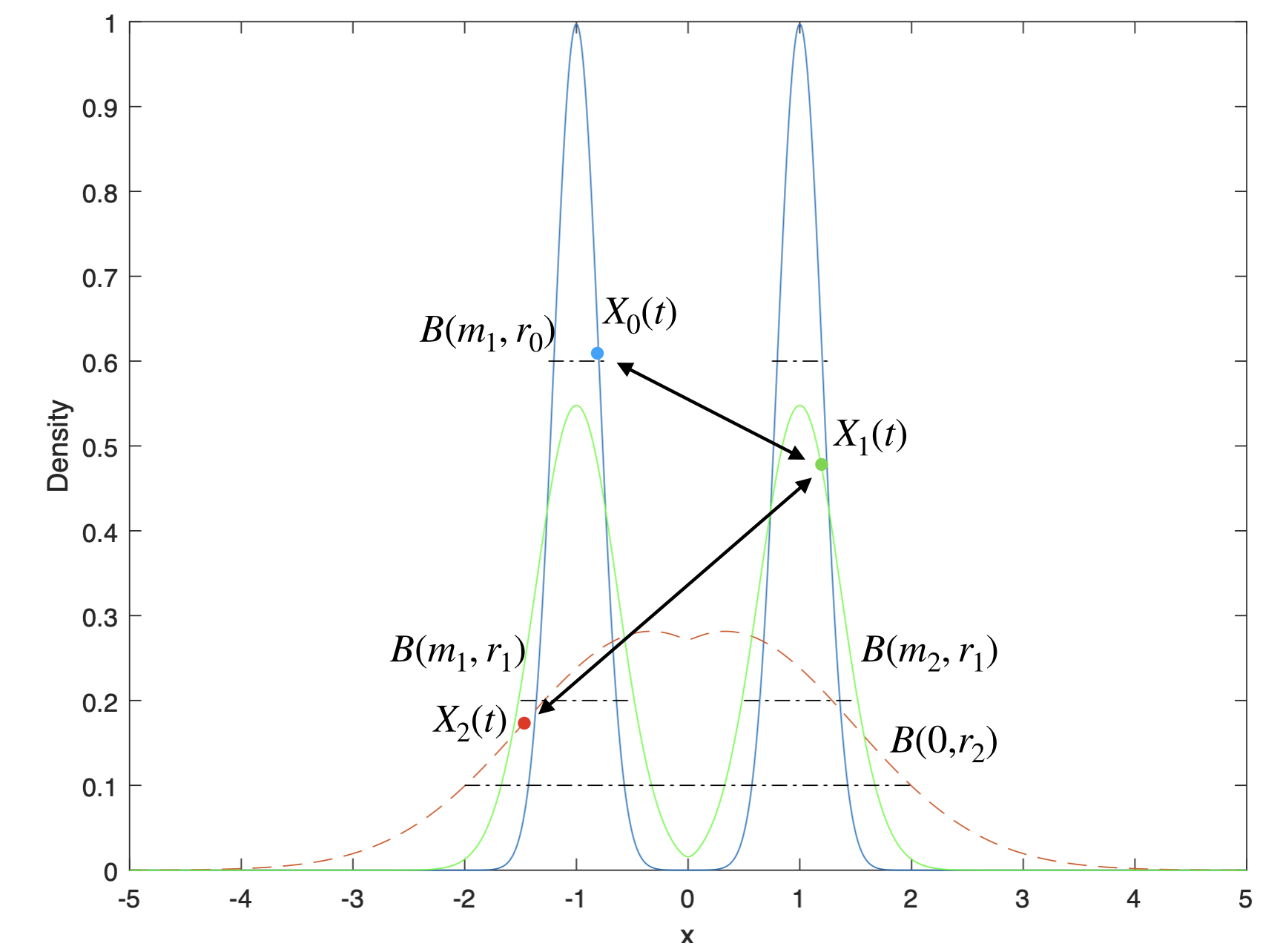}
\end{center}
\caption{mReLD: The tall (blue) solid, short (green) solid, and (red) dashed lines plot the density functions of $\pi_0,\pi_1,\pi_2$. The (black) dash-dot intervals are ``high-probability" areas for all three  processes.}
\label{fig:2}
\end{figure}

\subsection{Main questions and answers in simple settings}
\label{sec:introresults}
While ReLD and mReLD are used extensively in MD simulations, there are very few rigorous mathematical analyses explaining how replica exchange improves  simulation efficiency. 
More specifically, there is little understanding of what kind of distributions can ReLD or mReLD sample efficiently.
This theoretical gap also leads to questions on how to choose the temperatures and swapping rates when implementing ReLD, and how to choose the sequence of temperatures in mReLD. 
Existing guidance on this matter comes mostly from physical intuition and crude estimations. For example, for mReLD, \cite{abraham2008ensuring, Kofke02} recommend choosing a geometric sequence of $\tau_i$ so that the exchange acceptance rate is around $20\%$. Yet, there is no rigorous analysis supporting that.

The goal of this paper is to close these theoretical gaps by investigating the convergence rate of  ReLD and mReLD on mixture type of densities. We choose mixture densities because they are simple ways to construct multimodal distributions and are used extensively in statistics literature \cite{friedman2001elements,theodoridis2015machine}. 
To provide a glimpse of our main result, we discuss our results and its implications through a simple example as Proposition \ref{prop:eg1}.

We first introduce a quantification for the convergence rate of different sampling schemes. Mathematically, the convergence rate of a continuous time Markov process $Z_t$ is characterized by a quantity called \emph{spectral gap}. To formally define the spectral gap, we first define the generator of $Z_t$ as
\[
\lc(f)(z):=\lim_{t\to 0} \frac1t\E [f(Z_t)-f(z)|Z_0=z],
\]
for $f\in \D(\lc)$ where $\D(\lc)$ is a subset of $\bbC_c^2(\R^d)$ such that the above limit exists and $\bbC_c^2(\R^d)$ is the space
of twice continuously differentiable function with compact support.
We also define the associated carr\'{e} du champ as 
\[
\Gamma(f)=\frac12(\lc(f^2)-2f\lc(f)),
\]
and the Dirichlet form as
\[
\calE(f)=\int \Gamma(f) \pi^z(dx)
\]
if the invariant measure of $Z_t$ is $\pi^z$. The inverse of the spectral gap of $Z_t$ can then be defined as 
\begin{equation}
\label{eqn:kappa}\kappa = \inf\left\{\frac{\calE(f)}{\var_{\pi^z}(f)}; f\in \bbC_c^2(\R^d), \var_{\pi}(f) \neq 0\right\}.
\end{equation}
Note the space $\bbC_c^2(\R^d)$ is a core of the domain of $\calE$ in our applications, and the definition of $\kappa$ can be extended to the full domain (see \cite{bakry2013analysis} section 1.13). We restricted our discussions to $\bbC_c^2(\R^d)$ for simplicity.

The reason why $\kappa$ controls the speed of convergence of $Z_t$ towards $\pi^z$ can be found in Theorem 4.2.5 of \cite{bakry2013analysis}. In particular, for any test function $f\in L^2(\pi^z)$, there is a constant $C_0$ such that  
\[
\int (\E[f(Z_t)|Z_0=z]-\E_{\pi^z} f(Z))^2\pi^z(z)dz\leq C_0 e^{-2t/\kappa}.
\]
In other words, the ``simulation" expectation $\E[f(Z_t)|Z_0=z]$ converges to target expectation exponentially fast with $\pi^z$-a.s. initial conditions, and the convergence rate is $1/\kappa$.


Using the inverse spectral gap $\kappa$, we can show that LD converges very quickly for a singular Gaussian distribution, but very slowly for a mixture of two singular Gaussians.  Let $\phi$ denote the density function of $d$-dimensional standard Gaussian random vector, i.e.,
$\phi(x)=(2\pi)^{-d/2} \exp(-\|x\|^2/2)$.

\begin{prop}
 \label{prop:eg1}
 The inverse of the spectral gap $\kappa$ for LD satisfies the following bounds:
 \begin{enumerate}[1)]
 \item For $\pi(x)\propto\phi(x/\epsilon)$, $\kappa \leq\epsilon^2.$
 \item 
For $\pi(x)\propto \frac{1}{2}\phi(x/\epsilon) + \frac{1}{2}\phi((x-m)/\epsilon)$ and $\epsilon\leq \frac{\|m\|}{16\sqrt{d}}$,\[\kappa \geq \frac{\epsilon^4}{80\|m\|^2}\exp\left(\frac{\|m\|^2}{64\epsilon^2}\right).\]
 \end{enumerate}
\end{prop}
The proof of Proposition \ref{prop:eg1} can be found in Appendix A. 
The first result in Proposition \ref{prop:eg1} is quite well-known. For the second scenario, an upper bound for $\kappa$ can be found in \cite{MS14,schlichting2019poincare}, but we are not aware of any lower bound in the literature. 
It is worth pointing out that the requirement of $\epsilon\leq \frac{\|m\|}{16\sqrt{d}}$ or similar requirement in the second scenario is necessary for the LD to have a slow convergence rate. In contrast, if $\epsilon$ is larger than $\frac{\|m\|}{\sqrt{d}}$, the centers of the two Gaussians will be within one standard deviation from each other. Consequently, the two Gaussian distributions are overlapped and it would be easy to ``transition" from one to the other by LD alone.

Proposition \ref{prop:eg1} also points out that one of the most challenging types of densities 
for LD to sample are mixtures of ``well-separated" components, even if each of them is Gaussian. 
 In particular, for small values of $\epsilon$, the Gaussian distribution has a very sharp concentration around its mode (i.e., a deep potential well). We will often refer such  highly concentrated distributions as singular distributions in subsequent developments. When sampling a single Gaussian distribution using LD, the spectral gap is lower bounded by $\epsilon^{-2}$. In this case, the smaller the value of $\epsilon$, the faster the convergence rate is. However, when sampling a mixture of two such Gaussians with well-separated modes using LD, the convergence rate can be very slow. Specifically, the spectral gap is upper bounded by $\tfrac{\epsilon^2}{\|m\|^2}\exp(-\tfrac{\|m\|^2}{64\epsilon^2})$, which is extremely small for small values of $\epsilon$.
Mixture of highly singular densities is quite common in many practical situations. Section \ref{sec:expMorse} demonstrates how such singularity arises when the sample size increases in Bayesian statistics.
We will mostly focus on this type of densities in our subsequent developments.

We next show that by properly choosing the temperature and the swapping intensity in ReLD, we can substantially improve the convergence rate for the Gaussian mixture example (including the second scenario in Proposition \ref{prop:eg1}).  To highlight the challenge in the sampling efficiency, we focus on the dependence of the inverse of the spectral gap on the parameter $\epsilon$ (i.e., depth of the potential well), while keeping all other model parameters fixed. Moreover, using mReLD for low dimensional distributions, i.e., $d\leq 2$, the swapping rate $\rho$ can be reduced from $\epsilon^{-d}$ to $\epsilon^{-d/K}$. 
%

\begin{theorem}
\label{thm:simple}
Suppose the target density is a mixture of isotropic Gaussian distributions
\begin{equation}
\label{eq:simpleGsmixture}
\pi(x)\propto \sum_{i=1}^I p_i \phi\left(\frac{x-m_i}{\epsilon}\right), \mbox{ where } \phi(x)=\frac{1}{(2\pi)^{\frac{d}{2}}}\exp(-\tfrac12 \|x\|^2),
\end{equation}
where $\|m_i\|$ are bounded by $O(1)$ constants. 
For ReLD with $\tau, \rho\propto \epsilon^{-d}$ with target density $\pi$, the inverse of the special gap, $\kappa=O(1)$, i.e., is independent of $\epsilon$. When $d\leq 2$, for mReLD with 
\[
\tau_k=\epsilon^{-\frac{2k}{K}},\quad k=1,\ldots,K,\quad \rho=\epsilon^{-d/K}
\]
the inverse of the special gap, $\kappa=O(1)$.
\end{theorem}

In this section, we choose the Gaussian mixture due to its simplicity for demonstration. In Section \ref{sec:main}, we study the convergence rate of ReLD and mReLD for mixtures of more general densities, e.g., log-concave densities. In particular, Theorems \ref{th:main1} and \ref{th:main2ez} provide the general statements, while Corollaries \ref{cor:ReLD} and \ref{cor:mReLD} demonstrate how they apply to mixture densities. Theorem \ref{thm:simple} can be seen as a simple case of these two corollaries. We also make connections with Morse-function type of conditions in Section \ref{sec:morse}.


\subsection{Literature review}

Most standard Markov chain Monte Carlo (MCMC) methods suffer from slow convergence rate when the target distribution has isolated modes, i.e., the target distribution is multimodal. 
Replica exchange Monte Carlo, which is also known as parallel tempering,
has been proposed to speed up the convergence and has seen promising performance in various application contexts, especially molecular dynamics (see, for example, \cite{swendsen1986replica, sugita1999replica, earl2005, abraham2015gromacs}).
A key question in implementations of the method is how to set/tune the hyper-parameters such as the temperature and the swapping rate. Most previous investigations rely on extensive simulation experiments and heuristic arguments \cite{sindhikara2010exchange,Kofke02,abraham2008ensuring}. \cite{Dupuis:2012} uses the large deviation theory to define a rate of convergence for the empirical measure of ReLD. In particular, \cite{Dupuis:2012} characterizes the large deviation rate function for the empirical measure and show that the rate increases with the swapping intensity $\rho$.
Recently, a series of work provide rigorous analysis on how to tune the related hyper-parameters to achieve an optimal acceptance probability of the swapping events, especially in the high dimension limit \cite{roberts2014minimising,tawn2019accelerating,tawn2020weight}. 
These analyses assume the existence of some fast samplers of $\pi_k$ and focus mostly on the equilibrium state behavior. In contrast, our spectral gap analysis focuses on concrete diffusion processes and characterizes non-equilibrium behavior by nature. 

It is also worth mentioning that the idea of ReLD can also be applied on stochastic optimization where a LD exchanges with a gradient flow to keep the latter from trapped inside local minimums. This has been investigated in \cite{dong2020replica} by the same authors of this paper, but the algorithms, results, and analysis there are very different from this paper.

A similar but slightly different sampling idea to ReLD is simulated tempering, which considers dynamically changing the temperature of the LD \cite{marinari1992simulated}. 
More generally, several tempering-based MCMC methods have been studied in the literature, including annealing MCMC \cite{geyer1995annealing}, tempered transition method \cite{neal1996sampling}, etc. Like ReLD, there are very few theoretical results about the algorithms.  \cite{madras2002markov} and \cite{woodard2009conditions} develop lower bounds for the spectral gap of general simulated tempering chains. These bounds are too loose to provide concrete guidance on how to choose the temperatures.
Recently, \cite{ge2018} establishes a tighter bound for simulated tempering LD. Their analysis specifics how to set the temperatures in the setting where the target distribution is a mixture of log-concave density with different modes but similar shape.
One main challenge in the implementation of simulated tempering algorithms is that one needs to estimate the normalizing constants of the steady-state distribution at different temperatures. On the other hand, replica exchange avoids the need to deal with these normalizing constants.

Sequential Monte Carlo (SMC) is another sampling idea that is similar to replica exchange MC \cite{smith2013sequential}. In particular, SMC sets up a sequence of distributions, $\{\pi_k: k\geq 0\}$, which are often also of form $\pi_k\propto \pi^{\beta_k}$, and transition kernels $T_k$ that maps $\pi_k$ to $\pi_{k-1}$. Then by applying the sequence of transition kernels, one transform samples from an easy to sample distribution $\pi_K$ to the hard to sample target distribution $\pi_0=\pi$. SMC has been studied extensively in last two decades, both in theory \cite{del2006sequential,del2012adaptive,beskos2014stability,beskos2016convergence} and on its extension to filtering and Multi-level Monte Carlo \cite{liu1998sequential,doucet2000sequential,vo2005sequential,beskos2017multilevel}. 
However, most existing analysis of SMC focuses on its stability in the high dimensional setting, where the key question is whether the effective sample size will degenerate with the dimension. These results often assume the transition kernels $T_k$ are mixing very quickly or having a non-degenerate spectral gap. Our results are quite different. First we consider sampling a mixture of highly singular distributions. Such analysis does not exists in the current literature for SMC to the best of our knowledge. Moreover, we do not assume the existence of spectral gaps but prove their existence. Finally, the analysis of ReLD requires techniques quite different from the ones used to study SMC. On the other hand, it may be interesting to compare ReLD and SMC on their performance for general multimodal distributions. But this is out of the scope of this paper.  

It is well known that the existence of a spectral gap leads to ergodicity or the convergence of stochastic systems to stationarity. This is why it has been studied extensively for various stochastic models \cite{hairer2008spectral} and Monte Carlo methods \cite{hairer2014spectral,andrieu2015convergence}. The connection between spectral gap and Poincar\'{e} inequality (PI) has been well documented in \cite{bakry2013analysis}. The study of PI has a long lasting interest in probability, due to its connection to partial differential equations, Stein's method, and central limit theorem \cite{chen1988central, beckner1989generalized}. Recently, the connection between PI and Lyapunov type of conditions has been revealed \cite{Bakry, MS14}. As will be made clear in later sections, our analysis also applies similar techniques. However, the existing literature focuses mostly on simple LD instead of ReLD.

\subsection{Organization and notation}
The rest of this article is organized as follow. In section \ref{sec:main}, we review some existing tools for the analysis of spectral gaps. We also provide the general setups for  mixture densities, ReLD, and mReLD. The main results are Theorems \ref{th:main1} and \ref{th:main2ez}, which provide estimates on the inverse spectral gap $\kappa$ for ReLD and mReLD respectively. In Section \ref{sec:app}, we demonstrate how to apply our results to mixtures of log-concave densities and the connection of mixture models to the Morse function assumptions in \cite{MS14}. In order to keep the discussion compact and focused, we allocate most of the proofs to the later part of the article. In particular, the proof of Theorem \ref{th:main1} is provided in Section \ref{sec:ReLD_proof}, while the proof of a more explicit version of Theorem \ref{th:main2ez} is provided in Section \ref{sec:mReLD_proof}. Various technical claims are verified in the Appendix.

Given two vectors $u,v\in\R^d$, we use $\|v\|$ to denote the $l_2$ norm of $v$, and $\langle v, u\rangle:=u^Tv$. Given a matrix $A$, we use $\|A\|$ to denote its $l_2$-operator norm. 
For any $f\in\bbC^2(\R^d)$, we use $\nabla f\in \R^d$ to denote its gradient, $\nabla^2 f\in \R^{d\times d}$ to denote its Hessian, and $\Delta f:=\text{tr}(\nabla^2 f)$. We also use $B(x_0,R)$ to denote the ball of radius $R$ with center being $x_0.$

When a distribution $\pi$ is given, we use $\E_\pi f $ and $\var_\pi f$ to denote the mean and variance of $f$ under $\pi$. For two distributions $\pi$ and $\nu$ on $\R^d$, we write their product on $\R^{2d}$ as $\pi\nu$.
Since we are considering mostly diffusion type of stochastic processes, it is reasonable to assume the associate distributions are absolutely continuous with respect to the Lebesgue measure. Therefore, when we refer to a density or distribution $\pi$, we assume it has a probability density function $\pi(x)$. This also allows us to use $\frac{\pi(x)}{\nu(x)}$ when referring to the Radon-Nikodym derivative between $\pi$ and $\nu$. 

 Our goal is to develop a proper upper bound for the inverse of the spectral gap $\kappa$, which can be translated to a lower bound for the spectral gap. As the underline distribution/process may involve several parameters, e.g., $\epsilon,d,M$ in the Gaussian mixture example and $\tau, \rho$ for mReLD, the exact characterization of the upper bound can get quite heavy. Therefore, we adopt the big $O$ notation.  
In particular, we say a quantity $A_\epsilon=O(f(\epsilon))$, if there is a constant $C$ independent of $\epsilon$, such that $A_\epsilon\leq Cf(\epsilon)$. $A_{\epsilon}=O(1)$ means $A_{\epsilon} \leq C$.
We also define $A_\epsilon=\Theta(f(\epsilon))$, if there are constants $c<C$ independent of $\epsilon$, such that $cf(\epsilon) \leq A_\epsilon\leq Cf(\epsilon)$. 

\section{Problem setup and main results} \label{sec:main}
We introduce the general setups of the replica exchange algorithm and study its performance when the target distribution is of mixture type in this section. 
Our development relies on applications of Poincar\'{e} inequality (PI). Therefore, we start by introducing some basic properties of PI, which will be utilized in our subsequent developments. We then give general assumptions for the type of density mixtures that our framework can handle. The main results are provided as Theorem \ref{th:main1} and \ref{th:main2ez}. Their proofs are allocated in Section \ref{sec:ReLD_proof} and \ref{sec:mReLD_proof}.
To keep our discussion focused, we allocate most technical verification of the claims to the Appendix \ref{app:sec2} unless specified otherwise.

\subsection{Poincar\'{e} inequality and Lyapunov function}
Recall that the basic LD is given by 
\[ dX(t)=\nabla \log \pi(X(t))dt + \sqrt{2}dB(t).\]
We denote $\lc_{\pi}$ as it generator, which has the following form:
\[
\lc_{\pi}(f)=\langle \nabla f, \nabla \log\pi\rangle + \Delta f.
\]
Then, the associated carr\'{e} du champ takes the form $\Gamma(f)=\|\nabla f\|^2$.  
The inverse of the spectral gap $\kappa$ in \eqref{eqn:kappa} can also be viewed as the coefficient in the Poincare\'{e} inequality (PI). 
\begin{defn}
A  density $\pi$ follows  $\kappa$-PI if the following holds
\[
\var_\pi (f)\leq  \kappa\int \|\nabla f(x)\|^2 \pi(x)dx,\quad \forall f\in \bbC_c^2(\R^d).
\]
\end{defn}
There are a lot of existing results regarding PI
\cite{chen1988central, beckner1989generalized, bakry2013analysis, Bakry, MS14}. We next review some of them that are particularly relevant to our analysis. 
The first result is that if a density $\pi$ follow a  $\kappa$-PI, then a mild perturbation of $\pi$ will also follow a PI:

\begin{prop}[Holley--Stroock perturbation principle]
\label{prop:equivalent}
Suppose for some operator $\Gamma$ and density $\pi$ the following holds:
\[
\var_\pi(f)\leq \kappa \int \Gamma(f)(x)\pi(x)dx. 
\]
Moreover, suppose there exists a constant $C\in(0,\infty)$ such that
\[
C^{-1}\leq \frac{\pi(x)}{\mu(x)}\leq C,\quad \forall x. 
\]
Then,  
\[
\var_\mu(f)\leq C^2\kappa \int \Gamma(f)(x)\mu(x)dx. 
\]
\end{prop}
\begin{proof}
Let $\bar{f}_\mu$ and $\bar{f}_\pi$ be the mean of $f$ under $\mu$ and $\pi$. Then
\begin{align*}
\var_\mu(f(X))&=\int (f(x)-\bar{f}_\mu)^2 \mu(x)dx\\
&\leq \int (f(x)-\bar{f}_\pi)^2 \mu(x)dx\\
&\leq C\int (f(x)-\bar{f}_\pi)^2 \pi(x) dx\leq C\kappa\int \Gamma(f)(x)\pi(x)dx\leq  C^2\kappa \int \Gamma(f)(x)\mu(x)dx.
\end{align*}
\end{proof}
Note that $\Gamma$ here can be the carr\'{e} du champ of the LD, but it can also be the carr\'{e} du champ of replica exchange processes described later.

Another useful result connects the Lyapunov function to the PI constant. The connection was first established in \cite{Bakry}. Here, we present a slight different version of it.

\begin{defn}
\label{defn:lyap}
A $\bbC^2$ function $V(x):\R^d\to [1,\infty)$ is a $(\lambda,h,B,C)$-Lyapunov function for a density $\nu(x)$ if the following holds
\[
\lc_{\nu}V(x) \leq -\lambda V(x) + h1_B(x),\quad \frac{\sup_{x\in B} \nu(x) }{\inf_{x\in B}\nu(x)}\leq C,
\]
 where $\lambda, h, C \in (0,\infty)$ are positive constants and $B\subset\R^d$ is a bounded domain.
\end{defn}
\begin{prop} \label{prop:bakry_lyap}
Suppose $\nu$ has a $(\lambda,h,B(x_0,R),C)$-Lyapunov function.
Then,
\[
\var_\nu(f)\leq 
\frac{1+hR^2C^2}{\lambda}\E_{\nu}[\|\nabla f(X)\|^2].
\]
\end{prop}
Proposition \ref{prop:bakry_lyap} provides a convenient way to compute (upper bound) the PI constant for a given density $\nu$. 
It plays a central role in our subsequent developments. Therefore we allocate its proof to Section \ref{sec:proof24}. Based on Proposition \ref{prop:bakry_lyap}, we define the following notion of density:
\begin{defn}
\label{defn:density}
We say $\nu$ is an $\Ly(R,q,a)$-density with the center $x_0$, if it has a $(\lambda,h,B(x_0,R),C)$-Lyapunov function, with 
\[
\frac{1+hR^2C^2}{\lambda}\leq q ~\mbox{ and } \sup_{x\in B(x_0,R)} \frac{u_{B(x_0,R)}(x)}{\nu(x)}\leq a,
\]
where $u_{B(x_0,R)}$ denotes the uniform distribution on $B(x_0,R)$.
\end{defn}

In our main theoretical development, we will consider replacing $\nu$ with $u_{B(x_0,R)}$, since the latter is easier to handle. The constant $a$ in the $\Ly(q,R,a)$-density roughly measures how well the approximation is. 
Also note in subsequent development, we often refer $x_0$ as a center instead of a mode. This is  because $\nu$ satisfying Definition \ref{defn:density} in general may have multiple modes, unless $\nu$ is log concave. 

We next provide a concrete example to demonstrate how to apply the connection between the Lyapunov function and PI.
In particular, we will show that log-concave densities have Lyapunov functions as in Definition \ref{defn:lyap}, and hence are $\Ly(R,q,a)$-densities. 
\begin{defn}
A density $\nu$ is a $(c,L)$-log-concave density if $H=-\log \nu$ is $\bbC^2$ and
\[
\langle \nabla H(x)-\nabla H(y), x-y\rangle\geq  c \|x-y\|^2,\quad \|\nabla^2 H(x)\|\leq L,\quad \forall x, y. 
\]
\end{defn}
For example, $\mathcal{N}(m,\Sigma)$ is $(\lambda_{\max}^{-1},\lambda_{\min}^{-1})$-log-concave, where $\lambda_{\min}$ and $\lambda_{\max}$ are the minimum and maximum eigenvalues of $\Sigma$.

\begin{lem}
\label{lem:GsnLyapunov}
If $\nu$ is $(c,L)$-log-concave and $m$ is its mode, then $V(x)=\tfrac{c}{d}\|x-m\|^2+1$ is a $(\lambda,h,B,C)$-Lyapunov function of $\nu$ with
\[
\lambda=c,\quad h=3c,\quad B=B\left(m,\sqrt{\frac{3d}{c}}\right),\quad C=\exp\left(\frac{3dL}{2c}\right). 
\]
This implies that $\nu$ is a $\Ly(r,q,a)$-density with 
\[
q=c^{-1}+\frac{9d}{c}\exp\left(\frac{3dL}{c}\right),\quad
r = \sqrt{\frac{3d}{c}}, \quad
a= \frac{1}{V_d}\exp\left(\frac{3Ld}{2c}\right) \left(\frac{4\pi}{3d}\right)^{d/2},
\]
where $V_d$ denote the volume of a $d$-dimensional ball with unit radius.
\end{lem}


We next provide a bound for $a$ in $\Ly(R,q,a)$-density (Definition \ref{defn:density}), based on a specific form of the $(\lambda, h, B(x_0,R), C)$-Lyapunov function.

\begin{lem}
\label{lem:gettinga}
If $\nu$ has a $(\lambda, h, B(x_0,R), C)$-Lyapunov function of form $V(x)=\gamma \|x-x_0\|^2+1$, then the following hold 
\[
\sup_{x\in B(x_0,R)} \frac{u_{B(x_0,R)}(x)}{\nu(x)}\leq a,\quad
a=\frac{C}{V_d}\exp\left(\frac14 \lambda R^2\right)\left(\frac{4\pi}{\lambda R^2}\right)^{d/2}. 
\]
\end{lem}

\subsection{Density Mixture}
As discussed in Section \ref{sec:intro}, we are interested in understanding how replica exchange improves the convergence of LD with a multimodal target density. One easy way to describe a multimodal density is through a mixture model: 
\begin{equation}
\label{eqn:mixture}
\pi(x)=\sum_{i=1}^I p_i \nu_i(x),
\end{equation}
where the weights $p_i$'s are nonnegative and sum up to one, and each $\nu_i$ has a single mode $m_i$.

Next, we discuss at an intuitive level what kind of mixture model would allow a replica exchange process $(X_t,Y_t)$ to sample efficiently. First of all, each $\nu_i$ should be ``easy" for a LD of form \eqref{eq:LD} to sample directly, since the exchange mechanism can only help $X_t$ visiting different modes but not sampling individual $\nu_i$ faster. This requirement can be formulated through the existence of an appropriate Lyapunov function for $\nu_i$ based on Proposition \ref{prop:bakry_lyap}. This is why we impose
\begin{aspt}
\label{aspt:mixture}
There are positive constants $r_i,q,a$ such that for $i=1,\dots, I$,
$\nu_i$ is an $\Ly(r_i,q,a)$-density with the center $m_i$. 
\end{aspt}

Second, $Y_t$ should be able to visit different $m_i$'s ``easily", otherwise it cannot help $X_t$ reaching certain modes. This requirement can be formulated as requiring that $m_i$ are not too far from each other. Since our problem is shift invariant, this is equivalent to assuming $\|m_i\|$ is bounded. By Proposition \ref{prop:piyez}, we will see that this is equivalent to imposing assumptions on the invariant measure for $Y(t)$, namely Assumption \ref{aspt:piy} below.  In what follows, we denote $M$ as a large constant such that 
\[\max_{1\leq i\leq I}\|m_i\|\leq M.\]

It is worth mentioning that \cite{MS14} imposes different assumptions on the Hamiltonian function $H(x)=-\log(\pi(x))$. In particular, it assumes $H(x)$ is a Morse function and there is an admissible partition so that a proper Lyapunov function exists within each partition. Admittedly, this might be a more general assumption, since not all densities can be written as a mixture \eqref{eqn:mixture}. However, this set of assumptions requires more technical definitions and verification. Moreover, it can be shown that under mild conditions, the setting in \cite{MS14} can be converted to a mixture. We will provide more details about the connection in Section \ref{sec:morse}.

\subsection{Replica-exchange Langevin diffusion} \label{sec:ReLD}
We next formulate the general replica exchange Langevin diffusion (ReLD). First, pick a density $\pi^y$ and consider the following two LDs driven by independent $d$-dimensional Browian motions $W^x(t)$ and $W^y(t)$:
\begin{equation}
\begin{split}
\label{eqn:2xLD}
dX(t)=\nabla \log \pi(X(t))dt+\sqrt{2}dW^x(t),\\
dY(t)=\tau\nabla \log \pi^y(Y(t))dt+\sqrt{2\tau}dW^y(t).\\
\end{split}
\end{equation}
During the simulation of $(X_t,Y_t)$, exchange event times are triggered by an exponential clock of rate $\rho$, at which we swap the positions of $X_t$ and $Y_t$ with probability $s(X_t,Y_t)$, where $s$ is given by
\[
s(x,y)=1\wedge \frac{\pi(y)\pi^y(x)}{\pi(x)\pi^y(y)}.
\]
It is easy to see that the ReLD discussed in Section \ref{sec:intro} is a special case of \eqref{eqn:2xLD} with 
\[
\pi^y(y)=\exp(-\tfrac1\tau H(y)- \tfrac{\|y\|^2}{2M^2}). 
\]

We consider a general $\pi^y$ here for two reasons. First, as we will discuss in Section \ref{sec:ReLD}, the temperature $\tau$ is often ``required" to be a large number, direct simulation of $Y(t)$ with the Euler-Maruyama scheme would require very small stepsizes. If $\pi^y$ is a simple density, for example, a Gaussian density, we can have direct access to the transition kernel of $Y(t)$ and avoid using the Euler-Maruyama scheme. Second, it is easier to impose requirements on $\pi^y$ for the replica exchange process to achieve good convergence rate.
 Consider a mixture-type target distribution as in \eqref{eqn:mixture}. Let $m_i$ denote the center of $\nu_i$, $i=1,\dots, I$.
We then impose the following assumptions on $\pi^y$:

\begin{aspt}
\label{aspt:piy}
There are constants $(R,Q,A)$ so that 
for each center $m_i$, $\pi^y$ is an $\Ly(R,Q,A)$-density with center $m_i$, $i=1,\dots, I$. 
\end{aspt}

We next provide some specific forms of $\pi^y$ that satisfies Assumption \ref{aspt:piy}. It indicates that Assumption \ref{aspt:piy} is similar to requiring all modes, $m_i$'s, being bounded. 


\begin{prop}
\label{prop:piyez}
 Suppose there exits a constant $M<\infty$ such that $\max_{1\leq i\leq I}\|m_i\|\leq M$.
\begin{enumerate}[1)]
\item If $\pi^y(x)\propto \phi(x/M)$, then Assumption \ref{aspt:piy} holds with 
\[
R^2=O(M^2 d), \quad Q=O(M^2 d\exp(12d)), \quad A = O(\exp(6d)).
\]
\item If $\pi^y(x)\propto \phi(x/M)\pi(x)^\beta$, where $\nu_i(x)$'s are $(c,L)$-log concave densities and 
$\beta \leq (dM^2 c + dM^2L^2/c)^{-1}$, then Assumption \ref{aspt:piy} holds with
\[
R^2=O(M^2 d), \quad Q=O(M^2 d\exp(20d)), \quad A = O(\exp(12d)).
\]

\end{enumerate} 
\end{prop}


Let $\zeta$ denote exponential clock clicking time for the exchange event. 
The generator of the ReLD, denoted by $\lc_R$, is then given by 
\begin{align*}
\lc_R f(x,y)&=\lim_{t\to 0}\frac{1}{t}\E[f(X_t,Y_t)-f(x,y)|X_0=x,Y_0=y]\\
&=\lim_{t\to 0}\frac{1}{t}\E[f(X_t,Y_t)-f(x,y)|X_0=x,Y_0=y,\zeta> t]e^{-\rho t}\\
&\quad +\lim_{t\to 0}\frac{1}{t}\E[f(X_t,Y_t)-f(x,y)|X_0=x,Y_0=y,\zeta\leq t](1-e^{-\rho t})\\
&=\lc_x f(x,y)+\tau\lc_y f(x,y)+\rho s(x,y) (f(y,x)-f(x,y)),
\end{align*}
where 
\[
\lc_x f(x,y):=\langle \nabla_x f(x,y), \nabla_x \log\pi(x)\rangle  
+\Delta_x f(x,y),
\]
\[
\lc_y f(x,y):=\langle \nabla_y f(x,y), \nabla_y \log\pi^y(y)\rangle  
+\Delta_y f(x,y).
\]
It is easy to verify that $\pi(x)\pi^y(y)$ is an invariant measure for ReLD. In particular 
\begin{align*}
\E_{\pi\pi^y}\lc_R f&=\E_{\pi\pi^y}(\lc_x f+\tau\lc_yf) +\rho \E_{\pi\pi^y} s(X,Y)(f(Y,X)-f(X,Y))  \\ 
&=\rho \E_{\pi\pi^y} s(X,Y)(f(Y,X)-f(X,Y)) \\
&=\rho \int f(y,x) (\pi(x)\pi^y(y))\wedge (\pi(y)\pi^y(x))dxdy\\
&\quad -\rho \int f(x,y) (\pi(x)\pi^y(y))\wedge (\pi(y)\pi^y(x))dxdy=0.
\end{align*}
The associated carre du champ for ReLD is given by 
\begin{align*}
\Gamma_R f(x,y)&=\frac12 (\lc_R(f^2)-2f\lc_R(f))\\
&=\frac12\left(\lc_x f^2+\tau\lc_y f^2 +\rho s(x,y)(f(y,x)^2-f(x,y)^2)\right)\\
&\quad -f(x,y)(\lc_x f+\tau \lc_y f+\rho s(x,y)(f(y,x)-f(x,y)))\\
&=\|\nabla_x f(x,y)\|^2+\tau\|\nabla_y f(x,y)\|^2+
\frac12 \rho s(x,y) (f(y,x)-f(x,y))^2. 
\end{align*}
Note that if we simply simulate $X_t$ and $Y_t$ according to \eqref{eqn:2xLD} without exchange, the carre du champ will be 
$\|\nabla_x f(x,y)\|^2+\tau\|\nabla_y f(x,y)\|^2$. The exchange mechanism contributes to an additional positive term $\frac12 \rho s(x,y) (f(y,x)-f(x,y))^2$ in $\Gamma_R$. While this definitely helps lowering the inverse spectral gap $\kappa$ in \eqref{eqn:kappa}, the extent of improvement is far from obvious. 

We next quantify the effect of the exchange mechanism on the spectral gap.
In addition to Assumptions \ref{aspt:mixture} and \ref{aspt:piy}, we also impose the following assumption
\begin{aspt}
\label{aspt:bound}
There are $r>0$ and $R>0$, such that
\[R_i \leq R \mbox{ and } \frac{R_i}{r_i}\leq \frac{R}{r} \mbox{ for $i=1\dots, I$}.\]
\end{aspt}

\begin{theorem} \label{th:main1}
For ReLD defined in \eqref{eqn:2xLD}, under Assumptions \ref{aspt:mixture}, \ref{aspt:piy}, and \ref{aspt:bound},
\[\var_{\pi\pi^{y}}(f(X,Y)) \leq \kappa \E_{\pi\pi^{y}}[\Gamma_R(f(X,Y))],\]
where 
\[\begin{split}
\kappa
&=\max\left\{3(56A+1)q, 
~\frac{3}{\tau}\left(57Q + 14aA\left(\frac{R^{d+1}}{r^{d-1}}\right)\left(\log\left(\frac{R}{r}\right)\right)^{1_{d=1}}\right),\frac{7aA}{\rho}\left(\frac{R}{r}\right)^d\right\}.
\end{split}\]
In particular, if $R,A,Q, a$ are $O(1)$ constants, then 
\[
\kappa=O\left(q+\left(\frac{1}{\tau}+\frac{1}{\rho}\right) \frac{1}{r^d}\right).
\]
 When $q<1$, if we set $\tau, \rho \geq U q^{-\alpha} r^{-d}$ for any $\alpha\leq 1$ and $U>0$, then $\kappa=O(U^{-1}q^{\alpha})$.
\end{theorem}

For mixture models of singular densities (i.e., isolated and highly concentrated modes), $r$ and $q$ are often very small. For example, as we will explain in more details in Section \ref{sec:app_ReLD}, $r^2,q=\Theta(\epsilon^2)$ for the Gaussian mixture in Proposition \ref{prop:eg1}. If we choose $\tau, \rho \geq r^{-d}$, then $\kappa=O(1)$, i.e., it does not depend on $r$ or $q$. If we choose $\tau, \rho \geq q^{-1}r^{-d}$, then $\kappa=O(q)$. In this case, the spectral gap is of the same order as smallest spectral gap of the component densities in the mixture.


\subsection{Multiple replica-exchange Langevin diffusion (mReLD)}
In this section, we introduce the general mReLD. Considering $K+1$ LD processes
\begin{equation}\label{eq:mReLD}
\begin{split}
dX_0(t) &= \tau_0\nabla \log\pi_0(X_0(t))dt + \sqrt{2\tau_0}dW_0(t),\\
dX_1(t) &= \tau_1\nabla \log\pi_1(X_1(t))dt + \sqrt{2\tau_1}dW_1(t),\\
&\dots\\
dX_K(t) &= \tau_K \nabla \log\pi_K(X_K(t))dt + \sqrt{2\tau_K}dW_K(t),
\end{split}\end{equation}
with $1=\tau_0 \leq \tau_1 \leq \dots \leq \tau_K$ and $\pi_0=\pi$.
Exchange between two adjacent levels takes place according to independent exponential clocks with rate $\rho$.  At an exchange event time $t$ for the pair $(k,k+1)$, $k=0, \dots, K-1$, 
$X_k(t)$ and $X_{k+1}(t)$ swap their positions (values) with probability $s_k(X_k(t), X_{k+1}(t))$, where
\[
s_k(x_k,x_{k+1})=1\wedge \frac{\pi_k(x_{k+1})\pi_{k+1}(x_k)}{\pi_k(x_k)\pi_{k+1}(x_k)}.
\]
The standard mReLD discussed in Section \ref{sec:intromReLD} is a special case of \eqref{eq:mReLD} with $\pi_k(x)\propto(\pi(x))^{1/\tau_k}$ for $1\leq k\leq K-1$, and $\pi_K(x)\propto\phi(x/M)(\pi(x))^{1/\tau_K}$. We consider a general setup of $\pi_k$, since it can simplify our discussion. Moreover, general $\pi_k$'s are also used in practice, such as the Umbrella method \cite{matthews2018umbrella}, where $\pi_k$  is a portion of $\pi$ but not a tempered version.

Let $\bx_{k:l}=(x_k, \dots, x_l)$ and $\pi_{k:l}=\pi_k\pi_{k+1}\cdots \pi_{l}$ be the product measure. Note that each  $x_k\in \R^d$ for $k=0,1,2\dots, K$. 
The generator of mReLD then takes the form:
\[\begin{split}
\lc^K_{R} (f(\bx_{0:k}))&=\sum_{k=0}^{K}\left(\tau_k \langle \nabla_{x_k} f(\bx_{0:k}), \nabla \log \pi_k(\bx_{0:k})) \rangle + \tau_k \Delta_{x_k} f(\bx_{0:k})\right)\\
&+\sum_{k=0}^{K}\rho s_k(x_k,x_{k+1})(f(\bx_{0:K}) - f(\bx_{0:(k-1)},x_{k+1}, x_k,  \bx_{(k+2):K}))
\end{split}\]
The corresponding carr\'e du champ operator and Dirichlet from are given by 
\[\begin{split}
\Gamma^K_R (f(\bx_{0:K})):=&\sum_{k=0}^{K}\tau_k\|\nabla_{x_k} f(\bx_{0:K})\|^2\\
&+ \sum_{k=0}^{K}\rho s_k(x_k,x_{k+1})(f(\bx_{0:K}) -  f(\bx_{0:(k-1)},x_{k+1}, x_k,  \bx_{(k+2):K}))^2
\end{split}\]
and  
$\calE^K_R (f)=\int \Gamma^K_R(f) \pi_{0:K}(d\bfx_{0:K})$
respectively.

We make the following assumptions about $\pi_k$'s.

\begin{aspt}
\label{aspt:mixture2}
There are positive constant $q_k, r_{k,i}, a_k$ for $k=0, \dots, K$, $i=1,\dots, I$, such that
\begin{enumerate}
\item For $k=0,\dots, K-1$,
$\pi_k(x) = \sum_{i=1}^{I}p_{i}\nu_{k,i}(x)$, where  
$\nu_{k,i}$ is an $\Ly( r_{k,i},q_k, a_{k})$-density with center $m_i$.
\item For each $m_i$, $\pi_K$ is an $\Ly(r_{K,i},q_K,a_K)$-density with center $m_i$. 
\end{enumerate}
\end{aspt}

\begin{aspt}
\label{aspt:bound2}
There is an increasing sequence $0<r_0<r_1<\dots<r_K$, such that for $k=0, \dots, K-1$,
\[\frac{r_{k+1,i}}{r_{k,i}} \leq \frac{r_{k+1}}{r_k},\]
and $r_{K,i} \leq r_K$ for $i=1,\dots, I$.
\end{aspt}



\begin{theorem} \label{th:main2ez}
For mReLD defined in \eqref{eq:mReLD}, suppose Assumptions \ref{aspt:mixture2} and \ref{aspt:bound2} hold, and $K,q_k,a_k,r_k$, $k=1,2,\dots, K$, are all bounded by a finite constant, i.e., $O(1)$. Then, 
\begin{equation}
\label{eqn:main2ez}
\var_{\pi_{0:K}}(f(\bX_{0:K})) \leq \kappa \E_{\pi_{0:K}}[\Gamma^K_R(f(\bX_{0:K}))],
\end{equation}
where
\[
\kappa=O\left(\max\left\{\frac{q_0}{\tau_0},\left(\frac{1}{\tau_k}+\frac{1}{\rho}\right)\left(\frac{r_{k}}{r_{k-1}}\right)^d,1\leq k\leq K\right\}\right).
\]
In particular,  when $q_0<1$, for any $\alpha\leq 1, U>0$, if we choose 
\[
\tau_k\geq U\left(\frac{1}{q_0}\right)^\alpha\left(\frac{r_{k}}{r_{k-1}}\right)^d,\quad \rho\geq  U\max_k \left(\frac{1}{q_0}\right)^\alpha\left(\frac{r_{k}}{r_{k-1}}\right)^d,
\]
then $\kappa=O(U^{-1}q_0^\alpha)$. 
\end{theorem}

For mixture models with small values of $r_0$ and $q_0$, if we can construct $\pi_k$'s such that $\tfrac{r_k}{r_{k-1}}=\Theta\Big(\left(\tfrac{r_K}{r_0}\right)^{1/K}\Big)$ for $k=1,\dots, K$, then we can set $\tau_k, \rho \geq \left(\tfrac{r_K}{r_0}\right)^{d/K}$ to achieve $\kappa=O(1)$. 
If we further enlarge $\tau_k, \rho\geq q_0^{-1}\left(\tfrac{r_K}{r_0}\right)^{d/K}$, $\kappa=O(q_0)$. In this case, the spectral gap matches the smallest spectral gap of the component densities in the mixture.

We remark that the exact estimate of $\kappa$ is quite complicated. Interested readers can find the explicit expression in Theorem \ref{th:main2}. Here we assume $K,q_k,a_k,r_k$, $1\leq k\leq K$, are all fixed $O(1)$ constants to simplify the characterization of the estimate.

\section{Applying replica-exchange to different densities}
\label{sec:app}
In this section, we demonstrate the performance of the replica exchange algorithm for some specific examples. 

\subsection{ReLD for mixture of log-concave densities} \label{sec:app_ReLD}
A general Gaussian mixture model can be written as $\pi(x)=\sum_{i=1}^I p_i \nu_i(x)$, where 
\[
\nu_i(x)= \frac{1}{\sqrt{\det (2\pi \Sigma_i)}}\exp(-\tfrac12 (x-m_i)^T\Sigma_i^{-1}(x-m_i)).
\]
Suppose 
\[
\C^{-1} l_i^2\preceq \Sigma_i\preceq  l_i^2,\quad l_m\leq l_i\leq l_M\leq 1. 
\]
where $\C$ is often known as the condition number. More generally, we can consider the scenario where 
$\nu_i$ is a $(l_i^{-2},\mathcal{C}l_i^{-2})$-log-concave density. We then have the following result
\begin{cor}
\label{cor:ReLD}Suppose $\pi=\sum_{i=1}^I p_i \nu_i$ where $\nu_i$'s are $(l_i^{-2},l_i^{-2}\C )$-log concave densities with modes $m_i$ and $\|m_i\|\leq M$. Let
\[
l_m=\min_i l_i,\quad l_M=\max_i l_i,\quad \tau \geq dM^2 l_M^{-2} + dM^2l_M^2l_m^{-4}\C^2 
\] 
Then $\var_{\pi\pi^y} f\leq \kappa\E_{\pi\pi^y} \Gamma_R(f)$
holds with
\[
\kappa=O\left(\exp(\C D d) \max\left\{dl_M^2,\frac{1}{\tau}dMl_m\left(\frac{M}{l_m}\right)^d, 
\frac{1}{\rho}\left(\frac{M}{l_m}\right)^d\right\}\right),
\]
where $D$ is a fixed constant. 
\end{cor}

\begin{proof}
Since $\nu_i$ is $(l_M^{-2},l_m^{-2}\C)$-log concave, by Lemma \ref{lem:GsnLyapunov}, Assumption \ref{aspt:mixture} holds with 
\[
q=l_M^2+9d l_M^2\exp(3 d\C),\quad r_i^2=3dl_i^2,\quad a=\frac{1}{V_d}\left(\frac{4\pi}{3d}\right)^{d/2}\exp\left(\frac{3d\C}{2}\right).
\]
Choosing $\pi^y(x)\propto \phi(x/M)(\pi(x))^\beta$ with 
\[
\beta=\frac{1}{\tau} \leq (dM^2 l_M^{-2} + dM^2l_M^2l_m^{-4}\C^2)^{-1},
\] 
Proposition \ref{prop:piyez} gives us
\[
R^2=O(M^2d),\quad Q=O(M^2d\exp(20d)),\quad A=O(\exp(12d)).
\]
Plug these estimates with $r^2=3dl_m^2$ into Theorem \ref{th:main1}, we have
\[
\begin{split}
\kappa
&=\max\left\{3(56A+1)q, 
~\frac{3}{\tau}\left(57Q + 14aA\left(\frac{R^{d+1}}{r^{d-1}}\right)\left(\log\left(\frac{R}{r}\right)\right)^{1_{d=1}}\right),\frac{7aA}{\rho}\left(\frac{R}{r}\right)^d\right\}\\
&=O\left(\exp(\C D d) \max\left\{dl_M^2,\frac{1}{\tau}dMl_m\left(\frac{M}{l_m}\right)^d, 
\frac{1}{\rho}\left(\frac{M}{l_m}\right)^d\right\}\right). 
\end{split}\]
\end{proof}
In below, we provide some interpretations and implications of Corollary \ref{cor:ReLD}.  As $\kappa$ is the inverse of the spectral gap, we refer to $1/\kappa$ as the convergence rate, i.e., the size of the spectral gap.

First, consider the setting in Theorem \ref{thm:simple} where $l_m^2=l_M^2=\epsilon^2$ and $\C=1$. By choosing 
$\tau,\rho=\Theta(\epsilon^{-d-2})$, $\beta=\tau^{-1}\leq \epsilon^2$ and
$\kappa=O(\epsilon^2)$. This matches the LD convergence rate when $\pi \propto \phi(x/\epsilon )$, i.e., a single Gaussian. 
We can also set 
$\tau,\rho=\Theta(\epsilon^{-d})$,
which leads to $\kappa=O(1)$ as stated in Theorem \ref{thm:simple}. 

In addition, our result allows the Gaussian components to be of different scales. For example, $l_1^2=l_m^2=\epsilon^2$ and $l_2^2=l_M^2=\epsilon$. $M,\C,d$ are fixed. In this case, if $\tau\geq\max\{\epsilon^{-d},\epsilon^{-3}\}$ and $\rho=\Theta(\epsilon^{-d-1})$, $\beta=\tau^{-1}\leq \epsilon^{-3}$ and   
$\kappa=O(l_M^{2})=O(\epsilon)$. This matches the LD convergence rate for $\pi=\nu_2$.  

In general, for fixed values of $d$ and $\C$, $\tau$ and $\rho$ need to scale as $M^d/l_m^d$ for the convergence rate to be of constant order. To see the intuition behind this, note that with a high temperature, $Y(t)$ can be roughly seen as a random search in the set $\{\|x\|\leq M\}$ with speed $\tau$. At any time $t$, the chance that it finds the radius $l_m$ neighborhood of a mode $m_i$
is $(l_m/M)^d$. Thus, to have a constant convergence rate, it is necessary for $Y_t$ to run at a speed $\tau=\Theta((M/l_m)^d)$. Meanwhile, $\rho$ is rate of checking whether replica exchange takes place, and it needs to be of the same scale as $\tau$.

In implementations, $\tau$ can be seen the simulation speed of $Y_t$ in \eqref{eqn:2xLD}, and $\rho$ is the frequency of exchange events. When applying discretization schemes like Euler-Maruyama for ReLD, the step size often need to scale as $\min\left\{\frac{1}{\tau},\frac{1}{\rho}\right\}$. If $M=O(1)$ and $l_m=O(\epsilon)$, the computational cost of ReLD is roughly $O(\epsilon^{-d})$. While this can be quite high, it is much better than the computational cost using LD, which is roughly $O(\exp(D\epsilon^{-2}))$ as shown in Proposition \ref{prop:eg1}. 
When taking computational cost into account, it is of practical interest to further reduce $\tau$ and $\rho$. This is why we discuss mReLD below. 

Lastly, $\kappa$ have an exponential dependence on $d$ and $\C$. This indicates that ReLD may not be a good sampler for high dimensional or highly anisotropic distributions. This is partly because LD and related algorithms such as MALA are not behaving well in these distributions \cite{roberts2001optimal}. This is quite well known in the literature. There are many existing techniques to fix these issues, such as dimension reduction, acceleration with Hessian information and Gibbs-type updates \cite{cui2014likelihood,cui2016dimension,girolami2011riemann,beskos2009optimal,
tong2020mala}. It might be of interest to investigate how these techniques can be integrated with ReLD. 

\subsection{Gaussian mixtures with mReLD}
In this section, we demonstrate how the mReLD result applies to the mixture models discussed in Section \ref{sec:app_ReLD}. Following the practical choice in MD simulation, we assume the invariant measure for $X_k(t)$ takes the form
\[
\pi_k(x)\propto (\pi(x))^{\beta_k},\quad k=0,1,\ldots,K-1
\]
for some inverse temperature $\beta_k\in [0,1]$. Note that this choice makes the drift term of $X_k(t)$ being a multiple of $\nabla\log (\pi(X_k(t)))$, which is accessible in general settings.  

We next characterize the spectral gap of mReLD when the target distribution is a mixture of log concave densities.
Our results partly depend on whether we need to synchronize $\tau_k$ with $\beta_k$. In particular, if the speed of simulation for $X_k(t)$, which is described by $\tau_k$, does not need to match the temperature $\frac{1}{\beta_k}$, then $(\beta_k)_{1\leq k\leq K}$ can be chosen as a geometric sequence for most efficient simulation. If $\tau_k$ needs to be $\frac{1}{\beta_k}$, $(\beta_k)_{1\leq k\leq K}$ can be a geometric sequence for $d=1,2$. But for $d\geq 3$, our analysis requires $\beta_k$ to be log geometric.


\begin{cor}
\label{cor:mReLD}
Suppose $\pi_0=\pi=\sum_{i=1}^I p_i \nu_i$, where $\nu_i$ is $(l^{-2}_i,l^{-2}_i\C )$-log concave densities with modes $m_i$ and $\|m_i\|\leq M$ for $i=1,\dots, I$. Suppose also that
\[
l_m=\min_i l_i,\quad l_M=\max_i l_i.
\] 
Consider running mReLD with 
\[
\pi_k(x)\propto (\pi(x))^{\beta_k},\quad k=1,\ldots, K-1,\quad \pi_K(x) \propto(\pi(x))^{\beta_K}\phi(x/M). 
\]
With $K,d,\C,M$ all being O(1) constants,
\begin{enumerate}
\item  if  $\beta_k=l_m^{\frac{2k}{K}},\tau_0=1$, and $\tau_k,\rho\geq l_m^{-\alpha-\frac{d}{K}}$ for $0\leq \alpha\leq 1$, $k=1,\dots, K$, then \eqref{eqn:main2ez} holds with $\kappa=O(l_M^{2\alpha})$;
\item if $d\leq 2$, $\tau_k=\beta_k^{-1}=l_m^{-\frac{2k}{K}}$ for $k=0, 1, \dots, K$, and $\rho\geq l_m^{-d/K}$, then \eqref{eqn:main2ez} holds with $\kappa=O(1)$;
\item if $d\geq 3$, $\tau_0=\beta_0=1$, $\tau_k=\beta_k^{-1}=l_m^{-2(\frac{d-2}{d})^{K-k}}$ for $k=1, \dots, K$, and $\rho\geq l_m^{-2}$, then \eqref{eqn:main2ez} holds with $\kappa=O\big(l_m^{-d(\tfrac{d-2}{d})^{K-1}}\big)$.
\end{enumerate}
 \end{cor}

We again consider the setting in Theorem \ref{thm:simple} where $l_m^2=l_M^2=\epsilon^2$ and $\C=1$. By choosing 
$\tau_k,\rho=\Theta(\epsilon^{-\frac{d}{K}-2})$, $\beta_k=\epsilon^{\frac{2k}{K}}$, for $k=1,\dots, K$, we have 
$\kappa=O(\epsilon^{-2})$. This matches the LD convergence rate when $\pi \propto \phi(x/\epsilon )$, i.e., a single Gaussian. 
We can also set 
$\tau_k,\rho=\Theta(\epsilon^{-d/K})$ and $\beta_k=\epsilon^{\frac{2k}{K}}$,
which leads to $\kappa=O(1)$. Comparing the discussion following Corollary \ref{cor:ReLD}, we note that the parameters $\tau_k,\rho$ are reduced from $\epsilon^{-d}$ to $\epsilon^{-d/K}$. This in practice can be computationally more desirable.

To prove Corollary \ref{cor:mReLD}, we first introduce an auxiliary lemma.
\begin{lemma}
\label{lem:constantsbeta}
For any given $\beta\in (0,1]$, if $\nu$ is a $(l^{-2},l^{-2}\C)$-log concave density with mode $m$, 
then  
\[
\mu(x)\propto(\nu(x))^\beta 
\] 
is a $\Ly(R_\beta,q_\beta,a_\beta)$ with $\lambda_\beta=\beta l^{-2}$, and certain constant $D$ so that
\[
q_\beta=O\left(\frac{d\exp(Dd\C)}{\lambda_\beta}\right),\quad R^2_\beta= \frac{4d}{\lambda_{\beta}}=O\left(\frac{d}{\lambda_\beta}\right),\quad A_\beta=O(\exp(Dd\C)). 
\]
\end{lemma}
\begin{proof}
We consider using $V(x)=\gamma \|x-m\|^2+1$, with 
\[\gamma=\frac{\lambda_{\beta}}{2d}\]
Denote $H(x)=-\log \nu(x)$. Then 
\begin{align*}
\lc_{\mu} V(x)&= -2\gamma\beta \langle x-m, \nabla H(x)\rangle+2d\gamma\\
&\leq -2\gamma \beta l^{-2} \|x-m\|^2+2d\gamma\\
&=-2\lambda_\beta \gamma \|x-m\|^2+2d\gamma\\
&\leq -\lambda_\beta V(x)+\left(-\lambda_{\beta}\gamma \|x-m\|^2 + \lambda_{\beta}+2d\gamma\right)\\
&\leq -\lambda_\beta V(x)+b_\beta 1_{\|x-m\|^2\leq R^2_\beta}.
\end{align*}
where
\[
b_\beta=\lambda_\beta+2d\gamma=2\lambda_\beta
\]
by our choice of $\gamma$, and
\[
R_\beta^2=\frac{b_\beta}{\gamma\lambda_\beta}= \frac{4d}{\lambda_\beta}.
\]
Note that
\[
\max (\log \mu(x)-\log \mu(y))=\beta\max (\log \nu(x)-\log \nu(y))
\]
Because
\[
\beta\max_{x,y\in B(m, R_\beta)} (\log \nu(x)-\log \nu(y))\leq \frac{1}{2}\beta l^{-2}\C R_\beta^2\leq 2d\C
\]
\[C_\beta\leq \exp(2d\C).\]
Lastly, by Lemma \ref{lem:gettinga},
\[
A_\beta=\frac{C_\beta}{V_d}\exp\left(\frac12 \lambda_\beta R_\beta^2\right)\left(\frac{4\pi}{\lambda_\beta R_\beta^2}\right)^{d/2}
=O(\exp(Dd\C )), 
\]
and by Definition \ref{defn:density},
\[
q_\beta=\frac{1+R^2_\beta C^2_\beta b_\beta}{\lambda_\beta}=O\left(\frac{d\exp(Dd\C)}{\lambda_\beta}\right).
\]
\end{proof}

\begin{proof}[Proof of Corollary \ref{cor:mReLD}]
Consider the following density:
\[
\pi'_k(x)\propto \sum_{i=1}^I p_i (\nu_i(x))^{\beta_k},k=1,\ldots,K-1,\quad \pi'_0=\pi_0,\pi_K'=\pi_K. 
\]
By Lemma \ref{lem:constantsbeta}, $\pi_k'$ satisfies Assumptions \ref{aspt:mixture2} and \ref{aspt:bound2} with
\[
r_{k,i}^2=\frac{4d l_i^2}{\beta_k},\quad r^2_k=\frac{4dl_m^2}{\beta_k},\quad q_k=O(1),\quad a_k=O(1),\quad k=0,\ldots,K-1.
\]
Moreover, by Lemma \ref{lem:GsnLyapunov}, $q_0=l_M^{2}(1+9d\exp(3d\C))$.
From
Proposition \ref{prop:piyez}, for
\[
\beta_K\leq \frac{1}{dM^2(l_M^{-2}+l_m^{-4}l_M^2\C^2)},
\]
we have
\[
r_{K,i}^2=O(M^2d)=O(1),
~ q_K=O(M^2d\exp(20d))=O(1),~ a_K=O(\exp(10d))=O(1).
\]
Then, by Theorem \ref{th:main2ez}  
\[
\var_{\pi'_{0:K}}(f(\bX_{0:K})) \leq \kappa' \E_{\pi'_{0:K}}[\Gamma^K_R(f(\bX_{0:K}))],
\]
with $\kappa'=O(l_M^{2\alpha})$ for some $\alpha \leq 1$, if the parameters $\tau_k,\rho$ satisfy 
\begin{equation}
\label{eqn:tauk}
\begin{split}
&\tau_k\geq U l^{-2\alpha}_M\left(\frac{\beta_{k-1}}{\beta_{k}}\right)^{d/2},\quad k=1,\ldots, K-1,\quad \tau_K\geq U l_M^{-2\alpha}\left(\frac{\beta_{K-1}}{l_m^2}\right)^{d/2}\\
&\rho\geq U l^{-2\alpha}_M max\left\{ \left(\frac{\beta_{K-1}}{l_m^2}\right)^{d/2}, \left(\frac{\beta_{k-1}}{\beta_{k}}\right)^{d/2}\right\}. 
\end{split}
\end{equation}
for some $U>0$

Note that since $x^{\beta_k}$ is concave for $0\leq \beta_k\leq 1$,
\[
\left(\sum_{i=1}^{I}p_i\nu_i(x)\right)^{\beta_k} 
\geq \sum_{i=1}^{I} p_i\nu_i(x)^{\beta_k}
\]
On the other hand, for $p_0=\min_{i\leq I} p_i$,
\[
\left(\sum_{i=1}^{I}p_i\nu_i(x)\right)^{\beta_k} 
\leq \max_i\nu_i(x)^{\beta_k}
\leq \frac{1}{p_0}\sum_{i=1}^{I} p_i\nu_i(x)^{\beta_k}.
\]
Therefore, $p_0 \pi_k(x)\leq \pi'_k(x)\leq \pi_k(x) \leq \frac{1}{p_0} \pi_k(x)$ for $k=1,\ldots,K-1$. By Proposition \ref{prop:equivalent}, 
\[
\var_{\pi_{0:K}}(f(\bX_{0:K})) \leq \kappa \E_{\pi_{0:K}}[\Gamma^K_R(f(\bX_{0:K}))],
\]
with $\kappa=p_0^{-2K}\kappa'$. 

We next verify that \eqref{eqn:tauk} holds. \\
In scenario 1, for $k=1,\dots, K$,
as $\beta_k=l_m^{2k/K}$, for $\alpha\geq 0$,
\[
l_M^{-2\alpha} \left(\frac{\beta_{k-1}}{\beta_k}\right)^{d/2} \leq l_m^{-2\alpha-d/K} \leq \tau_k
\mbox{ and } 
l^{-2\alpha}_M \max_k \left(\frac{\beta_{k-1}}{\beta_{k}}\right)^{d/2}
\leq l_m^{-2\alpha-d/K} \leq \rho.
\]
Thus, \eqref{eqn:tauk} holds.\\
In scenario 2, \eqref{eqn:tauk} holds with $\alpha=0$. In particular, $\beta_k=l_m^{\frac{2k}{K}}<1$ and $d\leq 2$,
\[
\left(\frac{\beta_{k-1}}{\beta_{k}}\right)^{d/2}=l_m^{-d/K} \leq l_m^{-2k/K}=\tau_k,\quad 
\max_k\left(\frac{\beta_{k-1}}{\beta_{k}}\right)^{d/2}=l_m^{-d/K} \leq \rho.
\]
Lastly, for scenario 3, note that with our choice of $\beta_k$ and $\tau_k$, $k=1,\ldots,K$
\[
\left(\frac{\beta_{k-1}}{\beta_{k}}\right)^{d/2}=l_m^{\tfrac{d}{2}(2(\frac{d-2}{d})^{K-k+1}-2(\frac{d-2}{d})^{K-k})}=l_m^{-2(\frac{d-2}{d})^{K-k}}=\tau_k. 
\]
Meanwhile, because
\[
\tau_0=1=l_m^{-d(\frac{d-2}{d})^{K-1}}\left(\frac{\beta_{0}}{\beta_{1}}\right)^{d/2}, 
\]
\eqref{eqn:tauk} holds with $\alpha=-d(\frac{d-2}{d})^{K-1}$. 
\end{proof}

\begin{rem}
Our big $O$ estimates hide a factor of $p_0^{-2K}$ because we use the perturbation argument in the proof of Corollary \ref{cor:mReLD}. In other words, this bound is quite loose when $K$ is large or some component only has a small weight $p_i$. In comparison, Corollary \ref{cor:ReLD} does not have this issue. It is of  interests to see how these technical difficulties can be alleviated and we leave it as a future research direction.
\end{rem}

\subsection{Morse Hamiltonian functions} \label{sec:morse}
In \cite{MS14}, a general density model based on Morse function is considered. In particular, \cite{MS14} consider densities of the form 
\[
\pi(x) \propto \exp(-H(x)/\epsilon)
\]
where $H$ is a nonnegative Morse function.
Due to Proposition \ref{prop:equivalent}, we say $\pi_\epsilon\propto \exp(-H_\epsilon(x)/\epsilon)$ (or $H_\epsilon$) is an $\epsilon$ perturbation of $\pi(x)$ (or $H(x)$) if 
\[
|H(x)-H_\epsilon(x)|\leq D\epsilon,\quad \forall x\in \R^d \mbox{ for some constant $D\in(0,\infty)$}.
\]
\cite{MS14} further assumes that $H$ has a finite set of local minimums $\{m_1, \dots, m_I\}$,  
a partition $\{\Omega_i\}_{1\leq i\leq I}$ of $\R^d$, and a $\epsilon$-perturbation of $H$, $H_{\epsilon}$ so that \begin{equation}
\label{eqn:MSmorse}
\frac{1}{2\epsilon} \Delta H_\epsilon (x)-\frac1{4\epsilon^2}\|\nabla H_\epsilon(x)\|^2 \leq -\frac{\lambda_0}{\epsilon},\quad \forall x\notin \cup B(m_i, a\sqrt{\epsilon}),
\end{equation}
where $B(m_i,a\sqrt{\epsilon})\subset \Omega_i$.  Moreover, $\Omega_i$ is the attraction basin of $m_i$ for gradient flows driven by $\nabla H_\epsilon$, i.e., 
\[
\Omega_i:=\{x\in \R^d: \lim_{t\to\infty} x_t=m_i,\dot{x}_t=-\nabla H_\epsilon(x_t),x_0=x\}.
\]
Under these assumptions, it has been established in \cite{MS14} that $\exp(\frac1{2\epsilon}H_\epsilon)$ is a Lyapunov function for $\pi_{\epsilon}\propto\exp(-H_{\epsilon}/\epsilon)$ on each $\Omega_i$. 
In particular, we have the following lemma.

\begin{lem}
Suppose $V(x)=\exp(\frac12 H(x))$ is $\bbC^2(\R^d)$ with 
\[
H(x)=-\infty \mbox{ and } V(x)=0 \mbox{ for } x\notin \Omega.
\]
 Moreover, for a region $B\subset \Omega$,
\[
\frac{1}{2}\Delta H(x)-\frac14\|\nabla H(x)\|^2 \leq -\lambda_0 \mbox{ for } x\in \Omega\setminus B, 
\]
Then $V(x)$ is a $(\lambda_0,h,B,C)$-Lyapunov function for $\nu\propto \exp(-H(x))$ with 
\[
h=\max_{x\in B}\left( -\frac14\|\nabla H(x)\|^2+\frac12\Delta H(x)+\lambda_0\right)V(x) ,\quad C=\frac{\max_{x\in B}\nu(x)}{\min_{x\in B}\nu(x)}.
\]
\end{lem}
\begin{proof}
For $x\in \Omega$,
\begin{align*}
\lc_\nu V(x)&=\left( -\frac14\|\nabla H(x)\|^2+\frac12\Delta H(x)\right)V(x)\\
&\leq -\lambda_0 V(x)+h1_{x\in B}
\end{align*}
For $x\notin \Omega$, $\lc_\nu V(x)=0$. 
\end{proof}

We next consider a transformation of the partition framework in \cite{MS14} into a mixture model. Define 
\[
d_i(x)=\min\{\|x-y\||y\in \Omega_i\} \mbox{ and } \Omega_i'=\left\{x: d^2_i(x)\leq \frac{1}{n}\right\}
\]
We assume $d^2_i(x)$ is $\bbC^2$ on $\Omega_i'$ for sufficiently large $n$ with bounded derivatives. 

\begin{prop} \label{prop:morse}
Suppose $\pi(x)\propto\exp(-\frac1\epsilon H(x))$, $\Omega_i'=\{x:0<d_i(x)<\frac{1}{\sqrt{n}}\}$, and the following conditions hold:
\begin{enumerate}
\item There is an $\epsilon$ perturbation $H_\epsilon$ such that \eqref{eqn:MSmorse} holds.
\item The boundary of $\Omega_i$ is regular enough so that $d_i^2(x)$ is $\bbC^2$ on $\Omega_i'$, and for any $x_n\to x\in \partial \Omega_i$, $\nabla d_i(x_n)\to v_\bot(x)$, where $v_\bot(x)$ is the outward direction orthogonal to $\partial \Omega_i$.
\item There is $D_{\epsilon}\in(0,\infty)$ such that
\[
\Delta d_i(x)\leq D_\epsilon,\quad \|\nabla H_\epsilon(x)\|\leq D_\epsilon,\quad \Delta H_\epsilon(x)\leq D_\epsilon.
\]
\end{enumerate}
Then, for $\epsilon$ sufficiently small, there is a density $\pi_{\epsilon}$, which is an $\epsilon$ perturbation of $\pi$ and
\[
\pi_{\epsilon}(x) \propto \sum_{i=1}^I p_i\nu_i(x),
\]
where $\nu_i$ has a $(\lambda_0/\epsilon, h_0/\epsilon, B(m_i, a\sqrt{\epsilon}), C )$-Lyapunov function for certain fixed constants $h_0$ and $C$.  
\end{prop} 
\begin{proof}
Consider a clamp function $\psi:\R\to\R$ satisfying 
\begin{enumerate}
\item $\psi$ is $\bbC^2$;
\item $\dot{\psi}< 0, \ddot{\psi}/(\dot{\psi})^2\leq C$;
\item $\psi(x)=1$ for all $x\leq 0$;
\item $\psi(x)=0$ for all $x\geq 1$. 
\end{enumerate}
Let
\[
\Psi_{i}(x)=\exp\left(\frac1\epsilon\log\psi(\sqrt{n} d_i(x))\right).
\]
Then, we can construct
\[
\pi_{\epsilon}\propto \sum_{i=1}^I \Psi_{i}(x)\exp(-\tfrac1\epsilon H_\epsilon(x))=\sum_{i=1}^I \exp\left(-\frac{1}{\epsilon}Q_{\epsilon,i}(x)\right),
\]
where
$Q_{\epsilon,i}(x)=-\log\psi(\sqrt{n}d_i(x))+H_\epsilon(x)$.

We next verify that
\begin{equation} \label{eq:ly_Q}
\frac1{2\epsilon}\Delta Q_{\epsilon, i}-\frac{1}{4\epsilon^2} \|\nabla Q_{\epsilon,i}\|^2\leq -\frac{\lambda_0}{\epsilon}.
\end{equation}
Note that \eqref{eq:ly_Q} holds for any $x\in \Omega_i$ since $Q_{\epsilon,i}(x)=H_\epsilon(x)$. 
When $x\in \Omega_i'\setminus\Omega_i$,  
\[
\nabla Q_{\epsilon,i}(x)=-\sqrt{n}\frac{\dot{\psi}(\sqrt{n}d_i(x))}{\psi(\sqrt{n}d_i(x))}\nabla d_i(x)+\nabla H_\epsilon(x)
\]
We first note that because i) $\nabla d_i(x_n)\to v_\bot (x)$ for any $x_n\to x\in \partial \Omega_i$, ii) $-\nabla H_\epsilon(x)$ points toward the inside of $\Omega_i$ for $x\in \partial \Omega_i$,
and iii) $\nabla^2 d_i$ and $\nabla^2H_\epsilon$ are bounded, for $n$ large enough,
$-\langle \nabla d_i(x), \nabla H_\epsilon(x)\rangle<0$ for $x\in \Omega_i'\setminus\Omega_i$. Then,
\[
\frac{1}{4\epsilon^2} \|\nabla Q_{\epsilon,n}(x)\|^2
\geq\frac1{4\epsilon^2}n\|\nabla d_i(x)\|^2\frac{\dot{\psi}(\sqrt{n}d_i(x))^2}{\psi(\sqrt{n}d_i(x))^2}+\frac{1}{4\epsilon^2}\|\nabla H_\epsilon(x)\|^2. 
\]
We next note that
\begin{align*}
\Delta Q_{\epsilon,n}(x)=&-n\frac{\ddot{\psi}(\sqrt{n}d_i(x))}{\psi(\sqrt{n}d_i(x))}\|\nabla d_i(x)\|^2+n\frac{\dot{\psi}(\sqrt{n}d_i(x))^2}{\psi(\sqrt{n}d_i(x))^2}\|\nabla d_i(x)\|^2\\
&-\sqrt{n}\frac{\dot{\psi}(\sqrt{n}d_i(x))}{\psi(\sqrt{n}d_i(x))}\Delta d_i(x)
+\Delta H_\epsilon(x)
\end{align*}
Thus, for $\epsilon$ small enough, 
\begin{align*}
&\frac1{2\epsilon}\Delta Q_\epsilon(x)-\frac{1}{4\epsilon^2} \|\nabla Q_{\epsilon,n}(x)\|^2\\
\leq& -\frac{n}{2\epsilon}\frac{\ddot{\psi}(\sqrt{n}d_i(x))}{\psi(\sqrt{n}d_i(x))}\|\nabla d_i(x)\|^2
+\frac{n}{2\epsilon}\frac{\dot{\psi}(\sqrt{n}d_i(x))^2}{\psi(\sqrt{n}d_i(x))^2}\|\nabla d_i(x)\|^2\\
&-\frac{n}{2\epsilon}\frac{\dot{\psi}(\sqrt{n}d_i(x))}{\psi(\sqrt{n}d_i(x))}\Delta d_i(x)+\frac{1}{2\epsilon}\Delta H_\epsilon(x)\\
&-\frac1{4\epsilon^2}n\frac{\dot{\psi}(\sqrt{n}d_i(x))^2}{\psi(\sqrt{n}d_i(x))^2}\|\nabla d_i(x)\|^2
-\frac{1}{4\epsilon^2}\|\nabla H_\epsilon(x)\|^2\\
\leq& \frac{1}{2\epsilon}\Delta H_\epsilon(x)-\frac{1}{4\epsilon^2}\|\nabla H_\epsilon(x)\|^2 \leq -\frac{\lambda_0}{\epsilon}
\end{align*}

Lastly, we note that
\[
 \exp\left(-\frac1\epsilon H_\epsilon(x)\right) \leq \sum_{i=1}^I \exp\left(-\frac{1}{\epsilon}Q_{\epsilon,i}(x)\right) \leq I\exp\left(-\frac1\epsilon H_\epsilon(x)\right).
\]
Moreover $q(x)\propto \exp\left(-\frac1\epsilon H_\epsilon(x)\right)$ is a $\epsilon$ perturbation of $\pi$. Thus, $\pi_{\epsilon}$ is a $\epsilon$ perturbation of $\pi$. 
\end{proof}

\subsection{Example: Bimodal densities from Bayesian statistics} 
\label{sec:expMorse}
In this section we provide a simple concrete example to demonstrate how mixtures of singular densities arise in practice, and how to implement the Morse function framework discussed above. 

Suppose we want to obtain the posterior density $p(x|y_1,\ldots,y_n)$ where the prior is $\mathcal{N}(0,2)$ and the observation model is given by 
\[
y_i=x^2+\xi_i,\quad \xi_i\sim \mathcal{N}(0,1).
\]
Then, the posterior density is given by 
\begin{align*}
p(x|y_1, \dots, y_n)\propto\exp\left(-\frac12\left(2x^2+\sum_{i=1}^n(x^2-y_i)^2\right)\right)
\propto \exp\left(-\frac n2(x^2-m_n)^2  \right).
\end{align*}
where $m_n=\tfrac1n\sum_{i=1}^n y_i-\tfrac{1}{n}$. It is easy to see that when $m_n>0$, $p(x|y_1,\dots, y_n)$ has two modes $\pm \sqrt{m_n}$. For $m_n=1$, this density is also known as the double-well potential. 
We next show that we can decompose the double-well potential into a mixture.

\begin{lem}
\label{lem:doublewell}
For $\pi(x)\propto \exp(-\tfrac12n(x^2-a^2)^2)$ with $a>0$, 
\[
\pi(x)\propto \nu_+(x)+\nu_-(x)
\]
where $\nu_+(x)=\exp(-\tfrac12n(x^2-a^2)^2)1\{x\geq 0\}$ 
and $\nu_-(x)=\exp(-\tfrac12n(x^2-a^2)^2)1\{x<0\}$.
Moreover, for $\epsilon$ sufficiently small, there is a density $\pi_{\epsilon}$, 
which is an $\epsilon$ perturbation of $\pi$ and
\[
\pi_{\epsilon}(x) \propto \nu_1(x) + \nu_2(x)
\]
where $\nu_1$ has a $(na^2, nh, B(a, \sqrt{n} r), C )$-Lyapunov function 
and $\nu_1$ has a $(na^2, nh, B(-a, \sqrt{n} r), C )$-Lyapunov function 
for certain fixed constants $h,C$.  
\end{lem}
\begin{proof}
Let $H(x)=\tfrac12 (x^2-a^2)^2$ and $\epsilon=1/n$.
We first note that
\[
\nabla H(x)=2x(x^2-a^2)
\mbox{ and }
\nabla^2 H(x)=6x^2-2a^2.
\]
Thus,
\[
\frac{1}{2\epsilon} \nabla^2 H(x)-\frac{1}{4\epsilon^2} \|\nabla H(x)\|^2
=\frac{3x^2}{\epsilon}-\frac{a^2}{\epsilon}-\frac1{\epsilon^2} x^2 (x^2-a^2)^2.
\]

%
When $|x-a|^2\geq 3\epsilon/a^2$ and $x>0$,
\[
\frac1{\epsilon^2} x^2 (x^2-a^2)^2=\frac1{\epsilon^2} x^2 (x-a)^2(x+a)^2
\geq \frac1{\epsilon^2} x^2 \frac{3\epsilon}{a^2} a^2\geq \frac{3x^2}{\epsilon}.
\]
Then,
\[
\frac{1}{2\epsilon} \nabla^2 H(x)-\frac{1}{4\epsilon^2} \|\nabla H(x)\|^2\leq -\frac{a^2}{\epsilon}.
\]
Similarly, when $|x+a|^2\geq 3\epsilon/a^2$ and $x<0$, we also have
\[
\frac{1}{2\epsilon} \nabla^2 H(x)-\frac{1}{4\epsilon^2} \|\nabla H(x)\|^2\leq -\frac{a^2}{\epsilon}.
\]
In this case, $H_\epsilon=H$ already satisfies \eqref{eqn:MSmorse}. (There is no saddle point for this problem.)

Next if we split $R$ into $\Omega_1=[0,\infty)$ and $\Omega_2=(-\infty,0]$. It is easy to see that $d_1(x)=-x$ is $\bbC^2$ in $(-\infty, 0)$. In addition, $\nabla d_1(x)=-1$, which is the same as the outward direction for $\Omega_1$ at $x=0$. 
Similarly, $d_2(x)=x$ is $\bbC^2$ in $(0,\infty)$ and $\nabla d_2(x)=1$ is the same as the outward direction for $\Omega_2$ at $x=0$.
Thus, the existence of the $\pi_{\epsilon}$ follows from Proposition \ref{prop:morse}. 

\end{proof}

\section{Analysis for ReLD} \label{sec:ReLD_proof}
In this section, we provide detailed analysis on how the replica-exchange mechanism speeds up the convergence
of ReLD over LD. We also present the proof of Theorem \ref{th:main1}.

\subsection{A roadmap}
\label{sec:roadmap}
Let $\bar \theta=\E_{\pi\pi^{y}}[f(X,Y)]$,
\[\eta_i(y)=\int f(x,y)\nu_i(x)dx \mbox{ and } \theta_i=\int \eta_i(y)\pi^{y}(y)dy\]
for $i=1,2, \dots, I$. 
First, based on the form of $\pi$, the variance of $f(X,Y)$ can be decomposed as 
\[\begin{split}
\var_{\pi\pi^{y}}(f(X,Y))&=\int (f(x,y)-\bar\theta)^2\pi(x)\pi^{y}(y)dxdy\\
&=\sum_{i=1}^Ip_i\int (f(x,y)-\bar\theta)^2\nu_i(x)\pi^{y}(y)dxdy.
\end{split}\]
Then, because
\[
f(x,y)-\bar{\theta}=(f(x,y)-\eta_i(y))+(\eta_i(y)-\theta_i)+(\theta_i-\bar{\theta}),
\]
by Cauchy-Schwarz inequality, we can further decompose the variance as
\begin{equation}\label{eq:main_decomp}
\begin{split} 
\var_{\pi\pi^{y}}(f(X,Y))\leq&3 \sum_{i=1}^Ip_i\underbrace{\int (f(x,y)-\eta_i(y))^2\nu_i(x)\pi^{y}(y)dxdy}_{\mbox{(A)}}\\
&+3\sum_{i=1}^Ip_i\underbrace{\int (\eta_i(y)-\theta_i)^2\pi^{y}(y)dy}_{\mbox{(B)}} + 3\underbrace{\sum_{i=1}^Ip_i(\theta_i-\bar \theta)^2}_{\mbox{(C)}}.
\end{split}\end{equation}
Note that part (A) is the variance of $f$ under 
$\nu_i$ with $y$ being fixed. Part (B) is the variance of $\eta_i$ under $\pi^y$. Since $\nu_i$ and $\pi^y$ satisfy the Lyapunov condition, parts (A) and (B) can be controlled using Proposition \ref{prop:bakry_lyap}. 

For part (C), as $\bar{\theta}=\sum_{i=1}^I p_i\theta_i$,  
\begin{align*}
\sum_{i=1}^{I}p_i(\theta_i-\bar\theta)^2=\sum_{i=1}^I p_i\left(\sum_{j=1}^{I}p_j(\theta_i-\theta_i)\right)^2
\leq \sum_{i,j} p_ip_j(\theta_i-\theta_j)^2 
\end{align*}
by Cauchy-Schwarz inequality.
Therefore, we need an upper bound for 
\[
(\theta_i-\theta_j)^2=\left(\E_{\nu_i\pi^y}[f(X,Y)] - \E_{\nu_j\pi^y}[f(X,Y)]\right)^2.
\]
When running the naive LD, \cite{MS14} provides an estimate of the difference between
$\E_{\nu_i}[f(X)]$ and $\E_{\nu_j}[f(X)]$ (see Theorem 2.12 in \cite{MS14}).
The estimate depends on the saddle height, and when $\nu_i\propto \phi((x-m_i)/\epsilon)$, it grows exponentially in $1/\epsilon$. One of the main technical contribution of this paper is to find an upper bound for the mean difference in the ReLD setting. In particular, we establish that the ratio between the mean difference square and the carre du champ of ReLD stays invariant when $\epsilon$ goes to zero. 

To achieve a better mean difference and subsequently the PI constant, we need to exploit the additional exchange term that arises in the carr\'e du champ for ReLD. In particular, we focus on the term
\[\begin{split}
&\E_{\pi\pi^y}\left[\rho s(X,Y)(f(Y,X)-f(X,Y))^2\right]\\
=&\rho\int (f(y,x)-f(x,y))^2 (\pi(x)\pi^y(y))\wedge(\pi(y)\pi^y(x))dxdy\\
\end{split}\] 
We first note that 
\[\begin{split}
&\sum_{i,j}p_ip_j\rho\int (f(y,x)-f(x,y))^2  (\nu_i(x)\pi^y(y))\wedge(\nu_j(y)\pi^y(x))dxdy\\
\leq& \rho\int (f(y,x)-f(x,y))^2 (\pi(x)\pi^y(y))\wedge(\pi(y)\pi^y(x))dxdy
\end{split}\]
In the following, we refer $(\nu_i(x)\pi^y(y))\wedge(\nu_j(y)\pi^y(x))$ as a ``maximal coupling density" as its formulation is similar to the $L_1$-maximal coupling between $(\nu_i(x)\pi^y(y))$ and $(\nu_j(y)\pi^y(x))$ \cite{lindvall02}.
However, it is our experience that this ``maximal coupling density" is still difficult to deal with. To resolve the challenge, we replace $\nu_i$ by $u_{B(m_i,r_i)}$, which is the uniform distribution on $B(m_i,r_i)$, and $\pi^y$ by $u_{B(m_j,R_j)}$ using appropriate bounding arguments. The ``maximal coupling density" with uniform distributions is much easier to handle, and we can build an upper bound for the transformed mean difference
\[
\left(\int f(x,y) u_{B(m_i,r_i)}(x)u_{B(m_j,R_j)}(y)dxdy-
\int f(x,y)u_{B(m_i,R_i)}(y)u_{B(m_j,r_j)}(x) dxdy\right)^2.
\]
Since uniform distributions will play a pivotal role in our analysis, in what follows, we first develop some auxiliary results regarding uniform distribution. 

\subsection{Auxiliary Lemmas}
For a given bounded  domain, $D\subset \mathbb{R}^d$, 
we denote $u_D$ as the uniform distribution on $D$. 
We also write $V_D$ as the volume of $D$. 
Then $u_D(x)=1/V_D$ for any $x\in D$.  
A special bounded convex domain is a ball. We denote $B(x_0,R)=\{x: \|x-x_0\|^2\leq R^2\}$ as a $d$-dimensional ball centered at $x_0$ and having radius $R$. When $x_0$ and $R$ is clear from the context, we may also write the ball as $B$ for conciseness. 

\begin{lemma} \label{lm:univar}
Consider a univariate density $p$ on $(-R,R)$, here $R$ can be $\infty$ if the support of $p$ is $\R$. 
Suppose there is a function $Q$ that is decreasing and differentiable on $[-R,0)$, and is increasing and differentiable on $(0,R]$.
In addition, $Q(0)=0$, and
\[\frac{\int_{x}^{R}Q(t)p(t)dt}{q(x)}\leq \kappa p(x) \mbox{ for $x>0$, and }
\frac{\int_{-R}^{x}Q(t)p(t)dt}{|q(x)|}\leq \kappa p(x) \mbox{ for $x<0$,}\]
where $q(x)=\frac{dQ(x)}{dx}$.
Then,
\[
\var_p(f)\leq \E_p[|f(X)-f(0)|^2]\leq \kappa\E_p[|\nabla f(X)|^2].
\]
Consequentially, $p$ follows a $\kappa$-PI. 
\end{lemma}
\begin{proof}
We first note that
\[(f(x)-f(0))^2=\left(\int_{0}^{x}\nabla f(y)dy\right)^2 \leq \left(\int_{0}^{x}|\nabla f(y)|^2\frac{1}{q(y)}dy\right)\left(\int_{0}^{x}q(y)dy\right).\]
\[\begin{split}
\int_{0}^{R}|f(x)-f(0)|^2p(x)dx 
\leq& \int_{0}^{R}\int_{0}^{x}|\nabla f(y)|^2\frac{1}{q(y)}dyQ(x)p(x)dx\\
=& \int_{0}^{R}|\nabla f(y)|^2\frac{\int_{y}^{R}Q(x)p(x)dx}{q(y)}dy\\
\leq&\kappa\int_{0}^{R}|\nabla f(y)|^2p(y)dy.
\end{split}\]
Similarly, we can show that
\[
\int_{-R}^{0}|f(x)-f(0)|^2p(x)dx \leq \kappa\int_{-R}^{0}|\nabla f(y)|^2p(y)dy.
\]
Thus, 
\[
\int_{-R}^{R}|f(x)-f(0)|^2p(x)dx \leq \kappa\int_{-R}^{R}|\nabla f(x)|^2p(x)dx.
\]
\end{proof}

\begin{lemma} \label{lm:uniform_ball}
Given a ball $B=B(x_0, R) \subset \mathbb{R}^d$, $u_B$ satisfies a $R^2$-PI:
\begin{equation} \label{eq:PI_bound} 
\var_{u_B}(f(X))\leq R^2\E_{u_B}[\|\nabla f(X)\|^2].
\end{equation}
\end{lemma}
\begin{proof}
We prove the result by induction on the dimension. It is without loss of generality to assume $x_0=0$.
First, for $d=1$, $B=[-R,R]$. In this case
\[\begin{split}
\var_{u_B}(f(X)) &= \int_{-R}^{R} (f(x) - \E_{u_B}[f(X)])^2\frac{1}{2R} dx\\
&=\frac{1}{2}\int_{-R}^{R} \int_{-R}^{R} (f(x) - f(x^{\prime}))^2\frac{1}{2R} \frac{1}{2R}dxdx^{\prime}\\
&=\frac{1}{8R^2}\int_{-R}^{R} \int_{-R}^{R} (\int_{x}^{x^{\prime}}|\nabla f(s)|ds)^2 dxdx^{\prime}\\
&\leq \frac{1}{8R^2} \int_{-R}^{R} \int_{-R}^{R}|x-x^{\prime}|\int_{-R}^{R}|\nabla f(s)|^2 ds dxdx^{\prime}\\
&\leq  \frac{2R^2}{3}\int_{-R}^{R}|\nabla f(s)|^2\frac{1}{2R} ds \leq R^2  \E_{u_D}[\|\nabla f(X)\|^2]
\end{split}\]
Next, suppose the \eqref{eq:PI_bound} holds for $(d-1)$-dimensional balls, and $B\subset \mathbb{R}^d$.
We use $(x,y)$, $x\in \mathbb{R}^{d-1}$ and $y\in \mathbb{R}$, to denote a $d$-dimensional vector. We also write
\[\tilde B_y=\{x: (x,y)\in B\} \mbox{ and } g(y) =\int_{\tilde B_y}f(x,y)\frac{1}{V_{\tilde B_y}}dx.\]
Note that $\tilde B_y\subset\mathbb{R}^{d-1}$ is a ball with radius $\sqrt{R^2-y^2}$.
We also note that
\begin{equation} \label{eq:decomp1}
\var_{u_B}(f(X,Y))=\E_{u_B}[(f(X,Y)-g(Y))^2]+\E_{u_B}[(g(Y) - \E_{u_B}[f(X,Y)])^2]
\end{equation}
We can then analyze the two parts on the right hand side of \eqref{eq:decomp1} one by one.
For the first part, for any fixed $y$, $f(x,y)$ is a $(d-1)$-dimensional function. By induction, 
\[\begin{split}
\int_{\tilde B_y}(f(x,y) - g(y))^2\frac{1}{V_{\tilde B_y}}dx &\leq  (R^2-y^2)\int_{\tilde B_y}\|\nabla_xf(x,y)\|^2\frac{1}{V_{\tilde B_y}}dx\\
&\leq R^2 \int_{\tilde B_y}\|\nabla_xf(x,y)\|^2\frac{1}{V_{\tilde B_y}}dx.
\end{split}\]
Therefore,
\begin{equation} \label{eq:lm1_bd1}
\E_{u_B}[(f(X,Y)-g(Y))^2]\leq  R^2\int_{B}\|\nabla_xf(x,y)\|^2\frac{1}{V_{B}}dxdy.
\end{equation}
For the second part of \eqref{eq:decomp1}, note that $V_{\tilde B_y}$ as a function of $y$ is increasing for $y\in[-R,0]$ and decreasing for $y\in[0,R]$.
Let $p(y)=V_{\tilde B_y}/V_B$, i.e., it is the marginal density of $y$ under $u_B$. In particular, $\E_p[g(Y)]=\E_{u_B}[f(X,Y)]$.
We also define $Q(y)=|y|$ and $q(y)=\sign(y)$. 
Then for $z>0$,
\[\frac{\int_{z}^{R}Q(y)p(y)dy}{q(z)} \leq p(z)\int_{z}^{R}Q(y)dy\leq \frac{R^2}{2}p(z).\]
Likewise, for $z<0$,
\[\frac{\int_{-R}^{z}Q(y)p(y)dy}{|q(z)|} \leq \frac{R^2}{2}p(z).\]
Thus, we can apply Lemma \ref{lm:univar}. In particular,
\begin{equation}\begin{split} \label{eq:lm1_bd2}
 \E_{u_B}[(g(Y) - \E_{u_B}[f(X,Y)])^2] &=\var_p(g(Y))\\
&\leq \E_{p}[(g(Y)-g(0))^2]\\
&\leq\frac{R^2}{2}\int_{-R}^R\|\nabla_y g(y)\|^2p(y)dy \mbox{ by Lemma \ref{lm:univar}}\\
&=\frac{R^2}{2}\int_{-R}^R\left(\nabla_y \int_{\tilde B_y} f(x,y)\frac{1}{V_{\tilde B_y}}dx\right)^2p(y)dy\\
&\leq \frac{R^2}{2}\int_{B}\|\nabla_y f(x,y)\|^2\frac{1}{V_B}dxdy
\end{split}\end{equation}
Combining the bounds in \eqref{eq:lm1_bd1} and \eqref{eq:lm1_bd2}, we have
\[\begin{split}
\var_{u_B}(f(X,Y)) &\leq R^2\int_{B}\|\nabla_xf(x,y)\|^2\frac{1}{V_{B}}dxdy + \frac{R^2}{2}\int_{B}\|\nabla_y f(x,y)\|^2\frac{1}{V_B}dxdy\\
&\leq R^2 \int_{B}\|\nabla f(x,y)\|^2\frac{1}{V_{B}}dxdy.
\end{split}\]
\end{proof}

%

For a given measure $\mu$, we denote 
\[
\mu_D(x)=\frac{\mu(x)1_{D}(x)}{\int_D\mu(y)dy},
\]
i.e., the measure $\mu$ conditional on being in the bounded domain $D$.
\begin{lemma}\label{lm:gen_uni}
Given a ball $B=B(x_0, R)\subset \mathbb{R}^d$, suppose 
\[\frac{\max_{x\in B}\mu(x)}{\min_{x\in B}\mu(x)} \leq C.\]
Then 
\[\var_{\mu_B}(f(X)) \leq C^2R^2 \E_{\mu_B}[\|\nabla f(X)\|^2].\]
\end{lemma}
\begin{proof}
Recall that $u_B$ is a uniform distribution on $B$. Then
\[V_B \min_{x\in B} \mu(x) \leq 1=\int_B \mu_B (x) \leq V_B\max_{x\in B}\mu(x).\]
This implies that
\begin{equation}\label{eq:lm2_bd}
\frac{1}{C}\leq \frac{\mu_B(x)}{u_B(x)} \leq C.
\end{equation}
Next
\[\begin{split}
\var_{\mu_B}(f(X)) &\leq \int_{B}(f(x)-\E_{u_B}[f(X)])^2\mu_B(x)dx\\
&\leq C\int_{B}(f(x)-\E_{u_B}[f(X)])^2u_B(x)dx \mbox{ by \eqref{eq:lm2_bd}}\\
&\leq CR^2\int_{B} \|\nabla f(x)\|^2 u_B(x)dx \mbox{ by Lemma \ref{lm:uniform_ball}}\\
&\leq C^2R^2 \int_{B} \|\nabla f(x)\|^2 \mu_B(x)dx \mbox{ by \eqref{eq:lm2_bd}}.
\end{split}\]
\end{proof}

\subsection{PI and Lyapunov function }
\label{sec:proof24}

We are now ready to prove Proposition \ref{prop:bakry_lyap}
\begin{proof}[Proof of Proposition \ref{prop:bakry_lyap}]
The arguments we use here are similar to the ones used in \cite{Bakry}. The only difference is that we use Lemma \ref{lm:gen_uni} to find the bounding constants explicitly. Note that for any constant $c$,
\[
\int (f(x)-c)^2 \nu(x)dx \leq \underbrace{\int\frac{-\lc_{\nu}V(x)}{\lambda V(x)}(f(x)-c)^2\nu(x)dx}_{\mbox{(I)}} + \underbrace{\int (f(x)-c)^2\frac{b}{\lambda V(x)}1_B(x)\nu(x)dx}_{\mbox{(II)}}.
\]
\noindent{\bf For part (I),} note that
\[\begin{split}
&\int\frac{-\lc_{\nu}V(x)}{V(x)}(f(x)-c)^2\nu(x)dx\\
=&\int \left\langle\nabla\left(\frac{(f(x)-c)^2}{V(x)}\right),\nabla V(x)\right\rangle \nu(x)dx\mbox{   by equation (1.7.1) in \cite{bakry2013analysis}}\\
=&2\int\frac{f(x)-c}{V(x)}\langle\nabla f(x),\nabla V(x)\rangle \nu(x)dx - \int\frac{(f(x)-c)^2}{V(x)^2}\|\nabla V(x)\|^2\nu(x)dx\\
=&\int\|\nabla f(x)\|^2\nu(x)dx - \int\left\|\nabla f(x) - \frac{f(x)-c}{V(x)}\nabla V(x)\right\|^2\nu(x)dx\\
\leq&\int\|\nabla f(x)\|^2\nu(x)dx.
\end{split}\]
\noindent{\bf For part (II),} recall $\nu_B$ is $\nu$ conditioned on that $x\in B$. Set
\[
c=\int f(x)u_B(x)dx.
\]
i.e., the expected value of $f$ over the uniform distribution on $B$.
Then,
\[\begin{split}
\int_{B}\frac{(f(x)-c)^2}{V(x)}\nu(x)dx &=\P_\nu(X\in B)\int_{B}\frac{(f(x)-c)^2}{V(x)}\nu_B(x)dx\\
&\leq \int_B(f(x)-c)^2 \nu_B(x)dx \mbox{ as $V(x)\geq 1$ and $\P_\nu(X\in B)\leq 1$}\\
&\leq C^2R^2\int_B\|\nabla f(x)\|^2\nu_B(x)dx \mbox{ by Lemma \ref{lm:gen_uni}}\\
&\leq C^2R^2\int \|\nabla f(x)\|^2\nu(x)dx.
\end{split}\]
Putting the two parts together, we have 
\[
\var_\nu(f) \leq \left(\frac{1}{\lambda}+\frac{bC^2R^2}{\lambda}\right)\E_{\nu}[\|\nabla f(X)\|^2].\]
\end{proof}

\subsection{Mean difference estimates}
As explained in Section \ref{sec:roadmap}, the key step in our analysis lies in building mean difference estimates. 
In this section, we present several important results for the mean difference estimates.

Our first result shows that we can replace a density having Lyapunov functions with a uniform distribution, while keeping the difference controlled. 
\begin{lemma} \label{lm:nu_uni}
Suppose $\nu$ has a $(\lambda,b,B(x_0,R),C)$-Lyapunov function, then 
\[\left(\E_{\nu}[f(X)]-\E_{u_B}[f(U)]\right)^2\leq 4\frac{1+bR^2C^2}{\lambda}\E_{\nu}[\|\nabla f(X)\|^2].\]
\end{lemma}
\begin{proof}
Let $\bar{f}_\nu=\E_{\nu}[f(X)]$ and $\bar{f}_{u_B}=\E_{u_B}[f(U)]$.
Then
\[\begin{split}
\left(\bar{f}_\nu - \bar{f}_{u_B}\right)^2 &\leq 2\E_\nu[(f(X)-\bar{f}_\nu)^2]+2\E_\nu[(f(X)-\bar{f}_{u_B})^2]\\
&\leq 4\frac{1+bR^2C^2}{\lambda}\E[\|\nabla f(X)\|^2] \mbox{ by Proposition \ref{prop:bakry_lyap}}.
\end{split}\]
\end{proof}

Our second result bound the mean difference when moving from a big Uniform ball to a small Uniform ball with the same center.
 
\begin{lemma}\label{lm:uni_uni}
Consider $B_r=B(x_0, r)$ and $B_R=B(x_0,R)$ with $R\geq r$. Then when $d=1$,
\[\left(\E_{u_{B_r}}[f(X)] - \E_{u_{B_R}}[f(X)]\right)^2\leq R^2\log(R/r)\E_{u_{B_R}}[\|\nabla f(X)\|^2];\]
when $d\geq 2$,
\[\left(\E_{u_{B_r}}[f(X)] - \E_{u_{B_R}}[f(X)]\right)^2\leq \frac{R^{d+1}}{(d-1)r^{d-1}}\E_{u_{B_R}}[\|\nabla f(X)\|^2].\]
\end{lemma}

\begin{proof}
Without loss of generality, we assume $x_0=0$. 

\noindent{\bf We first consider the case in which  $r=1$ and $d\geq2$.}
Let $C_V$ denote the volume of a $d$-dimensional unit ball. 
Consider the spherical coordinate of $x$.
In particular, let $t\in[0,R]$ denote the radial coordinate,
and $\theta=(\theta_1, \theta_2, \dots, \theta_{d-1})$ denote the angular coordinate, i.e., it is a $(d-1)$ dimensional vector with
$\theta_i\in [0,\pi]$ for $i=1, \dots, d-2$ and $\theta_{n-1}\in[0,2\pi)$. 
We also write $\xi(\theta)$ be a d-dimensional vector on $S^{d-1}$  with $\xi_1(\theta)=\cos(\theta_1)$, 
\[\xi_i(\theta)=\sin(\theta_1)\cdots\sin(\theta_{i-1})\cos(\theta_i)\]
for $1<i<d$, and $\xi_{d}(\theta)=\sin(\theta_1)\cdots\sin(\theta_{d-1})$. Then, $x=r\xi(\theta)$.
We also write 
\[d_{S^{d-1}}\theta = \sin^{d-2}(\theta_1)\sin^{d-3}(\theta_2)\dots \sin(\theta_{d-1}) d\theta\] 
and $\Omega = [0,\pi]^{d-2}\times[0,2\pi)$. Then
\[\begin{split}
\E_{u_{B_R}}[f(X)]&=\frac{1}{C_V R^d}\int_{\Omega}\int_{0}^{R} f(t\xi(\theta)) t^{d-1} dt d_{S^{d-1}}\theta\\
&=\frac{1}{C_V}\int_{\Omega}\int_{0}^{1} f(Rt\xi(\theta)) t^{d-1} dt d_{S^{d-1}}\theta.
\end{split}\]
Using the spherical coordinate representation, we have
\[\begin{split}
&\left(\E_{u_{B_1}}[f(X)] - \E_{u_{B_R}}[f(X)]\right)^2\\ 
= &\left(\frac{1}{C_V}\int_{\Omega}\int_{0}^{1} f(t\xi(\theta)) t^{d-1} dt d_{S^{d-1}}\theta - \frac{1}{C_V}\int_{\Omega}\int_{0}^{1} f(Rt\xi(\theta)) t^{d-1} dt d_{S^{d-1}}\theta\right)^2\\
\leq& \frac{1}{C_V}\int_{\Omega}\int_{0}^{1}(f(Rt\xi(\theta)) - f(t\xi(\theta)))^2t^{d-1}dt d_{S^{d-1}}\theta ~\mbox{ by Jensen's inequality}\\
=& \frac{1}{C_V}\int_{\Omega}\int_{0}^{1}\left(\int_{1}^{R}\sum_{i=1}^{d}\nabla_{x_i}f(st\xi(\theta))t\xi_i(\theta)ds\right)^2
t^{d-1}dt d_{S^{d-1}}\theta \\
\leq& \frac{1}{C_V}\int_{\Omega}\int_{0}^{1}R\int_{1}^{R}\left(\sum_{i=1}^{d}\nabla_{x_i}f(st\xi(\theta))t\xi_i(\theta)\right)^2ds
t^{d-1}dt d_{S^{d-1}}\theta ~ \mbox{ by Jensen's inequality}\\
\leq& \frac{R}{C_V}\int_{\Omega}\int_{0}^{1}\int_{1}^{R}\|\nabla f(st\xi(\theta))\|^2ds t^{d+1}dt d_{S^{d-1}}\theta ~\mbox{ by Cauchy-Schwarz inequality and $\|\xi\|=1$}\\
=& \frac{R}{C_V} \int_{1}^{R}\int_{\Omega}\int_{0}^{1}\|\nabla f(st\xi(\theta))\|^2t^{d+1}dt d_{S^{d-1}}\theta ds\\
=& \frac{R}{C_V}\int_{1}^{R}\int_{\Omega}\int_{0}^{s}\|\nabla f(r\xi(\theta))\|^2r^{d+1}dr d_{S^{d-1}}\theta \frac{1}{s^{d+2}}ds\mbox{ by letting $r=st$} \\
\leq&\frac{R}{C_V}\int_{1}^{R}\int_{\Omega}\int_{0}^{s}\|\nabla f(r\xi(\theta))\|^2r^{d-1}dr d_{S^{d-1}}\theta \frac{1}{s^{d}}ds\\
\leq& R^{d+1} \left(\frac{1}{C_VR^d}\int_{\Omega}\int_{0}^{R}\|\nabla f(r\xi(\theta))\|^2r^{d-1}dr d_{S^{d-1}}\theta\right) \int_{1}^{R}\frac{1}{s^{d}} ds\\
\leq&\frac{R^{d+1}}{d-1}\E_{u_{B_R}}[\|\nabla f(X)\|^2]. 
\end{split}\]

\noindent{\bf When $d=1$,} following similar arguments as above, we can show that
\[\begin{split}
\left(\E_{u_{B_1}}[f(X)] - \E_{u_{B_R}}[f(X)]\right)^2
= & \left(\frac{1}{2}\int_{-1}^{1}f(Rt)-f(t)dt\right)^2\\
\leq&\frac{1}{2}\int_{-1}^{1}R\int_{1}^{R}\|\nabla f(st)\|^2ds t^2 dt\\
=&R^2\int_{1}^{R}\frac{1}{2R}\int_{-s}^{s}\|\nabla f(r)\|^2r^2dr \frac{1}{s^3}ds \\
\leq &R^2\int_{1}^{R}\frac{1}{2R}\int_{-s}^{s}\|\nabla f(r)\|^2dr \frac{s^2}{s^3}ds \\
\leq&R^2(\log(R) -1)\left(\frac{1}{2R}\int_{-R}^{R}\|\nabla f(r)\|^2dr\right).
\end{split}\]

\noindent{\bf For general $r>0$,} we can simply let $Z=X/r$, $g(X)=f(X/r)$ and $q=R/r$. As 
\[
\E_{u_{B_1}}[g(Z)]=\E_{u_{B_r}}[f(X)] \mbox{ and }
\E_{u_{B_q}}[g(Z)]=\E_{u_{B_R}}[f(X)],
\]
\[\begin{split}
\left(\E_{u_{B_r}}[f(X)]-\E_{u_{B_R}}[f(X)]\right)^2
=&\left(\E_{u_{B_1}}[g(Z)]-\E_{u_{B_q}}[g(Z)]\right)^2\\
&\leq \frac{q^{d+1}}{d-1}\E_{u_{B_q}}[\|\nabla g(X)\|^2].
\end{split}\]
Then, as
\[
\E_{u_{B_q}}[\|\nabla g(Z)\|^2]
=r^2\E_{u_{B_R}}[\|\nabla f(X)\|^2],
\]
we have our claim. 
\end{proof}

The next result is our main result on the mean difference.
\begin{prop}
\label{prop:swap}
Consider four densities $\nu^x_1,\nu^x_2,\nu^y_1,\nu^y_2$.
Suppose $\nu^x_i$ is a $\Ly(r_i,q,a)$-density with center $m_i$ for $i=1,2$. Similarly, suppose $\nu^y_i$ is a $\Ly(R_i,Q,A)$-density with center $m_i$ for $i=1,2$. Moreover, for $i=1,2$,  there are constants $R,r, a,A$ so that 
$R_i\leq R$ and $R_i/r_i \leq R/r$.
Then 
\[\begin{split}
&\left(\E_{\nu_1^x\nu_2^{y}}[f(X,Y)]-\E_{\nu^x_2\nu_1^{y}}[f(X,Y)] \right)^2\\
\leq&\Xi_x \int \|\nabla_x f(x,y)\|^2(\nu^x_1(x)\nu_2^{y}(y)+\nu^x_2(x)\nu_1^{y}(y))dxdy\\
&+\Xi_y \int \|\nabla_y f(x,y)\|^2(\nu^x_1(x)\nu_2^{y}(y)+\nu^x_2(x)\nu_1^{y}(y))dxdy\\
&+\Xi_e \int(f(x,y) - f(y,x))^2\left(\nu^x_1(x)\nu^{y}_2(y)\wedge \nu_2^{x}(y) \nu^y_1(x)\right)dxdy,
\end{split}\]
where
\[\begin{split}
\Xi_x =28qA,
~ \Xi_y=28Q + 7aA\left(\frac{R^{d+1}}{r^{d-1}}\right)\left(\log\left(\frac{R}{r}\right)\right)^{1_{d=1}},
\mbox{ and }\Xi_e=7\left(\frac{R}{r}\right)^daA .
\end{split}\]
\end{prop}

\begin{proof}
To simplify the notations, we define $\B_i=B(m_i,R_i)$, $B_i=B(m_i,r_i)$, and 
\[
\eta_i(y)=\int f(x,y)\nu^x_i(x)dx
\]
for $i=1,2$.

\noindent{\bf Step 1.} We change $\pi_2^y$ to $u_{\B_2}$.
By Lemma \ref{lm:nu_uni}, we can control the difference in means by
\[\begin{split}
&\left(\E_{\pi_1^x\nu_2^{y}}[f(X,Y)] - \E_{\nu_1^xu_{\B_2}}[f(X,Y)]\right)^2\\
=&\left(\int \eta_1(y) \nu_2^{y}(y)dy - \int \eta_1(y)u_{\B_2}(y)dy\right)^2\\ 
\leq& 4Q\int\|\nabla_y \eta_1(y)\|^2\nu_2^{y}(y)dy\\
\leq& 4Q\int\|\nabla_yf(x,y)\|^2\nu^x_1(x)\nu_2^{y}(y)dy
\mbox{ by Jensen's inequality.}
\end{split}.\]
Likewise,  we change $\pi_1^x$ to $u_{B_1}$. By Lemma \ref{lm:nu_uni}, we can control the difference in means by:
\[\begin{split}
&\left(\E_{\pi_1^xu_{\B_2}}[f(X,Y)] - \E_{u_{B_1}u_{\B_2}}[f(X,Y)]\right)^2\\
=&\left(\int \left(\int f(x,y)\nu^x_1(x)dx-\int f(x,y)u_{B_1}(x)dx\right) u_{\B_2}(y)dy\right)^2\\ 
\leq&\int \left(\int f(x,y)\nu^x_1(x)dx-\int f(x,y)u_{B_1}(x)dx\right)^2 u_{\B_2}(y)dy  \mbox{ by Jensen's inequality}\\
\leq& 4q\int\|\nabla_xf(x,y)\|^2\nu^x_1(x)u_{\B_2}(y)dy\\
\leq& 4qA\int\|\nabla_xf(x,y)\|^2\nu^x_1(x)\pi^y_2(y)dy.
\end{split}\]

\noindent{\bf Step 2.} We replace $u_{\B_2}$ with $u_{B_2}$. By Lemma \ref{lm:uni_uni}, we can control the difference in means by
\[\begin{split}
&\left(\E_{u_{B_1}u_{\B_2}}[f(X,Y)] - \E_{u_{B_1}u_{B_2}}[f(X,Y)]\right)^2\\
\leq&\int \left(\int f(x,y)u_{\B_2}(y)dy-\int f(x,y)u_{B_2}(y)dy\right)^2 u_{B_1}(x)dx\\
\leq& \frac{R^{d+1}}{r^{d-1}}\int \|\nabla_yf(x,y)\|^2u_{\B_2}(y)u_{B_1}(x)dydx\\
\leq& \frac{R^{d+1}}{r^{d-1}}a A \int \|\nabla_yf(x,y)\|^2\nu_1^x(x)\nu_2^{y}(y)dxdy,
\end{split}\]
when $d\geq 2$. If $d=1$, an additional $\log (R/r)$ is needed. 

\noindent{\bf Step 3.} The mean difference in exchanging $B_1$ and $B_2$ can be bounded by the additional replica exchange carre du champ term:
\[\begin{split}
&\left(\E_{u_{B_1}u_{B_2}}[f(X,Y)] - \E_{u_{B_2}u_{B_1}}[f(X,Y)]\right)^2\\
\leq&\int(f(x,y) - f(y,x))^2u_{B_1}(x)u_{B_2}(y)dxdy ~\mbox{ 
by Jensen's inequality}\\
\leq& \left(\frac{R}{r}\right)^d \int(f(x,y) - f(y,x))^2 (u_{B_1}(x)u_{\B_2}(y)\wedge u_{\B_1}(x)u_{B_2}(y))dxdy \mbox{ as $\frac{u_{B_i}(x)}{u_{\B_i}(x)}=\frac{R_i^d}{r_i^d}\leq\left(\frac{R}{r}\right)^d$}\\
\leq&\left(\frac{R}{r}\right)^d aA
\int(f(x,y) - f(y,x))^2 \left(\nu^x_1(x)\nu^{y}_2(y)\wedge \nu_2^{x}(y) \nu^y_1(x)\right)dxdy.
\end{split}\]
Putting the three steps together, we have
\[\begin{split}
&\left(\E_{\nu_1^x\nu_2^{y}}[f(X,Y)]-\E_{\nu^x_2\nu_1^{y}}[f(X,Y)] \right)^2\\
\leq & 
7\left(\E_{\nu_1^x\nu_2^{y}}[f(X,Y)] - \E_{\nu_1^xu_{\B_2}}[f(X,Y)]\right)^2 + 7\left(\E_{\nu_1^xu_{\B_2}}[f(X,Y)] - \E_{u_{B_1}u_{\B_2}}[f(X,Y)]\right)^2\\
&+7\left(\E_{u_{B_1}u_{\B_2}}[f(X,Y)] - \E_{u_{B_1}u_{B_2}}[f(X,Y)]\right)^2\\
&+ 7\left(\E_{\nu_2^x\nu_1^{y}}[f(X,Y)] - \E_{\nu^x_2u_{\B_1}}[f(X,Y)]\right)^2 +7\left(\E_{\nu^x_2u_{\B_1}}[f(X,Y)] - \E_{u_{B_2}u_{\B_1}}[f(X,Y)]\right)^2\\
&+ 7\left(\E_{u_{B_2}u_{\B_1}}[f(X,Y)] - \E_{u_{B_2}u_{B_1}}[f(X,Y)]\right)^2\\
&+ 7\left(\E_{u_{B_1}u_{B_2}}[f(X,Y)] - \E_{u_{B_2}u_{B_1}}[f(X,Y)]\right)^2\\
\leq&\Xi_x \int \|\nabla_x f(x,y)\|^2(\nu^x_1(x)\nu_2^{y}(y)+\nu^x_2(x)\nu_1^{y}(y))dxdy\\
&+\Xi_y \int \|\nabla_y f(x,y)\|^2(\nu^x_1(x)\nu_2^{y}(y)+\nu^x_2(x)\nu_1^{y}(y))dxdy\\
&+\Xi_e \int(f(x,y) - f(y,x))^2\left(\nu^x_1(x)\nu^{y}_2(y)\wedge \nu_2^{x}(y) \nu^y_1(x)\right)dxdy,
\end{split}\]
where
\[\begin{split}
\Xi_x =28qA,
~ \Xi_y=28Q + 7aA\left(\frac{R^{d+1}}{r^{d-1}}\right)\left(\log\left(\frac{R}{r}\right)\right)^{1_{d=1}},
\mbox{ and }\Xi_e=7\left(\frac{R}{r}\right)^daA .
\end{split}\]
\end{proof}

Under Assumption \ref{aspt:piy}
for each center $m_i$, $i=1,\dots, I$, $\pi^y$ is a $\Ly(R_i,Q,A)$-density with 
$R_i\leq R, R_i/r_i\leq R/r$.
Thus, by setting $\nu_i^x=\nu_i$ and $\nu_i^y = \pi^y$, we have the following corollary of Proposition \ref{prop:swap}.

\begin{cor} \label{cor:mean_difference}
Under Assumptions \ref{aspt:mixture} and \ref{aspt:piy},
\[\begin{split}
&\left(\int f(x,y)\nu_i(x)\pi^y(y) dxdy - \int f(x,y)\nu_j(x)\pi^y(y)dxdy\right)^2\\
\leq& \Xi_x \int \|\nabla_x f(x,y)\|^2(\nu_i(x)\pi^y(y)+\nu_j(x)\pi^y(y))dxdy\\
&+\Xi_y \int \|\nabla_y f(x,y)\|^2(\nu^x_i(x)\pi^y(y)+\nu_j(x)\pi^y(y))dxdy\\
&+\Xi_e\int(f(x,y) - f(y,x))^2\left(\nu^x_i(x)\pi^{y}(y)\wedge \nu_j^{x}(y) \pi^y(x)\right)dxdy.
\end{split}\]
where
\[\begin{split}
\Xi_x =28qA,
~ \Xi_y=28Q + 7aA\left(\frac{R^{d+1}}{r^{d-1}}\right)\left(\log\left(\frac{R}{r}\right)\right)^{1_{d=1}},
\mbox{ and }\Xi_e=7\left(\frac{R}{r}\right)^daA .
\end{split}\]
\end{cor}

\subsection{Proof of Theorem \ref{th:main1}}
Recall that
\[\begin{split} 
&\int (f(x,y)-\bar\theta)^2\pi^x(x)\pi^{y}(y)dxdy\\
=&\sum_{i=1}^{I}p_i\int (f(x,y)-\bar\theta)^2\nu_i(x)\pi^{y}(y)dxdy\\
\leq&3\sum_{i=1}^{I}p_i \underbrace{\int (f(x,y)-\eta_i(y))^2\nu_i(x)\pi^{y}(y)dxdy}_{\mbox{(A)}}
+3\sum_{i=1}^{I}p_i\underbrace{\int (\eta_i(y)-\theta_i)^2\pi^{y}(y)dy}_{\mbox{(B)}}\\
&+3\sum_{i,j}p_ip_j\underbrace{(\theta_i-\theta_j)^2}_{\mbox{(C)}}.
\end{split}\]
\paragraph{For part (A).} By Assumptions \ref{aspt:mixture} and Proposition \ref{prop:bakry_lyap}, we have
\[
\int (f(x,y)-\eta_i(y))^2\nu_i(x)\pi^{y}(y)dxdy
\leq q\int \|\nabla_xf(x,y)\|^2\nu_i(x)\pi^{y}(y)dxdy.
\]
\paragraph{For part (B).} By Assumptions \ref{aspt:piy} and Proposition \ref{prop:bakry_lyap}, we have
\[\begin{split}
&\int (\eta_i(y)-\theta_i)^2\pi^{y}(y)dy\\
\leq& Q \int \|\nabla\eta_i(y)\|^2\pi^{y}(y)dy\\
\leq& \frac{Q}{ \tau} \tau \int \|\nabla_yf(x,y)\|^2\nu_i(x)\pi^{y}(y)dy \mbox{ by Jensen's inequality.}
\end{split}\]
\paragraph{For part (C).} By Corollary \ref{cor:mean_difference}, we have
\[\begin{split}
&\left(\int f(x,y)\nu_i(x)\pi^y(y) dxdy - \int f(x,y)\nu_j(x)\pi^y(y)dxdy\right)^2\\
\leq& \Xi_x \int \|\nabla_x f(x,y)\|^2(\nu_i(x)\pi^y(y)+\nu_j(x)\pi^y(y))dxdy\\
&+\frac{\Xi_y}{\tau} \tau \int \|\nabla_y f(x,y)\|^2(\nu^x_i(x)\pi^y(y)+\nu_j(x)\pi^y(y))dxdy\\
&+\frac{\Xi_e}{\rho}\rho \int(f(x,y) - f(y,x))^2\left(\nu^x_i(x)\pi^{y}(y)\wedge \nu_j^{x}(y) \pi^y(x)\right)dxdy.
\end{split}\]

Putting the bounds for (A), (B), and (C) together, as
\[
\sum_{i,j} p_ip_j (\nu_i(x)\pi^y(y)+\nu_j(x)\pi^y(y)) = 2\pi(x)\pi^y(y)
\]
and
\[\begin{split}
\sum_{i,j} p_ip_j (\nu_i(x)\pi^y(y)\wedge\nu_j(y)\pi^y(x)) &\leq 
\left(\sum_{i, j}p_ip_j\nu_i(x)\pi^y(y)\right)\wedge \left(\sum_{i, j}p_ip_j\nu_j(y)\pi^y(x)\right)\\
&=\pi(x)\pi^y(y) \wedge \pi(y)\pi^y(x),
\end{split}\]
we have
\[\begin{split}
\var_{\pi\pi^y}(f(X,Y))=&\int (f(x,y)-\bar\theta)^2\pi^x(x)\pi^{y}(y)dxdy\\
\leq& 3\left(q+2\Xi_x\right) \int \|\nabla_x f(x,y)\|^2\pi(x)\pi^y(y)dxdy\\
&+3\left(\frac{Q}{ \tau} + 2\frac{\Xi_y}{\tau}\right)\int \tau \|\nabla_y f(x,y)\|^2\pi(x)\pi^y(y)dxdy\\
&+3\frac{\Xi_e}{\rho}\rho\int(f(x,y) - f(y,x))^2\left(\pi(x)\pi^{y}(y)\wedge \pi^{x}(y) \pi^y(x)\right)dxdy\\
\leq &\kappa\E_{\pi\pi^y}[\Gamma_R(f(X,Y))],
\end{split}\]
where
\[\begin{split}
\kappa
=\max\left\{3(56A+1)q, 
~\frac{3}{\tau}\left(57Q + 14aA\left(\frac{R^{d+1}}{r^{d-1}}\right)\left(\log\left(\frac{R}{r}\right)\right)^{1_{d=1}}\right),\frac{7aA}{\rho}\left(\frac{R}{r}\right)^d\right\}.
\end{split}\]

\section{Analysis of Multiple ReLD} \label{sec:mReLD_proof}
We rephrase Theorem \ref{th:main2ez} into a more detailed version as follows:
\begin{theorem} \label{th:main2}
For mReLD defined in \eqref{eq:mReLD}, under Assumptions \ref{aspt:mixture2} and \ref{aspt:bound2},
\[\var_{\pi_{0:K}}(f(\bX_{0:K})) \leq \kappa \E_{\pi_{0:K}}[\Gamma^K_R(f(\bX_{0:K}))],\]
where
\[\begin{split}
\kappa=&\max_{0\leq k\leq K-1}\max\left\{\sum_{h=2}^{k-2}\frac{3(4\alpha)^{k-h+1}}{\tau_k}\left(8\alpha\gamma\Xi_{x_k}+2\gamma\Xi_{y_{k-1}}\right)\right.\\
&\left.+\frac{3}{\tau_{k}}\left((8\alpha\gamma+2\gamma)\Xi_{x_{k}}+2\gamma\Xi_{y_{k-1}}+2q_k\right),
\sum_{h=0}^{k}\frac{3(4\alpha)^{k-h+2}}{\rho}\gamma \Xi_{e_k}\right\},
\end{split}\]
for any $\alpha,\gamma>1$ with $\tfrac{1}{\alpha}+\tfrac{1}{\gamma}=1$,
and
\[\begin{split}
\Xi_{x_k}& =28q_ka_{k+1},\\
\Xi_{y_k}&=28q_{k+1} 
+ 7\frac{\left(r_{k+1}\right)^{d+1}}{\left(r_k\right)^{d-1}}a_ka_{k+1}\left(\log\left(\frac{r_{k+1}}{r_k}\right)\right)^{1_{d=1}},\\
&\Xi_{y_{(-1)}}=0,\quad\Xi_{e_k}=7\left(\frac{r_{k+1}}{r_k}\right)^da_ka_{k+1}.
\end{split}\]
When $I=2$, we can further refine $\kappa$ to
\[\begin{split}
\kappa=&\max_{0\leq k\leq K-1}\max\left\{\sum_{h=2}^{k-2}\frac{3\alpha^{k-h+1}}{\tau_k}\left(2\alpha\gamma\Xi_{x_k}+2\gamma\Xi_{y_{k-1}}\right)\right.\\
&\left.+\frac{3}{\tau_{k}}\left((2\alpha\gamma+2\gamma)\Xi_{x_{k}}+2\gamma\Xi_{y_{k-1}}+2q_k\right),
\sum_{h=0}^{k}\frac{3\alpha^{k-h+2}}{\rho}\gamma \Xi_{e_k}\right\}.
\end{split}\]
\end{theorem}

The proof of Theorem \ref{th:main2} builds on the analysis of ReLD and an induction argument. 
We provide a roadmap of our proving strategy first.

\subsection{A roadmap}
\label{sec:mReLDmap}
We introduce the following notations 
\[\begin{split}
&\E_{r:h}[f(\bX_{0:K})] =\int f(\bX_{0:r-1},\by_{r:h},\bX_{h+1:K})\pi_{r:h}(\by_{r:h})d\by_{r:h},\\
&\E_k [f(\bX_{0:K})]  = \int f(\bX_{0:k-1},y,\bX_{k+1:K})\pi_k(y)dx, 
\end{split}\] 
and we write $\E_{(K+1):K} [f(\bX_{0:K})] = f(\bX_{0:K})$ for convenience. 
We also write
\[
\var_{0:K}(f(\bX_{0:K}))=\E_{0:K}\left[\left(f(\bX_{0:K})-\E_{0:K} f(\bX_{0:K})\right)^2\right]
\]
We first note the following decomposition
\[
f(\bX_{0:K}) - \E f(\bX_{0:K}) = \sum_{k=0}^K \left(\E_{(k+1):K} f(\bX_{0:K}) - \E_{k:K} f(\bX_{0:K})\right).
\]
Note that for $j<k$,
\begin{align*}
&\E_{0:K} \left[\left(\E_{(j+1):K} f(\bX_{0:K}) - \E_{j:K} f(\bX_{0:K})\right)\left(\E_{(k+1):K} f(\bX_{0:K}) - \E_{k:K} f(\bX_{0:K})\right)\right]\\
=&\E_{0:K} \left[\E_{k:K}\left[ \left(\E_{(j+1):K} f(\bX_{0:K}) - \E_{j:K} f(\bX_{0:K})\right)\left(\E_{(k+1):K} f(\bX_{0:K}) - \E_{k:K} f(\bX_{0:K})\right)\right]\right]\\
=&\E_{0:K} \left[\left(\E_{(k+1):K} f(\bX_{0:K}) - \E_{k:K} f(\bX_{0:K})\right)\E_{k:K}\left[ \left(\E_{(j+1):K} f(\bX_{0:K}) - \E_{j:K} f(\bX_{0:K})\right)\right]\right]=0.
\end{align*}
Thus, we have the following variance decomposition 
\[\begin{split}
&\E_{0:K}\left[\left(f(\bX_{0:K}) - \E f(\bX_{0:K})\right)^2\right]\\
=& \sum_{k=0}^K \E_{0:K}\left[\left(\E_{(k+1):K} f(\bX_{0:K}) - \E_{r:K} f(\bX_{0:K})\right)^2\right]\\
=& \sum_{k=0}^{K} \E_{0:K}\left[\E_k\left[\left(\E_{(k+1):K} f(\bX_{0:K})- \E_{k:K} f(\bX_{0:K})\right)^2\right]\right]
\end{split}.\]
The above decomposition allows us to focus on 
\[
\E_k\left[\left(\E_{(k+1):K} f(\bX_{0:K})- \E_{k:K} f(\bX_{0:K})\right)^2\right]
\]
individually. Let
\[
\bW_k = \bX_{0:(k-1)}, ~ Y_k = X_{k+1}, ~ \bZ_k=\bX_{(k+2):K}.
\]
We also write $\pi_k^{\bz} = \pi_{(k+2):K}$.
For a fixed $\bW_k=\bw_k$, we define
\[\begin{split}
&g_k(\bfw_k,x_k, y_k)=\int f(\bw_k, x_k, y_k, \bz_k)\pi_k^{\bz}(\bz_k)d\bz_k, \\
&\eta_{k,i}(\bfw_k, y_k)= \int g_k(\bfw_k, x_k, y_k)\nu_{k,i}(x_k)dx_k,\\
&\theta_{k,i}(\bfw_k) = \int \eta_{k,i}(\bfw_k,y_k) \pi_{k+1}(y_k)dy_k, \mbox{ and }
\bar \theta_k(\bfw_k) = \sum_{i=1}^{I}p_{i}\theta_{k,i}(\bfw_k).
\end{split}\]
Note that with this notation
\[
\E_{k+1:K}f(\bX_{0:K})=\int g_k(\bW_k,X_k,y_k)\pi_{k+1}(y_k)dy_k,
\]
\[
\E_{k:K}f(\bX_{0:K})=\bar{\theta}_k(\bW_k). 
\]

Following similar lines of argument as \eqref{eq:main_decomp}, we have
\[\begin{split}
&\int \left(\int g_k(\bfw_k, x_k, y_k) \pi_{k+1}(y_k)dy_k - \bar \theta_k\right)^2 \pi_k(x_k)dx_k\\
\leq&\int \left(g_k(\bfw_k, x_k, y_k) - \bar \theta_k\right)^2 \pi_k(x_k)\pi_{k+1}(y_k)dx_kdy_k \\
\leq&3\sum_{i=1}^{I}p_{i}\underbrace{\int \left( g_k(\bfw_k, x_k, y_k)- \eta_{k,i}(\bfw_k, y_k)\right)^2\nu_{k,i}(x_k)dx_k\pi_{k+1}(y_k)dy_k}_{\mbox{(A)}}\\
&+3\sum_{i=1}^{I}p_{i} \underbrace{\int (\eta_{k,i}(\bfw_k, y_k) -  \theta_{k,i}(\bfw_k))^2 \pi_{k+1}(y_k)dy_k}_{\mbox{(B)}}\\
&+3\sum_{ij}p_{i}p_{j} \underbrace{( \theta_{k,i}(\bfw_k) - \theta_{k,j}(\bfw_k))^2}_{\mbox{(C)}}.
\end{split}\]
We note that part (A) and (B) are variance of functions in a single mixture component, so they are 
 easy to control using Proposition \ref{prop:bakry_lyap}.
Part (C), the mean difference, requires some further development, which we lay out in the next subsection.

\subsection{Mean difference estimates}
Define
\[
\Ex_k(\bx_{0:K}) = (f(\bw_k, x_k, y_k, \bz_k)-f(\bw_k, y_k, x_k, \bz_k))^2s_k(x_k, y_k)
\]
and 
\[
\Gamma_k(\bx_{0:K}) = \sum_{l=k}^{K}\left(\tau_l\|\nabla_{x_l} f(\bx_{0:K})\|^2 +  \rho\Ex_l(\bx_{0:K})\right).
\]
We also define
for $k=0,1,\dots, K-1$, let 
\[\begin{split}
\Xi_{x_k}& =28q_k a_{k+1},\\
\Xi_{y_k}&=28q_k 
+ 7\frac{\left(r_{k+1} \right)^{d+1}}{\left(r_k\right)^{d-1}}a_ka_{k+1}\left(\log\left(\frac{r_{k+1}}{r_k}\right)\right)^{1_{d=1}},\\
\Xi_{e_k}&=7\left(\frac{r_{k+1}}{r_k}\right)^da_ka_{k+1}.
\end{split}\]

\begin{proposition} \label{prop:mix2}
Under Assumptions \ref{aspt:mixture2} and \ref{aspt:bound2}, for $k\leq K-1$, 
\[
\sum_{i,j}p_{i}p_{j} \left(\theta_{k,i}(\bw_k) - \theta_{k,j}(\bw_k)\right)^2 \leq \Xi_k \E_{k:K} \left[\Gamma_k(\bw_k,\bX_{k:K})\right],
\]
where for any fixed $\alpha,\gamma>1$ with $\tfrac{1}{\alpha}+\tfrac{1}{\gamma}=1$, when $I=2$,
\[
\Xi_k=\max\left\{ \max_{k+1\leq l\leq K-1} \alpha^{l-k-1}\left(\frac{2\alpha\gamma\Xi_{x_l}}{\tau_{l}}+\frac{2\gamma\Xi_{y_{l-1}}}{\tau_l}\right), ~
\frac{2\gamma \Xi_{x_{k}}}{\tau_{k}}, ~ \max_{k\leq l\leq K-1} \alpha^{l-k} \frac{\gamma \Xi_{e_l}}{\rho}\right\}.
\] 
when $I>2$.
\[
\Xi_k=\max\left\{ \max_{k+1\leq l\leq K-1} (4\alpha)^{l-k-1}\left(\frac{8\alpha\gamma\Xi_{x_l}}{\tau_{l}}+\frac{2\gamma\Xi_{y_{l-1}}}{\tau_l}\right), ~
\frac{2\gamma \Xi_{x_{k}}}{\tau_{k}}, ~ \max_{k\leq l\leq K-1} (4\alpha)^{l-k} \frac{\gamma \Xi_{e_l}}{\rho}\right\}.
\] 
\end{proposition}
\begin{proof}
We prove the proposition by induction. 
\textbf{When $I=2$}, for any fixed $\bfw_k$, we want to show that
\begin{equation}\label{eq:induct2}
\begin{split}
&\sum_{i,j}p_{i}p_{j} \left(\theta_{k,i}(\bw_k) - \theta_{k,j}(\bw_k)\right)^2\\
=&2p_1p_2\left(\theta_{k,1}(\bw_k) - \theta_{k,2}(\bw_k)\right)^2 \\
\leq&\sum_{l=k+1}^{K-1} \alpha^{l-k-1}\left(\frac{2\alpha\gamma\Xi_{x_l}}{\tau_{l}}+\frac{2\gamma\Xi_{y_{l-1}}}{\tau_l}\right)\int \tau_l \|\nabla_{x_l}f(\bw_k, x_k, y_k, \bz_{k})\|^2\pi_{k}(x_k)\pi_{k+1}(y_{k})\pi_{k}^{\bz}(\bz_k)dx_kdy_{k}d\bz_k\\
&+ \frac{2\gamma\Xi_{x_{k}}}{\tau_{k}}\int \tau_{k} \|\nabla_{x_k}f(\bw_k, x_k, y_k, \bz_{k})\|^2\pi_{k}(x_k)\pi_{k+1}(y_{k})\pi_{k}^{\bz}(\bz_k)dx_kdy_{k}d\bz_k\\
&+\sum_{l=k}^{K-1} \alpha^{l-k} \frac{\gamma \Xi_{e_l}}{\rho} \rho \int \Ex_l(\bw_k, x_k, y_k, \bz_{k})\pi_{k}(x_k)\pi_{k+1}(y_{k})\pi_{k}^{\bz}(\bz_k)dx_kdy_{k}d\bz_k.
\end{split}\end{equation}
\noindent{\bf For $k=K-1$,} by Corollary \ref{cor:mean_difference}, with $\pi_k=\pi,\pi_{k+1}=\pi^y$, we have
\[\begin{split}
&2p_{1}p_{2} \left( \theta_{k,1}(\bw_k) - \theta_{k,2}(\bw_k)\right)^2 \\
\leq& 2\frac{\Xi_{x_k}}{\tau_k}\tau_k \int \|\nabla_{x_k} f(\bw_k,x_{k},y_k)\pi_{k}(x_{k})\|^2\pi_{k}(x_{k})\pi_{k+1}(y_k)dx_{k}dy_k\\
&+ 2\frac{\Xi_{y_k}}{\tau_{k+1}}\tau_{k+1} \int \|\nabla_{y_k} f(\bw_k,x_{k},y_k)\pi_{k}(x_{k})\|^2\pi_{k}(x_{k})\pi_{k+1}(y_k)dx_{k}dy_k\\
&+\frac{\Xi_{e_k}}{\rho}\rho\int \left(f(\bw_k,x_{k},y_k)-f(\bw_k,y_{k},x_k)\right)^2
\left(\pi_{k}(x_{k})\pi_{k+1}(y_k)\wedge \pi_{k}(y_{k})\pi_{k+1}(x_k)\right)dx_{k}dy_k
\end{split}\]
\noindent{\bf Suppose \eqref{eq:induct2} holds  when $k$ is replaced by $k+1$. Now, for $k$,} let
\[ \begin{split}
\zeta_{kij}(\bw_k)&=\int f(\bw_k, x_k, y_{k}, \bz_k)\nu_{k,i}(x_k)\nu_{k+1,j}(y_{k})\pi_{k}^{\bz}(\bz_k)dx_kdy_{k}d\bz_k\\
&=\int \theta_{k+1,i}(\bfw_k,x_k)\nu_{k,i}(x_k)dx_k.
\end{split}\]
Then, 
\begin{align}
\notag
&2p_1p_2({\theta}_{k,1}-{\theta}_{k,2})^2\\
\notag
=&2p_{1}p_2 [({\theta}_{k,1}-\zeta_{k12})+
(\zeta_{k12}-\zeta_{k21})+(\zeta_{k21}-{\theta}_{k,2})]^2\\
\notag
=&2p_{1}p_2 [p_1(\zeta_{k11}-\zeta_{k12})+
(\zeta_{k12}-\zeta_{k21})+p_2(\zeta_{k21}-\zeta_{k22})]^2\\
\notag
\leq& 2p_{1}p_2 [\alpha p_1(\zeta_{k11}-\zeta_{k12})^2+
\gamma (\zeta_{k21}-\zeta_{k12})^2+\alpha p_2(\zeta_{k22}-\zeta_{k21})^2][\frac{1}{\alpha}p_1+\frac{1}{\gamma}+\frac{1}{\alpha}p_2]\\
\label{tmp:dec2}
\leq& \underbrace{\alpha 2 p_1p_2(p_1(\zeta_{k11}-\zeta_{k12})^2+p_2(\zeta_{k21}-\zeta_{k22})^2)}_{\mbox{(a)}}+\underbrace{2p_1p_2\gamma (\zeta_{k21}-\zeta_{k12})^2}_{\mbox{(b)}}.
\end{align}

For part (a), by induction, 
\[\begin{split}
&\alpha 2 p_1p_2(p_1(\zeta_{k11}(\bfw_k)-\zeta_{k12}(\bfw_k))^2+p_2(\zeta_{k21}(\bfw_k)-\zeta_{k22}(\bfw_k))^2)\\
\leq& \alpha\sum_{i=1}^2 p_i  \int 2p_1p_2(\bar{\theta}_{k+1,j}(\bfw_k,x_k)-\bar{\theta}_{k+1,h}(\bfw_k,x_k))^2\nu_{k,i}(x_k)dx_k \\
\leq&\sum_{l=k+2}^{K-1} \alpha^{l-k-1}\left(\frac{2\alpha\gamma\Xi_{x_l}}{\tau_{l}}+\frac{2\gamma\Xi_{y_{l-1}}}{\tau_l}\right)
\int \tau_l \|\nabla_{x_l}f(\bw_k, x_k, y_k, \bz_{k})\|^2\pi_{k}(x_k)\pi_{k+1}(y_{k})\pi_{k}^{\bz}(\bz_k)dx_kdy_{k}d\bz_k\\
&+ \frac{2\alpha \gamma\Xi_{x_{k+1}}}{\tau_{k+1}}\int \tau_{k+1} \|\nabla_{y_k}f(\bw_k, x_k, y_k, \bz_{k})\|^2\pi_{k}(x_k)\pi_{k+1}(y_{k})\pi_{k}^{\bz}(\bz_k)dx_kdy_{k}d\bz_k\\
&+\sum_{l=k+1}^{K-1} \alpha^{l-k}\frac{\gamma \Xi_{e_l}}{\rho} \rho \int \Ex_l(\bw_k, x_k, y_k, \bz_{k})\pi_{k}(x_k)\pi_{k+1}(y_{k})\pi_{k}^{\bz}(\bz_k)dx_kdy_{k}d\bz_k.
\end{split}\]
For part (b), 
by the mean difference estimate in Proposition \ref{prop:swap}, we have
\[\begin{split}
&(\zeta_{k12}(\bw_k)-\zeta_{k21}(\bw_k))^2\\ 
\leq &\frac{\Xi_{x_k}}{\tau_k} \tau_k \int \|\nabla_{x_k}f(\bw_k, x_k, y_{k}, \bz_k)\|^2 (\nu_{k,1}\nu_{k+1,2}+\nu_{k,2}\nu_{k+1,1})(x_k,y_{k})\pi_{k}^{\bz}(\bz_k)dx_kdy_{k}d\bz_k\\
&+\frac{\Xi_{y_k}}{\tau_{k+1}} \tau_{k+1} \int \|\nabla_{y_k}f(\bw_k, x_k, y_{k}, \bz_k)\|^2 (\nu_{k,1}\nu_{k+1,2}+\nu_{k,2}\nu_{k+1,1})(x_k,y_{k})\pi_{k}^{\bz}(\bz_k)dx_kdy_{k}d\bz_k\\
&+\frac{\Xi_{e_k}}{\rho} \rho \int (f(\bw_k, x_k, y_{k}, \bz_k)-f(\bw_k, y_k, x_k, \bz_k))^2\cdot\\
&\quad\quad\quad \left(\nu_{k,1}(x_k)\nu_{k+1,2}(y_{k})\wedge \nu_{k,2}(y_k)\nu_{k+1,1}(x_{k})\right)\pi_{k}^{\bz}(\bz_k)dx_kdy_{k}d\bz_k.
\end{split}\]
Note that as
\[
2p_1p_2 (\nu_{k,1}(x_k)\nu_{k+1,2}(y_k)+\nu_{k,2}(x_k)\nu_{k+1,1}(y_k))\leq 2\pi_k(x_k)\pi_{k+1}(y_k),
\]
and because $a\wedge c+b\wedge d\leq (a+b)\wedge (c+d)$
\[\begin{split}
&2p_1p_2(\nu_{k,1}(x_k)\nu_{k+1,2}(y_k))\wedge(\nu_{k,2}(y_k)\nu_{k+1,1}(x_k))\\
=&(p_1\nu_{k,1}(x_k)p_2\nu_{k+1,2}(y_k))\wedge(p_2\nu_{k,2}(y_k)p_1\nu_{k+1,1}(x_k))\\
&+(p_2\nu_{k,2}(y_k)p_1\nu_{k+1,1}(x_k))\wedge(p_1\nu_{k,1}(x_k)p_2\nu_{k+1,2}(y_k))\\
\leq& (\pi_k(x_k)\pi_{k+1}(y_k))\wedge (\pi_k(y_k)\pi_{k+1}(x_k)),
\end{split}\]
we have
\[\begin{split}
&2p_1p_2 (\zeta_{k12}-\zeta_{k21})^2\\ 
\leq &2\frac{\Xi_{x_k}}{\tau_k} \tau_k \int \|\nabla_{x_k}f(\bw_k, x_k, y_{k}, \bz_k)\|^2 \pi_{k}(x_k)\pi_{k+1}(y_{k})\pi_{k}^{\bz}(\bz_k)dx_kdy_{k}d\bz_k\\
&+2\frac{\Xi_{y_k}}{\tau_{k+1}} \tau_{k+1} \int \|\nabla_{y_k}f(\bw_k, x_k, y_{k}, \bz_k)\|^2 \pi_{k}(x_k)\pi_{k+1}(y_{k})\pi_{k}^{\bz}(\bz_k)dx_kdy_{k}d\bz_k\\
&+\frac{\Xi_{e_k}}{\rho} \rho \int (f(\bw_k, x_k, y_{k}, \bz_k)-f(\bw_k, y_k, x_k, \bz_k))^2\cdot\\
&\quad\quad\quad \left(\pi_{k}(x_k)\pi_{k+1}(y_{k})\wedge \pi_{k}(y_k)\pi_{k+1}(x_{k})\right)\pi_{k}^{\bz}(\bz_k)dx_kdy_{k}d\bz_k.
\end{split}\]
Combining parts (a) and (b) we have
\[
\begin{split}
&2p_1p_2\left(\bar \theta_{k,1} - \bar \theta_{k,2}\right)^2 \\
\leq&\sum_{l=k+1}^{K-1} \alpha^{l-k-1}\left(\frac{2\alpha\gamma\Xi_{x_l}}{\tau_{l}}+\frac{2\gamma\Xi_{y_{l-1}}}{\tau_l}\right)\int \tau_l \|\nabla_{x_l}f(\bw_k, x_k, y_k, \bz_{k})\|^2\pi_{k}(x_k)\pi_{k+1}(y_{k})\pi_{k}^{\bz}(\bz_k)dx_kdy_{k}d\bz_k\\
&+ \frac{2\gamma\Xi_{x_{k}}}{\tau_{k}}\int \tau_{k} \|\nabla_{x_k}f(\bw_k, x_k, y_k, \bz_{k})\|^2\pi_{k}(x_k)\pi_{k+1}(y_{k})\pi_{k}^{\bz}(\bz_k)dx_kdy_{k}d\bz_k\\
&+\sum_{l=k}^{K-1} \alpha^{l-k} \frac{\gamma \Xi_{e_l}}{\rho} \rho \int \Ex_l(\bw_k, x_k, y_k, \bz_{k})\pi_{k}(x_k)\pi_{k+1}(y_{k})\pi_{k}^{\bz}(\bz_k)dx_kdy_{k}d\bz_k.
\end{split}
\]
Setting 
\[
\Xi_k=\gamma \max\left\{ \max_{k+1\leq l\leq K-1} \alpha^{l-k}\left(\frac{2\Xi_{x_l}}{\tau_{l}}+\frac{2\Xi_{y_{l-1}}}{\tau_l}\right), ~
\frac{2\Xi_{x_{k}}}{\tau_{k}}, ~ \max_{k\leq l\leq K-1} \alpha^{l-k-1} \frac{\Xi_{e_l}}{\rho}\right\},
\] 
we have the result.\\

\noindent{\bf When $I>2$,} using a similar induction argument, we want to show that for any fixed $\bfw_k$,
\begin{equation}\label{eq:induct}
\begin{split}
&\sum_{i,j}p_{i}p_{j}\left(\theta_{k,i}(\bw_k) - \theta_{k,j}(\bw_k)\right)^2 \\
\leq&\sum_{l=k+1}^{K-1} (4\alpha)^{l-k-1}\left(\frac{8\alpha\gamma\Xi_{x_l}}{\tau_{l}}+\frac{2\gamma\Xi_{y_{l-1}}}{\tau_l}\right)\cdot\\
&\int \tau_l \|\nabla_{x_l}f(\bw_k, x_k, y_k, \bz_{k})\|^2\pi_{k}(x_k)\pi_{k+1}(y_{k})\pi_{k}^{\bz}(\bz_k)dx_kdy_{k}d\bz_k\\
&+ \frac{2\gamma\Xi_{x_{k}}}{\tau_{k}}\int \tau_{k} \|\nabla_{x_k}f(\bw_k, x_k, y_k, \bz_{k})\|^2\pi_{k}(x_k)\pi_{k+1}(y_{k})\pi_{k}^{\bz}(\bz_k)dx_kdy_{k}d\bz_k\\
&+\sum_{l=k}^{K-1} (4\alpha)^{l-k} \frac{\gamma Xi_{e_l}}{\rho} \rho \int \Ex_l(\bw_k, x_k, y_k, \bz_{k})\pi_{k}(x_k)\pi_{k+1}(y_{k})\pi_{k}^{\bz}(\bz_k)dx_kdy_{k}d\bz_k.
\end{split}\end{equation}
\noindent{\bf For $k=K-1$,} by Corollary \ref{cor:mean_difference}, with $\pi=\pi_{K-1},\pi^y=\pi_K$, we have
\[\begin{split}
&\sum_{i,j}p_{i}p_{j} \left(\theta_{k,i}(\bw_k) - \theta_{k,j}(\bw_k)\right)^2 \\
\leq& 2\frac{\Xi_{x_k}}{\tau_k}\tau_k \int \|\nabla_{x_k} f(\bw_k,x_{k},y_k)\pi_{k}(x_{k})\|^2\pi_{k}(x_{k})\pi_{k+1}(y_k)dx_{k}dy_k\\
&+ 2\frac{\Xi_{y_k}}{\tau_{k+1}}\tau_{k+1} \int \|\nabla_{y_k} f(\bw_k,x_{k},y_k)\pi_{k}(x_{k})\|^2\pi_{k}(x_{k})\pi_{k+1}(y_k)dx_{k}dy_k\\
&+\frac{\Xi_{e_k}}{\rho}\rho\int \left(f(\bw_k,x_{k},y_k)-f(\bw_k,y_{k},x_k)\right)^2
\left(\pi_{k}(x_{k})\pi_{k+1}(y_k)\wedge \pi_{k}(y_{k})\pi_{k+1}(x_k)\right)dx_{k}dy_k.
\end{split}\]
\noindent{\bf Suppose \eqref{eq:induct} holds  when $k$ is replaced by $k+1$. Now, for $k$,} 
We first note that $\theta_{k,i}=\sum_{h=1}^{I}p_{k+1,h}\zeta_{kih}$ and
\begin{equation}
\label{tmp:decI}
\begin{split}
&\sum_{i,j}p_{i}p_{j}(\theta_{k,i} - \theta_{k,j})^2\\ 
=&\sum_{i,j}p_{i}p_{j} \left((\theta_{k,i}-\zeta_{kij}) + (\zeta_{kij}-\zeta_{kji})+(\zeta_{kji}-\theta_{k,j})\right)^2\\
\leq& 2\alpha\sum_{i,j}p_{i}p_{j}\underbrace{(\theta_{k,i}-\zeta_{kij})^2}_{\mbox{(a)}} + \gamma\sum_{i,j}p_{i}p_{j}\underbrace{(\zeta_{kij}-\zeta_{kji})^2}_{\mbox{(b)}} + 2\alpha\sum_{i,j}p_{i}p_{j}\underbrace{(\zeta_{kji}-\theta_{k,j})^2}_{\mbox{(c)}}.
\end{split}
\end{equation}
Note that part (a) and (c) are symmetric. For part (a), we have
\[\begin{split}
\sum_{i,j}p_{i}p_{j}(\zeta_{kij}-\theta_{k,i})^2=\sum_{i,j}p_{i}p_{j}\left(\sum_{h=1}^I p_{h} (\zeta_{kij}-\zeta_{kih})\right)^2
\leq \sum_{i, j,h} p_i p_{j} p_{h}\left(\zeta_{kij}-\zeta_{kih}\right)^2.
\end{split}
\]
By induction, 
\[\begin{split}
&\sum_i p_i \sum_{j,h}p_{j}p_{h} \left(\zeta_{kij}(\bw_k)-\zeta_{kih}(\bw_k)\right)^2\\
\leq&\sum_i p_i \sum_{j,h}p_{j}p_{h} \int (\theta_{k+1,j}(\bfw_k,x_k)-\theta_{k+1,h}(\bfw_k,x_k))^2\nu_{k,i}(x_k)dx_k \\
\leq&\sum_{l=k+2}^{K-1} (4\alpha)^{l-k-2}\left(\frac{8\alpha\gamma \Xi_{x_l}}{\tau_{l}}+\frac{2\gamma\Xi_{y_{l-1}}}{\tau_l}\right)\cdot\\
&\int \tau_l \|\nabla_{x_l}f(\bw_k, x_k, y_k, \bz_{k})\|^2\pi_{k}(x_k)\pi_{k+1}(y_{k})\pi_{k}^{\bz}(\bz_k)dx_kdy_{k}d\bz_k\\
&+ \frac{2\gamma\Xi_{x_{k+1}}}{\tau_{k+1}}\int \tau_{k+1} \|\nabla_{y_k}f(\bw_k, x_k, y_k, \bz_{k})\|^2\pi_{k}(x_k)\pi_{k+1}(y_{k})\pi_{k}^{\bz}(\bz_k)dx_kdy_{k}d\bz_k\\
&+\sum_{l=k+1}^{K-1} (4\alpha)^{l-k-1}\frac{\gamma \Xi_{e_l}}{\rho} \rho \int \Ex_l(\bw_k, x_k, y_k, \bz_{k})\pi_{k}(x_k)\pi_{k+1}(y_{k})\pi_{k}^{\bz}(\bz_k)dx_kdy_{k}d\bz_k.
\end{split}\]
For part (b), recall that 
\[ 
\zeta_{kij}=\int f(\bw_k, x_k, y_{k}, \bz_k)\nu_{k,i}(x_k)\nu_{k+1,j}(y_{k})\pi_{k}^{\bz}(\bz_k)dx_kdy_{k}d\bz_k
\]
and
\[ 
\zeta_{kji}=\int f(\bw_k, x_k, y_{k}, \bz_k)\nu_{k,j}(x_k)\nu_{k+1,i}(y_{k})\pi_{k}^{\bz}(\bz_k)dx_kdy_{k}d\bz_k.
\] 
By the mean difference estimate in Proposition \ref{prop:swap}, we have
\[\begin{split}
&(\zeta_{kij}(\bw_k)-\zeta_{kji}(\bw_k))^2\\ 
\leq &\frac{\Xi_{x_k}}{\tau_k} \tau_k \int \|\nabla_{x_k}f(\bw_k, x_k, y_{k}, \bz_k)\|^2 (\nu_{k,i}\nu_{k+1,j}+\nu_{k,j}\nu_{k+1,i})(x_k,y_{k})\pi_{k}^{\bz}(\bz_k)dx_kdy_{k}d\bz_k\\
&+\frac{\Xi_{y_k}}{\tau_{k+1}} \tau_{k+1} \int \|\nabla_{y_k}f(\bw_k, x_k, y_{k}, \bz_k)\|^2 (\nu_{k,i}\nu_{k+1,j}+\nu_{k,j}\nu_{k+1,i})(x_k,y_{k})\pi_{k}^{\bz}(\bz_k)dx_kdy_{k}d\bz_k\\
&+\frac{\Xi_{e_k}}{\rho} \rho \int (f(\bw_k, x_k, y_{k}, \bz_k)-f(\bw_k, y_k, x_k, \bz_k))^2\cdot\\
&\quad\quad\quad \left(\nu_{k,i}(x_k)\nu_{k+1,j}(y_{k})\wedge \nu_{k,j}(y_k)\nu_{k+1,i}(x_{k})\right)\pi_{k}^{\bz}(\bz_k)dx_kdy_{k}d\bz_k.
\end{split}\]
Note that 
\[
\sum_{i,j} p_ip_j\nu_{k,i}(x_k)\nu_{k+1,j}(y_k)=\pi_k(x_k)\pi_{k+1}(y_k),
\]
and because $a\wedge c+b\wedge d\leq (a+b)\wedge (c+d)$, 
\[
\sum_{i,j} (p_ip_j\nu_{k,i}(x_k)\nu_{k+1,j}(y_k))\wedge(p_ip_j\nu_{k,j}(y_k)\nu_{k+1,i}(x_k))\leq (\pi_k(x_k)\pi_{k+1}(y_k))\wedge (\pi_k(y_k)\pi_{k+1}(x_k)).
\]
Thus,
\[\begin{split}
&\sum_{i,j}p_{i}p_{j} (\zeta_{kij}(\bw_k)-\zeta_{kji}(\bw_k))^2\\ 
\leq &2\frac{\Xi_{x_k}}{\tau_k} \tau_k \int \|\nabla_{x_k}f(\bw_k, x_k, y_{k}, \bz_k)\|^2 \pi_{k}(x_k)\pi_{k+1}(y_{k})\pi_{k}^{\bz}(\bz_k)dx_kdy_{k}d\bz_k\\
&+2\frac{\Xi_{y_k}}{\tau_{k+1}} \tau_{k+1} \int \|\nabla_{y_k}f(\bw_k, x_k, y_{k}, \bz_k)\|^2 \pi_{k}(x_k)\pi_{k+1}(y_{k})\pi_{k}^{\bz}(\bz_k)dx_kdy_{k}d\bz_k\\
&+\frac{\Xi_{e_k}}{\rho} \rho \int (f(\bw_k, x_k, y_{k}, \bz_k)-f(\bw_k, y_k, x_k, \bz_k))^2\\
&\quad\quad\quad \left(\pi_{k}(x_k)\pi_{k+1}(y_{k})\wedge \pi_{k}(y_k)\pi_{k+1}(x_{k})\right)\pi_{k}^{\bz}(\bz_k)dx_kdy_{k}d\bz_k.
\end{split}\]
Combining parts (a), (b), and (c), we have
\[\begin{split}
&\sum_{i,j}p_{i}p_{j}\left(\theta_{i,k}(\bw_k) - \theta_{j, k}(\bw_k)\right)^2 \\
\leq&\sum_{l=k+1}^{K-1} (4\alpha)^{l-k-1}\left(\frac{8\alpha\gamma\Xi_{x_l}}{\tau_{l}}+\frac{2\gamma\Xi_{y_{l-1}}}{\tau_l}\right)\cdot\\
&\int \tau_l \|\nabla_{x_l}f(\bw_k, x_k, y_k, \bz_{k})\|^2\pi_{k}(x_k)\pi_{k+1}(y_{k})\pi_{k}^{\bz}(\bz_k)dx_kdy_{k}d\bz_k\\
&+ \frac{2\gamma\Xi_{x_{k}}}{\tau_{k}}\int \tau_{k} \|\nabla_{x_k}f(\bw_k, x_k, y_k, \bz_{k})\|^2\pi_{k}(x_k)\pi_{k+1}(y_{k})\pi_{k}^{\bz}(\bz_k)dx_kdy_{k}d\bz_k\\
&+\sum_{l=k}^{K-1} (4\alpha)^{l-k} \frac{\gamma\Xi_{e_l}}{\rho} \rho \int \Ex_l(\bw_k, x_k, y_k, \bz_{k})\pi_{k}(x_k)\pi_{k+1}(y_{k})\pi_{k}^{\bz}(\bz_k)dx_kdy_{k}d\bz_k.
\end{split}\]
Setting 
\[
\Xi_k=\max\left\{ \max_{k+1\leq l\leq K-1} (4\alpha)^{l-k-1}\left(\frac{8\alpha\gamma\Xi_{x_l}}{\tau_{l}}+\frac{2\gamma\Xi_{y_{l-1}}}{\tau_l}\right), ~
\frac{2\gamma \Xi_{x_{k}}}{\tau_{k}}, ~ \max_{k\leq l\leq K-1} (4\alpha)^{l-k} \frac{\gamma \Xi_{e_l}}{\rho}\right\},
\] 
we have the result.
\end{proof}

\subsection{Proof of Theorem \ref{th:main2}}
Recall that 
\[\begin{split}
&\E_k\left[\left(\E_{(k+1):K} f(\bw_k, \bX_{k:K})- \E_{k:K} f(\bw_k, \bX_{k:K})\right)^2\right]\\
\leq&3\sum_{i=1}^{I}p_{i}\underbrace{\int \left( g_k(\bw_k, x_k, y_k)- \eta_{k,i}(\bw_k, y_k)\right)^2\nu_{k,i}(x_k)dx_k\pi_{k+1}(y_k)dy_k}_{\mbox{(A)}}\\
&+3\sum_{i=1}^{I}p_{i} \underbrace{\int (\eta_{k,i}(\bw_k, y_k) -\theta_{k,i}(\bw_k))^2 \pi_{k+1}(y_k)dy_k}_{\mbox{(B)}}\\
&+3\sum_{ij}p_{i}p_{j} \underbrace{(\theta_{k,i}(\bw_k) - \theta_{k,j}(\bw_k))^2}_{\mbox{(C)}}.
\end{split}\]

\noindent{\bf For part (A),} 
\[\begin{split}
&\sum_{i=1}^{I}p_{i}\int \left( g_k(\bw_k, x_k, y_k)- \eta_{k,i}(\bw_k, y_k)\right)^2\nu_{k,i}(x_k)dx_k\pi_{k+1}(y_k)dy_k\\
\leq& \sum_{i=1}^{I}p_{i} q_k\int \|\nabla_{x_k} g_k(\bw_k, x_k,y_k)\|^2\nu_{k,i}(x_k)\pi_{k+1}(y_k)dx_kdy_k\\
\leq& q_k \int \|\nabla_{x_k} g_k(\bw_k, x_k,y_k)\|^2\pi_{k}(x_k)\pi_{k+1}(y_k)dx_kdy_k.
\end{split}\]
\noindent{\bf For part (B),} 
\[\begin{split}
&\sum_{i=1}^{I}p_{i}\int (\eta_{k,i}(\bw_k, y_k) - \theta_{k,i}(\bw_k))^2 \pi_{k+1}(y_k)dy_k\\
\leq& \sum_{i=1}^{I}p_{i} q_{k+1}\int \|\nabla_{y_k} g_k(\bw_k, x_k,y_k)\|^2\nu_{k,i}(x_k)\pi_{k+1}(y_k)dx_kdy_k\\
\leq& q_{k+1} \int \|\nabla_{y_k} g_k(\bw_k, x_k,y_k)\|^2\pi_{k}(x_k)\pi_{k+1}(y_k)dx_kdy_k.
\end{split}\]
\noindent{\bf For part (C),} from Proposition \ref{prop:mix2}, 
\[\begin{split}
&\sum_{i,j}p_{i}p_{j}\left(\theta_{i,k}(\bw_k) - \theta_{j,k}(\bw_k)\right)^2 \\
\leq&\sum_{l=k+1}^{K-1} (4\alpha)^{l-k-1}\left(\frac{8\alpha\gamma \Xi_{x_l}}{\tau_{l}}+\frac{2\gamma\Xi_{y_{l-1}}}{\tau_l}\right)\cdot\\
&\int \tau_l \|\nabla_{x_l}f(\bw_k, x_k, y_k, \bz_{k})\|^2\pi_{k}(x_k)\pi_{k+1}(y_{k})\pi_{k}^{\bz}(\bz_k)dx_kdy_{k}d\bz_k\\
&+ \frac{2\gamma \Xi_{x_{k}}}{\tau_{k}}\int \tau_{k} \|\nabla_{x_k}f(\bw_k, x_k, y_k, \bz_{k})\|^2\pi_{k}(x_k)\pi_{k+1}(y_{k})\pi_{k}^{\bz}(\bz_k)dx_kdy_{k}d\bz_k\\
&+\sum_{l=k}^{K-1} (4\alpha)^{l-k} \frac{\gamma\Xi_{e_l}}{\rho} \rho \int \Ex_l(\bw_k, x_k, y_k, \bz_{k})\pi_{k}(x_k)\pi_{k+1}(y_{k})\pi_{k}^{\bz}(\bz_k)dx_kdy_{k}d\bz_k.
\end{split}\]
Putting parts (A), (B), and (C) together, we have
\[\begin{split}
&\E_{0:K}\left[\E_k\left[\left(\E_{(k+1):K} f(\bX_{0:K})- \E_{k:K} f(\bX_{0:K})\right)^2\right]\right]\\
\leq&\sum_{l=k+2}^{K-1} \frac{3(4\alpha)^{l-k-1}}{\tau_l}\left(8\alpha\gamma\Xi_{x_l}+2\gamma\Xi_{y_{l-1}}\right)\int \tau_l \|\nabla_{x_l}f(\bx_{0:K})\|^2\pi_{0:K}(\bx_{0:K})d\bx_{0:K}\\
&+\frac{3}{\tau_{k+1}}\left(8\alpha\gamma\Xi_{x_{k+1}}+2\gamma\Xi_{y_{k}}+q_{k+1}\right)\int\|\nabla_{x_{k+1}}f(\bx_{0:K})\|^2\pi_{0:K}(\bx_{0:K})d\bx_{0:K}\\
&+ \frac{3}{\tau_k}\left(2\gamma \Xi_{x_{k}} + q_k\right)\int \tau_{k} \|\nabla_{x_k}f(\bx_{0:K})\|^2\pi_{0:K}(\bx_{0:K})d\bx_{0:K}\\
&+\sum_{l=k}^{K-1} 3(4\alpha)^{l-k} \frac{\gamma \Xi_{e_l}}{\rho} \rho \int \Ex_l(\bx_{0:K})\pi_{0:K}(\bx_{0:K})d\bx_{0:K}.
\end{split}\]
Then
\[\begin{split}
&\E_{0:K}\left[\left(f(\bX_{0:K}) - \E f(\bX_{0:K})\right)^2\right]\\
\leq& \sum_{k=0}^{K} \E_{0:K}\left[\E_k\left[\left(\E_{(k+1):K} f(\bX_{0:K})- \E_{k:K} f(\bX_{0:K})\right)^2\right]\right]\\
\leq& \kappa \E_{0:K}[\Gamma_R(f(\bX_{0:K}))],
\end{split}\]
where
\[\begin{split}
\kappa=&\max_{0\leq k\leq K-1}\max\left\{\sum_{h=2}^{k-2}\frac{3(4\alpha)^{k-h+1}}{\tau_k}\left(8\alpha\gamma\Xi_{x_k}+2\gamma\Xi_{y_{k-1}}\right)
+\frac{3}{\tau_{k}}\left((8\alpha\gamma+2\gamma)\Xi_{x_{k}}+2\gamma\Xi_{y_{k-1}}+2q_k\right),\right.\\
&\left.\sum_{h=0}^{k}\frac{3(4\alpha)^{k-h+2}}{\rho}\gamma \Xi_{e_k}\right\}.
\end{split}\]
When $I=2$, we can further improve $\kappa$ to
\[\begin{split}
\kappa=&\max_{0\leq k\leq K-1}\max\left\{\sum_{h=2}^{k-2}\frac{3\alpha^{k-h+1}}{\tau_k}\left(2\alpha\gamma\Xi_{x_k}+2\gamma\Xi_{y_{k-1}}\right)
+\frac{3}{\tau_{k}}\left((2\alpha\gamma+2\gamma)\Xi_{x_{k}}+2\gamma\Xi_{y_{k-1}}+2q_k\right),\right.\\
&\left.\sum_{h=0}^{k}\frac{3\alpha^{k-h+2}}{\rho}\gamma \Xi_{e_k}\right\}.
\end{split}\]

Recall that
\[\begin{split}
\Xi_{x_k}& =28q_ka_{k+1},\\
\Xi_{y_k}&=28q_{k+1} 
+ 7\frac{\left(r_{k+1}\right)^{d+1}} {\left(r_k\right)^{d-1}}a_ka_{k+1}\left(\log\left(\frac{r_{k+1}}{r_k}\right)\right)^{1_{d=1}},\\
\Xi_{e_k}&=7\left(\frac{r_{k+1}}{r_k}\right)^da_ka_{k+1}.
\end{split}\]
\section{Conclusion and future directions}
Langevin diffusion (LD) is a popular sampling method, but its convergence rate can be significantly reduced if the target distribution is a mixture of singular densities. Replica exchange Langevin diffusion (ReLD) is a method that can circumvent this issue. It employs an additional LD process sampling a high temperature version of the target distribution, and swap the values of the two processes according to a Metropolis-Hasting type of law. More generally, multiple replica exchange (mReLD) employs $K$ additional LD processes sampling the target distribution with different temperature coefficients. In this work, we formulate a framework to quantify the spectral gap of ReLD and mReLD. Our analysis show that the spectral gap of ReLD does not degenerate when the mixture component becomes singular, as long as the simulation parameters of ReLD scale proportionally to the singularity parameter $\epsilon^d$. While using mReLD can achieve the same convergence rate, the simulation parameters have a weaker dependence on the singularity.

While our results close some theoretical gaps for ReLD and mReLD, there are several questions left unanswered. First, ReLD and mReLD are stochastic processes, but not executable sampling algorithms. How to turn them into efficient MCMC algorithms is an interesting research question. While we roughly estimate the computational complexity of ReLD and mReLD with the appropriate simulation parameters, the exact details of their implementation cost are not fully understood. Second, our spectral gap estimate for mReLD is intuitively not tight for when the component size $I$ and the number of replicas $K$ are large. This is much due to the analysis techniques we used. It will be a very interesting future work if these issues can be fixed.

\section*{Acknowledgement}
The research of JD is supported by National Science Foundation, Grant DMS-1720433.
The research of XTT is supported by the National University of Singapore grant R-146-000-292-114.

\begin{appendix}
\section{Proof of Proposition \ref{prop:eg1}}
\begin{proof}
\textbf{Claim 1)}
Let 
\[
\E_{k}[f(X)]:=\E[f(X)|X_1, \dots, X_{k-1}, X_{k+1}, \dots, X_d]
\]
and
\[
\var_{k}(f(X))=\E_k[(f(X)-\E_{k}f(X))^2].
\]
We also define 
\[
\E_{1:k}[f(X)]=\E[f(X)|X_{k+1}, \dots, X_d].
\]
Then 
\[
f(X)-\E[f(X)] = \sum_{k=1}^{d} \left(\E_{1:(k-1)}[f(X)]-\E_{1:k}[f(X)]\right),
\]
where $\E_{1:(-1)}[f(X)]:= f(X)$. Note that if $j<k$, then 
\[\begin{split}
&\E\left[(\E_{1:(j-1)}f(X)-\E_{1:j}f(X))(\E_{1:(k-1)}f(X)-\E_{1:k}f(X))\right]\\
=&\E\left[\E_{1:j}\left[(\E_{1:(j-1)}f(X)-\E_{1:j}f(X))(\E_{1:(k-1)}f(X)-\E_{1:k}f(X))\right]\right]\\
=&\E\left[(\E_{1:(k-1)}f(X)-\E_{1:k}f(X)) \E_{1:j}\left[\E_{1:(j-1)}f(X)-\E_{1:j}f(X)\right]\right]=0.
\end{split}\]
Therefore,
\begin{equation}\label{eq:Stein}
\begin{split}
\E[(f(X)-\E f(X))^2] &= \sum_{k=1}^{d} \E\left[\left(\E_{1:(k-1)}f(X)-\E_{1:k}f(X)\right)^2\right]\\
& =  \sum_{k=1}^{d} \E\left[\left(\E_{1:(k-1)}\left[(f(X)-\E_{k}[f(X)]\right]\right)^2\right]\\
&\mbox{as $\E_{1:(k-1)}[\E_kf(X)]=\E_{1:k} f(X)$ by independence}\\
&\leq \sum_{k=1}^{d} \E\left[\E_{1:(k-1)}\left[\left(f(X)-\E_{k} f(X)\right)^2\right]\right]\\
& = \sum_{k=1}^{d} \E\left[\var_k(f(X))\right]
\end{split}\end{equation}

Next, for a one-dimensional Gaussian random variable, we have
\[
\int_{x}^{\infty} y\frac{1}{\sqrt{2\pi\epsilon^2}}\exp\left(-\frac{y^2}{2\epsilon^2}\right)dy
= \epsilon^2 \frac{1}{\sqrt{2\pi\epsilon^2}} \exp\left(-\frac{x^2}{2\epsilon^2}\right)
\] 
Then, by Lemma \ref{lm:univar} with $p(x)=\tfrac{1}{\sqrt{2\pi\epsilon^2}} \exp(-\tfrac{x^2}{2\epsilon^2})$ and $Q(x)=|x|$, we have
\[
\var_k(f(X))\leq \epsilon^2\E_k[|\nabla_{x_k} f(X)|^2]
\]
From \eqref{eq:Stein}, we have
\[\begin{split}
\E[(f(X)-\E[f(X)])^2] &\leq \sum_{k=1}^{d} \E\left[\var_k(f(X))\right]\\
&\leq \sum_{k=1}^{d} \E\left[\epsilon^2\E_k[ |\nabla_{x_k} f(X)|^2 ]\right]\\
&= \epsilon^2\E\left[\sum_{k=1}^{d}|\nabla_{x_k} f(X)|^2\right]
=\epsilon^2\E\left[\|\nabla f(X)\|^2\right]
\end{split}\]

\noindent\textbf{Claim 2)}
Let $\nu(x) \propto \phi(x/\epsilon)$. Then, because $\nabla \phi(x/\epsilon)=-\frac1{\epsilon^2} \phi(x/\epsilon) x$, 
\[\begin{split}
&\calE\left(\frac{\nu}{\pi}\right)\\
=& 4\int \left\|\nabla \frac{\phi(x/\epsilon)}{\phi(x/\epsilon)+\phi((x-m)/\epsilon)} \right\|^2\pi(x)dx\\
=& \frac{4}{\epsilon^4}\int \left\| \frac{\phi(x/\epsilon)x(\phi(x/\epsilon)+\phi((x-m)/\epsilon))-\phi(x/\epsilon)(\phi(x/\epsilon)x+\phi((x-m)/\epsilon)(x-m))}{(\phi(x/\epsilon)+\phi((x-m)/\epsilon))^2} \right\|^2\pi(x)dx\\
=& \frac{4\|m\|^2}{\epsilon^4}\int \left|\frac{\phi(x/\epsilon)\phi((x-m)/\epsilon)}{(\phi(x/\epsilon)+\phi((x-m)/\epsilon))^2} \right|^2\pi(x)dx\\
=&\frac{4\|m\|^2}{\epsilon^4}\int \left|\frac{1}{r(x) + \tfrac{1}{r(x)} + 2} \right|^2\pi(x)dx ~~~\mbox{ where $r(x)=\phi((x-m)/\epsilon)/\phi(x/\epsilon)$}\\
=&\frac{4\|m\|^2}{\epsilon^4}\left(\int_{\{\|x\|^2\leq\|m\|^2/16\} \bigcup\{\|x-m\|^2\leq\|m\|^2/16\}}  \left|\frac{1}{r(x) + \tfrac{1}{r(x)} + 2} \right|^2\pi(x)dx\right.\\
&\left.+\int_{\{\|x\|^2>\|m\|^2/16\}\bigcap\{\|x-m\|^2>\|m\|^2/16\}}  \left|\frac{1}{r(x) + \tfrac{1}{r(x)} + 2} \right|^2\pi(x)dx\right).
\end{split}\]
We first note that when $\|x\|^2\leq\|m\|^2/16$, 
\begin{equation} \label{eq:r_bound1}
r(x)=\exp\left(-\frac{\|x-m\|^2-\|x\|^2}{2\epsilon^2}\right) \leq \exp\left(-\frac{\tfrac{9}{16}\|m\|^2 - \tfrac{1}{16}\|m\|^2}{2\epsilon^2}\right) = \exp\left(-\frac{\|m\|^2}{4\epsilon^2}\right).
\end{equation}
Likewise, we can show that when $\|x-m\|^2\leq\|m\|^2/16$, $1/r(x) \leq \exp\bigl(-\tfrac{\|m\|^2}{4\epsilon^2}\bigr)$.
Thus,
\begin{equation} \label{eq:prop1_bound1}
\int_{\{\|x\|^2\leq\|m\|^2/16\} \bigcup\{\|x-m\|^2\leq\|m\|^2/16\}}  \left|\frac{1}{r(x) + \frac{1}{r(x)} + 2} \right|^2\pi(x)dx
\leq \exp\left(-\frac{\|m\|^2}{4\epsilon^2}\right).
\end{equation}
Next, we note that the following always hold:
\[\left|\frac{1}{r(x) + \tfrac{1}{r(x)} + 2} \right|^2\leq \frac{1}{16}.\]
Therefore
\begin{equation} \label{eq:prop1_bound2}
\begin{split}
&\int_{\{\|x\|^2>\|m\|^2/16\}\bigcap\{\|x-m\|^2>\|m\|^2/16\}}  \left|\frac{1}{r(x) + \frac{1}{r(x)} + 2} \right|^2\pi(x)dx\\
\leq & \frac{1}{16}\int_{\{\|x\|^2>\|m\|^2/16\}\bigcap\{\|x-m\|^2>\|m\|^2/16\}} \pi(x)dx\\
 \leq &\frac{1}{16} \left(\int_{\left\{\|z\|^2>\frac{\|m\|^2}{16\epsilon^2}\right\}} \phi(z)dz
 + \int_{\left\{\|z-m\|^2>\frac{\|m\|^2}{16\epsilon^2}\right\}}\phi(z-m)dz\right)\\
 \leq &\frac{1}{16} 2\exp\left(- \frac{d}{2}\left(\frac{\|m\|^2}{16d\epsilon^2} - \frac{1}{2} - \log\left(\frac{\|m\|^2}{8d\epsilon^2}\right)\right)\right) \mbox{ by Cramer's bound}\\
\leq& \frac{1}{8} \exp\left(-\frac{1}{64}\frac{\|m\|^2}{\epsilon^2}\right) \mbox{ for $\epsilon\leq \frac{\|m\|}{16\sqrt{d}}$}.
\end{split}
\end{equation}
The last inequality holds because  when $\epsilon\leq \frac{\|m\|}{16\sqrt{d}}$,
\[
\frac{\|m\|^2}{8d\epsilon^2} \geq 32,\mbox{ and } \log\left(\frac{\|m\|^2}{8d\epsilon^2}\right)< \frac{\|m\|^2}{32d\epsilon^2}.
\]
Putting \eqref{eq:prop1_bound1} and \eqref{eq:prop1_bound2} together, we have
\[
\calE\left(\frac{\nu}{\pi}\right) \leq \frac{4\|m\|^2}{\epsilon^4}\left(\exp\left(-\frac{\|m\|^2}{4\epsilon^2}\right) + \frac{1}{8} \exp\left(-\frac{1}{64}\frac{\|m\|^2}{\epsilon^2}\right)\right)
\leq \frac{5\|m\|^2}{\epsilon^4}\exp\left(-\frac{\|m\|^2}{64\epsilon^2}\right).
\]
On the other hand, 
\[
\begin{split}
\chi^2(\nu\|\pi)&=\int \left(\frac{\nu(x)}{\pi(x)}-1\right)^2\pi(x)dx\\
&\geq\left(\int \left|\frac{\nu(x)}{\pi(x)}-1\right|\pi(x)dx\right)^2 \\
&= \frac{1}{4}\frac{1}{\epsilon^{2d}}\left(\int\left|\phi(x/\epsilon) - \phi((x-m)/\epsilon))\right|dx\right)^2 \\
&=\frac{1}{4}\frac{1}{\epsilon^{2d}}\left(\int\left|1 - r(x)\right|\phi(x/\epsilon)dx\right)^2\\
&\geq \frac{1}{4}\frac{1}{\epsilon^{2d}} \left(\int_{\{\|x\|^2\leq \|m\|^2/16\}}\left|1 - \exp\left(-\frac{\|m\|^2}{4\epsilon^2}\right)\right|\phi(x/\epsilon)dx\right)^2
\mbox{ by \eqref{eq:r_bound1}}\\
&\geq \frac{1}{8} \left(\int_{\left\{\|z\|^2\leq \frac{\|m\|^2}{16\epsilon^2}\right\}}\phi(z)dz\right)^2\mbox{ by replacing $x/\epsilon$ with $z$}\\ 
&=\frac{1}{8} \left(1-\int_{\left\{\|z\|^2>\frac{\|m\|^2}{16\epsilon^2}\right\}}\phi(z)dz\right)^2 \\
&\geq \frac{1}{8} \left(1-\exp\left(-\frac{1}{64}\frac{\|m\|^2}{\epsilon^2}\right)\right)^2 \mbox{ using Cramer bound again like \eqref{eq:prop1_bound2} for $\epsilon\leq \frac{\|m\|}{16\sqrt{d}}$}\\
&\geq \frac{1}{16} 
\end{split}
\]
Above all,
\[
\kappa = \max_{u: u\ll \pi} \frac{\chi^2(u\|\pi)}{\calE (u/\pi)} \geq \frac{\chi^2(\nu\|\pi)}{\calE (\nu/\pi)} 
\geq \frac{\epsilon^4}{80\|m\|^2}\exp\left(\frac{\|m\|^2}{64\epsilon^2}\right).
\]
\end{proof}

\section{Proof of results in Section \ref{sec:main}} \label{app:sec2}
\subsection{Proof of Lemma \ref{lem:GsnLyapunov}}
\begin{proof}
Recall that $H(x)=-\log\nu(x)$. Without loss of generality, we assume $m=0$ and $H(0)=0$.
As
\[
\|\nabla H(x)\|\|x\|\geq \langle\nabla H(x), x\rangle\geq c\|x\|^2,
\]
\[
\|\nabla H(x)\|^2\geq c\langle\nabla H(x), x\rangle.
\]
Then, by convexity of $H$, we have
\[
H(0)\geq H(x)-\langle \nabla H(x),x\rangle,
\]
which implies that
\[
cH(x)\leq c\langle \nabla H(x),x\rangle\leq \|\nabla H(x)\|^2.
\]

For $V(x)=\frac cd\|x\|^2+1$, 
\begin{align*}
\lc_\nu V(x)&=-\frac{2c}d\langle \nabla H(x), x\rangle+2c
\leq -2\frac {c^2}d\|x\|^2+2c
\leq -c  V(x)+3c 1_{\|x\|^2\leq \frac{3d}{c}}.
\end{align*}
In addition, as $\|\nabla^2 H(x)\|\leq L$, for some $x'$ on the line segment between $x$ and $0$,
\[
H(x)= H(0)+\frac12 x^T\nabla^2 H(x')x\leq \frac12 L\|x\|^2.
\]
Thus, if $\|x\|^2\leq \frac{3d}{c}$
\[
\frac{\sup_{x\in B} \nu(x) }{\inf_{x\in B}\nu(x)}\leq 
\exp\left(\frac{3dL}{2c} \right)
\]
We have thus verify the conditions of Definition \ref{defn:lyap}.
Applying Proposition \ref{prop:bakry_lyap}, we get
\[
q=\frac{1+3c \tfrac{3d}{c} \exp(\tfrac{3dL}{c})}{c}. 
\]

Next, note that because
\[
 H(x) - H(0)\geq H(x)-H(x/2) \geq \langle \nabla H(x/2), x/2\rangle \geq \frac{c}{4}\|x\|^2,
\]
we have
\[
\frac{1}{4}c \|x\|^2 \leq H(x) \leq \frac{1}{2}L\|x\|^2,
\]
which implies that
\[
\exp\left(-\frac12 L \|x\|^2\right)\leq \exp(-H(x))\leq \exp\left(-\frac14 c\|x\|^2\right).
\]
Therefore, 
\[
\int \exp(- H(x))dx\leq \int \exp\left(-\frac14 c\|x\|^2\right)\leq \left(\frac{4\pi}{c}\right)^{d/2},
\]
and for $\|x\|^2\leq \frac{3d}{c}$
\[
\exp(-H(x))\geq \exp(-\tfrac{3dL}{2c}).
\]
This leads to our estimate of $a$
\[
a=\frac{\int_{B} \exp(- H(x))dx}{\exp(-H(x))V_d (3d/c)^{\frac{d}{2}}}\leq \frac{1}{V_d}\exp\left(\frac{3Ld}{2c}\right) \left(\frac{4\pi}{3d}\right)^{d/2}.
\]
%
%
%

%
%

\end{proof}

\subsection{Proof of Lemma \ref{lem:gettinga}}
\begin{proof}
Without loss of generality, we assume $x_0=0$. Let $\nu_0=\min_{x\in B(0,R)} \nu(x)$ and $H(x)=-\log \nu(x)$. 

Note that because
\[
\lc_\nu V(x) =-2\gamma\langle \nabla H(x), x \rangle+d\gamma,
\]
$\lc_\nu V(x)\leq -\lambda V(x)$ when $\|x\|\geq R$ indicates that 
\[
\langle \nabla H(x), x\rangle\geq  \frac 12\lambda \|x\|^2
\]
Let $y=\frac{R}{\|x\|}x$, then 
\begin{align*}
H(x)-H(y)&=\int_0^1 \langle \nabla H(y+s(x-y)), x-y\rangle ds\\
&=\int_0^1 \langle \nabla H(y+s(x-y)), y+s(x-y)\rangle \frac{\|x\|-R}{R+s(\|x\|-R)}ds\\
&\geq \frac{1}{2}\lambda \int^1_0\|y+s(x-y)\|\|x-y\|ds\\
& \geq \frac{1}{2}\lambda \int^1_0\langle y+s(x-y),x-y\rangle ds
=\frac14\lambda (\|x\|^2-R^2). 
\end{align*}
Therefore 
\[
\nu(x)\leq \nu(y) \exp\left(-\frac14 \lambda(\|x\|^2-R^2)\right)\leq C\nu_0 \exp\left(-\frac14 \lambda(\|x\|^2-R^2)\right),
\]
as $\|y\|^2=R^2$.
Meanwhile, for $\|x\|\leq R$,
\[
\nu(x)\leq C\nu_0\leq C\nu_0\exp\left(-\frac14 \lambda(\|x\|^2-R^2)\right).
\]
Then, because $\int \nu(x)dx=1$, 
\[
1\leq C\nu_0 \exp\left(\frac14 \lambda R^2\right)\left(\frac{4\pi}{\lambda}\right)^{d/2}. 
\] 
This implies
\[
a\leq \frac{u_{B(0,R)}(x)}{\nu_0} \leq \frac{C}{V_dR^d}\exp\left(\frac14 \lambda R^2\right)\left(\frac{4\pi}{\lambda}\right)^{d/2}.
\]
\end{proof}

\subsection{Proof of Proposition \ref{prop:piyez}}
In this section, we prove a more general version of Proposition \ref{prop:piyez}:
\begin{prop}
\label{prop:piy}
\begin{enumerate}[1)]
\item If $\pi^y(x)\propto \phi(x/M)$, then Assumption \ref{aspt:piy} holds with 
\[\begin{split}
&R^2=3M^2 (2d+1), 
\quad Q=2M^2\left(1 + \frac{9}{2}(2d+1)\exp(12d+8)\right),\\
&A=\frac{1}{V_d}\left(\frac{2\pi}{3(2d+1)}\right)^{d/2}\exp\left(6d+4\right).
\end{split}\]
\item Suppose $\pi(x)=\sum_{i=1}^I p_i\nu_i(x)$, where $\nu_i(x)$ are $(c,L)$-log concave densities with modes $\|m_i\|\leq M$.
If $\pi^y(x)\propto \pi(x)^\beta$ with $\beta=d(2M^2c + 2M^2L^2/c)^{-1}$,
then Assumption \ref{aspt:piy} holds with 
\[\begin{split}
&R^2=20M^2\left(1+\frac{L^2}{c^2}\right), 
\quad Q=M^2\left(1+\frac{L^2}{c^2}\right)\left(\frac{4}{d}+100\exp\left(44\frac{dL}{c}\right)\right),\\
& A= \frac{1}{V_d}\left(\frac{4\pi}{5d}\right)^{d/2}\exp\left(22\frac{dL}{c}+\frac{5}{4}d\right).
\end{split}\]
\item If $\pi^y(x)\propto \phi(x/M)\pi(x)^\beta$, where $\nu_i(x)$ are $(c,L)$-log concave densities with modes $\|m_i\|\leq M$ and $\beta \leq (dM^2 c + dM^2L^2/c)^{-1}$, then Assumption \ref{aspt:piy} holds with
\[\begin{split}
&R^2=5M^2(2d+1),\quad Q=2M^2\left(1+\frac{25}{2}(2d+1)\exp\left(20d+30\right)\right),\\
&A=\frac{1}{V_d}\left(\frac{8\pi}{5(2d+1)}\right)^{d/2}\exp\left(12d+16\right).
\end{split}\]

\end{enumerate} 

\begin{rem}
\label{rem:noreg}
This proposition also considers scenario 2) where $\pi^y$ does not contain a Gaussian regularization. In this case, when implementing LD on $\pi^y$, the drift term of the diffusion process \eqref{eq:LDY} does not include $-Y(t)/M^2$. However, this would require the inverse temperature $\beta$ to be of a specific order, which can be different from the inverse of  $\tau$ used in \eqref{eq:LDY}. In comparison, scenario 3) only requires $\beta$ to be smaller than the threshold value, so we are free to choose any large enough $\tau$ and let $\beta=\frac{1}{\tau}$. The reason we have this freedom is because the additional regularization keeps the probability mass of $\pi^y$ staying inside $B(0,M)$. 
\end{rem}

\end{prop}
\begin{proof}
{\bf For claim 1),} $\pi^y(x)\propto \phi(x/M)$. Consider $V_i(x)=\gamma\|x-m_i\|^2+1$ with $\gamma= \left(M^2(2d+1)\right)^{-1}$. 
We first note that 
\begin{align}
\notag
\lc_{\pi^y} V_i(x)&=-\frac{2\gamma}{M^2}\langle x-m_i, x\rangle
+2d\gamma\\
\notag
&=-\frac{\gamma}{M^2}\| x-m_i\|^2+\frac{\gamma}{M^2}\|m_i\|^2 - \frac{\gamma}{M^2}\|x\|^2+2d\gamma\\
\notag
&\leq -\frac{\gamma}{M^2}\|x-m_i\|^2+(2d+1)\gamma\\
\notag
&\leq -\frac{1}{2M^2}V_i(x)+\left(\frac{1}{2M^2}+(2d+1)\gamma\right) 1_{\|x-m_i\|^2\leq 3M^2(2d+1)}\\
\label{tmp:LW1}
&= -\frac{1}{2M^2}V_i(x)+\frac{3}{2M^2}1_{\|x-m_i\|^2\leq 3M^2(2d+1)}.
\end{align}
Then, the bounding constants for the Lyapunov function are 
\[
\lambda=\frac1{2M^2},\quad h=\frac{3}{2M^2}, \quad R^2=3M^2 (2d+1),
\]
In addition, for $x\in B(m_i,R)$, we have $\|x\|^2\leq 2R^2+2M^2$. Thus, the density ratio can be bounded by
\[
C=\exp\left(\frac{R^2+M^2}{M^2}\right)=\exp(6d+4)
\]
By Proposition \ref{prop:bakry_lyap},
\[
Q=\frac{1+hR^2 C^2}{\lambda}=2M^2 + 9M^2(2d+1)\exp(12d+8). 
\]
Moreover, 
\[
A=\frac{(2\pi M^2)^{d/2}\exp(\tfrac{R^2+M^2}{M^2})}{V_d R^d}
=\frac{1}{V_d}\left(\frac{2\pi}{3(2d+1)}\right)^{d/2}\exp\left(6d+4\right).
\]

\noindent{\bf For claim 2),}  $\pi^y(x)\propto \pi(x)^\beta $. Consider $V_i(x)=\gamma \|x-m_i\|^2+1$ with
\[
\gamma\leq\left(2M^2+2M^2 \frac{L^2}{c^2}+\frac{2d}{\beta c}\right)^{-1}.
\]
Let $H_i(x)=-\log \nu_i(x)$.
We first note that
\[
\nabla \log \pi^y(x)=\beta\nabla \log \pi(x)=-\frac{ \beta\sum_{i=1}^Ip_i \nu_i(x)\nabla H_i(x) }{\sum_{i=1}^I  p_i\nu_i(x)}
\]
and
\begin{align*}
&-\langle \nabla V_j(x), \nabla H_i(x)\rangle\\
=&-2\gamma\langle x-m_j, \nabla H_i(x)-\nabla H_i(m_i)\rangle \\
\leq& -2\gamma\langle x-m_i, \nabla H_i(x)-\nabla H_i(m_i)\rangle+2\gamma\langle m_j-m_i, \nabla H_i(x)-\nabla H_i(m_i)\rangle\\
\leq& -2c\gamma \|x-m_i\|^2+2\gamma L\|m_j-m_i\|\|x-m_i\|\\
\leq& - c\gamma \|x-m_i\|^2+\frac{\gamma L^2}{c}\|m_j-m_i\|^2\\
\leq& - \frac{1}{2}c\gamma \|x-m_j\|^2 + c\gamma\|m_j-m_i\|^2 + \frac{\gamma L^2}{c}\|m_j-m_i\|^2\\
\leq& -\frac12c V_j(x)+2c\gamma M^2
+2\frac{\gamma L^2}{c}M^2+\frac12c.
\end{align*}
Then,
\begin{align*}
\lc_{\pi^y}V_j(x)&\leq 
-\frac1{2 }\beta cV_j(x)+2 \beta c\gamma M^2
+2\beta\frac{\gamma L^2}{c}M^2+\frac1{2}\beta c+2\gamma d\\
&=-\frac14\beta cV_j(x) - \frac{1}{4}\beta c\gamma \|x-m_j\|^2 + \beta c\gamma\left(2 M^2
+2\frac{L^2M^2}{c^2}+\frac{2d}{\beta c}\right)+\frac1{4}\beta c\\
&\leq -\frac14 \beta cV_j(x)+\frac54\beta c 1_{\|x-m_j\|^2\leq \frac{5}{\gamma}}.
\end{align*}
For $\beta=d\left(M^2c + M^2L^2/c\right)^{-1}$,
\[
R^2 = \frac{5}{\gamma} = 20M^2\left(1+\frac{L^2}{c^2}\right).
\]
Next, we note that if
$\|x-m_j\|^2\leq \frac{5}{\gamma}$, 
\[
\|x-m_i\|^2\leq \frac{10}{\gamma}+4M^2\leq \frac{11}{5}R^2.
\]
Then, note that for any region $B$, if we let 
\[
\psi(x):=\max_i \frac{\max_{x\in B} \nu_i(x)}{\min_{x\in B} \nu_i(x)}
\]
\[
\frac{\max_{x\in B} \pi^y(x)}{\min_{x\in B} \pi^y(x)}\leq \frac{(\sum_i p_i\max_{x\in B} \nu_i(x))^\beta}{(\sum_i p_i\min_{x\in B} \nu_i(x))^\beta}\leq \frac{(\sum_i p_i\psi\min_{x\in B} \nu_i(x))^\beta}{(\sum_i p_i\min_{x\in B} \nu_i(x))^\beta}=(\psi(x))^\beta. 
\]
Therefore 
\[\begin{split}
C=\frac{\max_{B(m_j,R)} \pi^y(x)}{\min_{B(m_j,R)} \pi^y(x)}
&\leq \max_i \left(\frac{\max_{B(m_j,R)} \nu_i(x)}{\min_{B(m_j,R)} \nu_i(x)}\right)^{\beta}\\
&\leq  \exp\left(\frac12L \frac{11}{5}R^2\beta\right)=\exp\left(22\frac{dL}{c}\right). 
\end{split}\]
By Proposition \ref{prop:bakry_lyap},
\[
Q=\frac{1+\frac54 \beta c \frac{5}{\gamma} \exp(44dL/c)}{\frac14 \beta c}
=M^2\left(1+\frac{L^2}{c^2}\right)\left(\frac{4}{d}+100\exp\left(44\frac{dL}{c}\right)\right).
\]
The estimate of $A$ can be obtained by Lemma \ref{lem:gettinga}, 
\[
A=\frac{C}{V_d}\exp\left(\frac{1}{16} \beta cR^2\right)\left(\frac{16\pi}{\beta c R^2}\right)^{d/2}
=\frac{1}{V_d}\exp\left(22\frac{dL}{c}+\frac{5}{4}d\right)\left(\frac{4\pi}{5d}\right)^{d/2}.
\]

\noindent{\bf For claim 3),}  $\pi^y(x)\propto \phi(x/M)\pi(x)^\beta$.
Consider $V_i(x)=\gamma\|x-m_i\|^2+1$. 
Combining our analysis in claim 1) and claim 2), we have
\[\begin{split}
\lc_{\pi^y}V_j(x) \leq& -\left(\frac{1}{2M^2}+\frac{\beta c}{4}\right)V_j(x)
-\left(\frac{1}{2M^2}+\frac{\beta c}{4}\right)\gamma\|x-m_j\|^2\\
&+\frac{1}{2M^2} + \beta c\gamma\left(2 M^2
+2\frac{L^2M^2}{c^2}+\frac{2d+1}{\beta c}\right)+\frac1{4}\beta c\\
\leq& -\frac{1}{2M^2}V_j(x)
-\frac{1}{2M^2}\gamma\|x-m_j\|^2\\
&+\frac{1}{2M^2} +  \beta c\gamma\left(2 M^2
+2\frac{L^2M^2}{c^2}\right)+\gamma(2d+1).
\end{split}\]
For $\beta \leq (dM^2 c + dM^2L^2/c)^{-1}$ and  $\gamma = (M^2(2d+1))^{-1}$,
we can set
\[
R^2=\frac{5}{\gamma} = 5M^2(2d+1) \mbox{ and } h=\frac{5}{2M^2}.
\]
Then
\[
\lc_{\pi^y}V_j(x) \leq -\frac{1}{2M^2}V_j(x) + h1_{\|x-m_j\|^2 \leq R^2}.
\]
We next note that as
\[
\frac{\max_{x\in B(m_j,R)} \phi(x/M)}{\min_{x\in B(m_j,R)}\phi(x/M)}
\leq \exp\left(\frac{R^2+M^2}{M^2}\right) = \exp(10d+6)
\]
and
\[\begin{split}
\frac{\max_{x\in B(m_j,R)} \pi(x)^{\beta}}{\min_{x\in B(m_j,R)}\pi(x)^{\beta}}
&\leq \exp\left(\frac{1}{2}L(2R^2+4M^2)\beta\right)\\
&\leq \exp\left(\frac{10+7/d}{1+L^2/c^2}\frac{L}{c}\right)
\leq \exp(17/2),
\end{split}\]
we have
\[
C \leq \frac{\max_{x\in B(m_j,R)} \phi(x/M)}{\min_{x\in B(m_j,R)}\phi(x/M)}\frac{\max_{x\in B(m_j,R)} \pi(x)^{\beta}}{\min_{x\in B(m_j,R)}\pi(x)^{\beta}}
\leq  \exp\left(10d+15\right)
\]
By Proposition \ref{prop:bakry_lyap},
\[
Q = \frac{1+hR^2C^2}{\tfrac{1}{2M^2}}
=2M^2\left(1+\frac{25}{2}(2d+1)\exp\left(20d+30\right)\right)
\]
Lastly, the constant $A$ can be obtain from Lemma \ref{lem:gettinga}:
\[
A=\frac{C}{V_d}\exp\left(\frac{R^2}{8M^2}\right)\left(\frac{8M^2\pi}{R^2}\right)^{d/2}
\leq \frac{1}{V_d}\exp\left(12d+16\right)\left(\frac{8\pi}{5(2d+1)}\right)^{d/2}.
\]
\end{proof}

\end{appendix}

\bibliographystyle{plain}
\bibliography{ReLDbib.bib}

\end{document}